\documentclass[11pt]{amsart}

\usepackage{mathrsfs}
\usepackage{amsfonts}
\usepackage{amssymb}
\usepackage{amsxtra}
\usepackage{dsfont}
\usepackage{color}
\usepackage[compress, sort]{cite}
\usepackage{graphicx}
\usepackage[shortlabels]{enumitem}

\usepackage[english]{babel}
\usepackage[T1]{fontenc}

\usepackage[type1]{newtxtext}
\usepackage{newtxmath}

\usepackage[margin=2.5cm, centering]{geometry}
\usepackage[colorlinks,citecolor=blue,urlcolor=blue,bookmarks=true]{hyperref}
\hypersetup{
pdfpagemode=UseNone,
pdfstartview=FitH,
pdfdisplaydoctitle=true,
pdfborder={0 0 0}, % No link borders
pdftitle={Sturm--Liouville operators with periodically modulated parameters. Part I: Regular case},
pdfauthor={Grzegorz Świderski and Bartosz Trojan},
pdflang=en-US
}

\newcommand{\CC}{\mathbb{C}}

\newcommand{\sS}{\mathbb{S}}
\newcommand{\tsS}{\widetilde{\sS}}
\newcommand{\NN}{\mathbb{N}}
\newcommand{\RR}{\mathbb{R}}

\newcommand{\calC}{\mathcal{C}}

\newcommand{\calF}{\mathcal{F}}

\newcommand{\calD}{\mathcal{D}}

\newcommand{\calT}{\mathcal{T}}

\newcommand{\scrD}{\mathscr{D}}

\newcommand{\Lloc}{L^1_{\mathrm{loc}}}

\newcommand{\frakp}{\mathfrak{p}}
\newcommand{\frakq}{\mathfrak{q}}
\newcommand{\frakw}{\mathfrak{w}}
\newcommand{\frakb}{\mathfrak{b}}

\newcommand{\frakT}{\mathfrak{T}}

\newcommand{\sigmaEss}{\sigma_{\mathrm{ess}}}
\newcommand{\sigmaAC}{\sigma_{\mathrm{ac}}}

\newcommand{\sigmaP}{\sigma_{\mathrm{p}}}
\newcommand{\sigmaSC}{\sigma_{\mathrm{sc}}}
\newcommand{\sigmaS}{\sigma_{\mathrm{s}}}

\newcommand{\Id}{\operatorname{Id}}

\newcommand{\sign}[1]{\operatorname{sign}({#1})}

\newcommand{\intr}{\operatorname{int}}

\newcommand{\pl}[1]{\foreignlanguage{polish}{#1}}

\newcommand{\abs}[1]{\lvert {#1} \rvert}
\newcommand{\sprod}[2]{\langle {#1}, {#2} \rangle}

\newcommand{\tr}{\operatorname{tr}}
\newcommand{\sym}{\operatorname{sym}}

\newcommand{\GL}{\operatorname{GL}}
\newcommand{\SL}{\operatorname{SL}}
\newcommand{\discr}{\operatorname{discr}}
\newcommand{\cl}{\operatorname{cl}}

\newcommand{\vphi}{\varphi}

\newcommand{\ud}{{\rm d}}
\newcommand{\ue}{\textrm{e}}
\newcommand{\supp}{\operatornamewithlimits{supp}}

\newcommand{\Dom}{\operatorname{Dom}}

\newcommand{\AC}{\operatorname{AC}}
\newcommand{\ACloc}{\AC_\mathrm{loc}}

\newcommand{\Wrk}{\operatorname{Wr}}

\newtheorem{claim}{Claim}
\newtheorem*{theorem*}{Theorem}

\newtheorem{theorem}{Theorem}[section]
\newtheorem{proposition}[theorem]{Proposition}
\newtheorem{lemma}[theorem]{Lemma}
\newtheorem{corollary}[theorem]{Corollary}

\newcounter{thm}

\newtheorem{main_theorem}[thm]{Theorem}

\numberwithin{equation}{section}
\theoremstyle{definition}
\newtheorem{example}{Example}
\newtheorem{remark}[theorem]{Remark}

\title{Sturm--Liouville operators with periodically modulated parameters. Part I: Regular case}
\author{Grzegorz Świderski}
\address{
	\pl{
		Grzegorz \'Swiderski \\
		Wydzia\l{} Matematyki,
        Politechnika Wroc\l{}awska\\
        Wyb. Wyspia\'{n}skiego 27\\
        50-370 Wroc\l{}aw\\
        Poland}
}
\email{grzegorz.swiderski@pwr.edu.pl}

\author{Bartosz Trojan}
\address{
		Bartosz Trojan\\
		Wydzia\l{} Matematyki,
        Politechnika Wroc\l{}awska\\
        Wyb. Wyspia\'{n}skiego 27\\
        50-370 Wroc\l{}aw\\
        Poland}
\email{bartosz.trojan@gmail.com}

\subjclass[2020]{Primary 34B24, 34L05; Secondary 34L20, 34C11}
\keywords{Sturm--Liouville operators; spectral analysis; density of states; Tur\'{a}n determinants; Christoffel--Darboux kernels}

\begin{document}
\selectlanguage{english}
\allowdisplaybreaks

\begin{abstract}
	We introduce a new class of Sturm--Liouville operators with periodically modulated parameters. Their spectral properties 
	depend on the monodromy matrix of the underlying periodic problem computed for the spectral parameter equal to $0$. 
	Under certain assumptions, by studying the asymptotic behavior of Christoffel functions and density of states, we prove that 
	the spectral density is a continuous positive everywhere function on the real line.
\end{abstract}

\maketitle

\section{Introduction}
A triple $(p, q, w)$ of functions on $\RR_+$ is called \emph{Sturm--Liouville parameters} if both $p$ and $w$ are positive almost 
everywhere, $q$ is real-valued, and $1/p, q, w \in \Lloc([0, \infty))$. Given Sturm--Liouville parameters one can consider 
the differential expression
\begin{equation}
    \label{eq:int:1}
    \tau u = \frac{1}{w} \big( -(pu')' + q u \big)
\end{equation}
whose domain is
\[
    \Dom(\tau) = \big\{ u \in \ACloc([0, \infty)) : pu' \in \ACloc([0,\infty)) \big\},
\]
whereas $\ACloc([0, \infty))$ denotes the space of locally absolutely continuous functions on $[0, \infty)$. Then the right-hand 
side of \eqref{eq:int:1} exists almost everywhere. Let $\sS^1$ be the unit sphere in $\RR^2$. For any $\eta \in \sS^1$ one can define the operator~$H_\eta$ by setting
\[
	H_\eta u = \tau u.
\]
Its domain is
\begin{align*}
    \Dom(H_\eta) = 
    \big\{ u \in \Dom(\tau) : 
    u, \tau u \in L^2([0, \infty),w), 
    \eta_2 u(0) - \eta_1 pu'(0) = 0 
    \big\}.
\end{align*}
The operator $H_\eta$ is called \emph{Sturm--Liouville operator}. The aim of this paper is to study spectral properties of 
Sturm--Liouville operators under certain hypotheses on their parameters.

The topic of Sturm--Liouville operators is a classical one (see, e.g. \cite{EverittHistory2005} for historical development in the early
20th century). It has been thoroughly studied in numerous textbooks (see, e.g. \cite{Marchenko2011, Zettl2005, Bennewitz2020, 
GesztesyBook2024, Weidmann1987}), which contain an extensive number of further references. Let us recall that taking 
$p \equiv 1$ and $w \equiv 1$, the operator $H_\eta$ is \emph{1-dimensional Schr\"{o}dinger operator} (on the half-line), which spectral
properties are important in quantum mechanics (see, e.g. \cite{Flugge1999, Teschl2014}). We refer to \cite{Everitt2005} for a list of 
important Sturm--Liouville parameters and some properties of the corresponding operators.

Spectral properties of $H_\eta$ are intimately related to analysis of solutions of eigenvalue equation $\tau u = z u$ for 
$z \in \CC$. In fact, in view of \emph{Weyl's alternative}, exactly one of the cases holds: 
\begin{itemize}
    \item For every $z \in \CC$ each solution to $\tau u = z u$ belongs to $L^2([0,\infty),w)$ (\emph{limit circle case}).
	\item There are $z \in \CC$ and $u \notin L^2([0,\infty),w)$ such that $\tau u = z u$ 
	(\emph{limit point case}).
\end{itemize}
Let us recall that the limit point case implies that, actually, for each $z \in \CC \setminus \RR$ the subset of 
$L^2([0,\infty),w)$ consisting of solutions to $\tau u = z u$ is one-dimensional. Therefore, the Weyl's theorem asserts that 
$H_\eta$ is self-adjoint if and only if $\tau$ is in the limit point case. 
If this is the case, then there exists a positive Borel measure on $\RR$, $\mu_\eta$, satisfying $\int_\RR \frac{1}{1+\lambda^2} \ud \mu_\eta(\lambda) < \infty$ and a unitary operator $\calF_\eta : L^2([0,\infty),w) \to L^2(\mu_\eta)$ such that $\calF_\eta H_\eta \calF_\eta^{-1} = M$ where
\[
    (Mf)(\lambda) = \lambda f(\lambda), \quad f \in \Dom(M)
\]
whereas $\Dom(M) = \{ f \in L^2(\mu_\eta) : \lambda \mapsto \lambda f(\lambda) \in L^2(\mu_\eta) \}$ (see, e.g. \cite[Theorem 4.3.7]{Bennewitz2020}). Therefore, all 
spectral information on $H_\eta$ is carried by the measure $\mu_\eta$. Recall that the measure $\mu_\eta$ comes from the Herglotz representation of Weyl--Titchmarsh~$m$ function, which encodes the initial conditions of $L^2([0,\infty),w)$ solutions of equation $\tau u = z u$ for $\Im z > 0$. Since the measure $\mu_\eta$ is canonical, it is of interest to study its properties in detail.

A well-studied example of Sturm–Liouville operators is the class with periodic coefficients (see, e.g. \cite{Brown2013}). 
For definiteness, let $(\frakp,\frakq,\frakw)$ be some Sturm--Liouville parameters which are $\omega$-periodic with some period~$\omega>0$. 
In particular, $\tau$ is in the limit point case. Then the so-called \emph{monodromy matrix}, $\frakT(\omega;z)$, (see Appendix~\ref{sec:11} for the precise definition) allows one to describe spectral properties of $H_\eta$. More precisely, let\footnote{For any $2 \times 2$ matrix $X$ we define its \emph{discriminant} by $\discr X = (\tr X)^2 - 4 \det X$.}
\[
    \Lambda_- = \{ z \in \RR : \discr \frakT(\omega;z) < 0 \}.
\]
It turns out that
\[
    \sigmaAC(H_\eta) = \sigmaEss(H_\eta) = \cl(\Lambda_-),
    \quad \text{and} \quad
    \sigmaSC(H_\eta) = \emptyset,
    \quad \text{and} \quad
    \sigmaP(H_\eta) \cap \Lambda_- = \emptyset.
\]
Moreover,
\begin{equation}
    \label{eq:int:2}
    \Lambda_- = \bigcup_{k=1}^\infty I_k
\end{equation}
where $I_k$ are pairwise disjoint non-empty open bounded intervals whose closures might touch one another.

In this article, we are particularly interested in a seemingly new class of Sturm--Liouville parameters. To be more precise, 
we say that $(p, q, w)$ are \emph{$\omega$-periodically modulated Sturm--Liouville parameters} if $(p, q, w)$ are
Sturm--Liouville parameters satisfying: There are $\omega$-periodic Sturm--Liouville parameters $(\frakp, \frakq, \frakw)$,
such that $p,\frakp \in \ACloc([0,\infty))$ are positive everywhere and 
\begin{equation}
    \label{eq:int:16}
    \lim_{n \to \infty} 
    \int_{0}^{\omega} \frac{w_n(t)}{p_n(t)} \ud t = 0, \quad 
    \lim_{n \to \infty} 
    \int_{0}^{\omega} \bigg|\frac{q_n(t)}{p_n(t)} - \frac{\frakq(t)}{\frakp(t)} \bigg| \ud t = 0,  \quad
    \lim_{n \to \infty} 
    \int_{0}^{\omega} \bigg|\frac{p_n'(t)}{p_n(t)} - \frac{\frakp'(t)}{\frakp(t)} \bigg| \ud t = 0,
\end{equation}
where
\[
    p_n(t) = p(t+n\omega), \quad
    q_n(t) = q(t+n\omega), \quad
    w_n(t) = w(t+n\omega), \quad n \geq 0, t \in \RR_+.
\]
We additionally assume that
\begin{equation}
    \lim_{x \to \infty} p(x) = +\infty.
\end{equation}
This class is inspired by a class of unbounded Jacobi matrices introduced in \cite{JanasNaboko2002}.

Let us recall that the Liouville transformation (see, e.g. \cite[Section 7]{Everitt2005}) applied to Sturm--Liouville parameters 
$(p,q,w)$, under certain regularity conditions, allows one to construct a unitary equivalent Schr\"{o}dinger operator $H$
on $[0,\infty)$ with a real-valued potential $V$. Our interest in $\omega$-periodically modulated Sturm--Liouville parameters
is partially motivated by the fact that these parameters sometimes lead to potentials being rapidly oscillating with unbounded 
amplitudes. In particular, Example~\ref{ex:2} describes potentials of a form
\begin{equation}
    \label{eq:124}
    V(x) = (1+x)^a \frakq \big( (1+x)^{(2+a)/2} -1 \big), \quad x \geq 0,
\end{equation}
for any $a>0$ and
\begin{equation}
    \label{eq:125}
    V(x) = \ue^{2x} \frakq(\ue^{x} - 1), \quad x \geq 0.
\end{equation}
The function $\frakq \in \Lloc([0,\infty))$ which appears in \eqref{eq:124} and \eqref{eq:125} is assumed to be real-valued 
and $\omega$-periodic. Recently in \cite[Example 1.5]{Lukic2025}, the authors studied potentials of the form
\begin{equation}
    \label{eq:126}
    V(x) = x^a \sin(x^b), \quad x \geq 0.
\end{equation}
They proved that if $c := b-a-1$ is positive, then $\sigmaEss(H) = [0,\infty)$. Observe that for $\frakq(t) = \sin(t)$
the potential \eqref{eq:124}, corresponds to $c = -a/2 < 0$. In view of Example \ref{ex:4}, there is 
$a_{\mathrm{crit}} \approx 0.968$ such that for all $a \in (a_\mathrm{crit},2]$ we have $\sigmaAC(H) = \RR$. Moreover,
if $a \in (0,\infty) \setminus [a_\mathrm{crit}, 2]$, then all self-adjoint extensions of $H$ have empty essential spectrum.
In other words, we observe spectral phase transition (cf. \cite{JanasNaboko2002} for a similar phenomenon in the discrete setting). 
Consequently, the spectral behavior of $H$ for $a_{\mathrm{crit}}$ may be of particular interest. We note that potentials of the form 
\eqref{eq:126} were also investigated in \cite{White1983}, although the conditions considered there guarantee that $H$ is semi-bounded.

In this paper, we show that spectral properties of $H_\eta$ for $\omega$-periodically modulated Sturm--Liouville parameters 
depend on the value of the monodromy matrix $\frakT(\omega;z)$ for $z=0$. In fact we have the following cases:
\begin{enumerate}[label=\rm (\Roman*), start=1, ref=\Roman*]
	\item \label{eq:PI} $|\tr \frakT(\omega;0)| < 2$;
	\item If $|\tr \frakT(\omega;0)| = 2$, then we have two subcases:
	\begin{enumerate}[label=\rm (\alph*), start=1, ref=II(\alph*)]
		\item 
		\label{eq:PIIa} $\frakT(\omega; 0)$ is diagonalizable;
		\item
		\label{eq:PIIb} $\frakT(\omega; 0)$ is \emph{not} diagonalizable;
	\end{enumerate}
	\item \label{eq:PIII} $|\frakT(\omega;0)|>2$.
\end{enumerate}
It turns out that we can describe them in terms of the intervals $I_k$ (cf. \eqref{eq:int:2}). More precisely, we are in 
Case~\ref{eq:PI} if $0 \in \Lambda_-$; in Case~\ref{eq:PIIa} if $0$ lies at the boundary of exactly two intervals; 
in Case~\ref{eq:PIIb} if $0$ lies at the boundary of exactly one interval, and in Case~\ref{eq:PIII} if 
$0 \notin \cl(\Lambda_-)$. In this article we shall consider regular cases: Case \ref{eq:PI} and Case \ref{eq:PIII}. The critical 
cases: Case \ref{eq:PIIa} and Case \ref{eq:PIIb}, are more involved and will be studied in the sequel.

A standard way of spectral analysis of Sturm--Liouville operators is the \emph{method of subordinacy} initially introduced 
by Gilbert and Pearson in \cite{Gilbert1987} for one dimensional Schr\"{o}dinger operators. Later, this method has been extended to the
Sturm--Liouville setting by Clark and Hinton in \cite{Clark1993}. Let us summarize the method of subordinacy: Suppose that $\tau$ 
is in the limit-point case. Let us denote by $K_L$ the \emph{Christoffel--Darboux kernels} that is (see, e.g. \cite{Maltsev2010})
\[
    K_L(z_1,z_2;\eta) = 
    \int_0^L 
    \mathsf{s}_\eta(t;z_1) \overline{\mathsf{s}_\eta(t;z_2)} w(t) {\: \rm d} t, \quad L > 0, z_1,z_2 \in \CC,
\]
where $\mathsf{s}_{\eta}(\cdot;z)$ is the solution of the equation $\tau u = z u$ satisfying $(u(0),pu'(0))^t = \eta$.
By \cite[Theorem 3.1]{Clark1993}, if for some compact interval with non-empty interior $K \subset \RR$, 
\begin{equation}
    \label{eq:int:7}
    \limsup_{L \to \infty} \sup_{\lambda \in K} \sup_{\eta,\eta' \in \sS^1} 
	\frac{K_L(\lambda,\lambda;\eta)}{K_L(\lambda,\lambda;\eta')}< \infty,
\end{equation}
then $\mu_\eta$ is absolutely continuous on $\intr(K)$, and there are two positive constants $c_1$ and $c_2$, such that the 
density of $\mu_\eta$ satisfies
\begin{equation}
    \label{eq:int:8}
    c_1 < \mu_\eta'(\lambda) < c_2
\end{equation}
for almost all $\lambda \in K$. In particular, for any $\eta \in \sS^1$, the operator $H_\eta$ is absolutely continuous 
on $\intr(K)$ and $K \subset \sigmaAC(H_\eta)$. 

In our study we impose on Sturm--Liouville parameters the following regularity conditions: We say that a function $f :[0, \infty) \to \RR$ 
belongs to $\calD_1^\omega(L^1;\RR)$, if
\[
    \sup_{n \geq 0} \int_{0}^\omega |f(n\omega+s)| \ud s < \infty,
\]
and
\[
    \sum_{n=0}^\infty \int_{0}^\omega \big| f \big( (n+1)\omega + s \big) - f(n\omega + s) \big| \ud s < \infty.
\]
These class has been introduced by Stolz in \cite{Stolz1991a}. We assume that
\begin{equation}
    \label{eq:D1-reg}
    \frac{q}{p}, \frac{w}{p}, \frac{p'}{p} \in \calD_1^\omega(L^1;\RR).
\end{equation}
Additionally, by analogy with Jacobi matrices, the \emph{Carleman condition} is defined as
\begin{equation}
    \label{eq:26}
    \int_0^\infty \frac{w(t)}{p(t)} \ud t = \infty.
\end{equation}
We set
\begin{equation}
    \label{eq:int:17}
    \varrho_n(t) = \sum_{j=0}^n \frac{w_j(t)}{p_j(t)}, \quad\text{and}\quad    
    \rho_L = \int_0^L \frac{w(t')}{p(t')} \ud t'.
\end{equation}
We are now in a position to state our main result, which addresses Case~\ref{eq:PI}.
\begin{main_theorem}
    \label{thm:A}
    Let $\omega > 0$. Suppose that $(p, q, w)$ are $\omega$-periodically modulated Sturm--Liouville parameters, such that 
	the monodromy matrix $\frakT$ corresponding to $(\frakp, \frakq, \frakw)$ satisfies $\abs{\tr \frakT(\omega; 0)} < 2$. 
    Assume that \eqref{eq:D1-reg} holds true. Then $\tau$ is in the limit circle case if \eqref{eq:26} is violated. 
    Suppose that
    \[
        \lim_{n \to \infty} \varrho_n(t) = \infty 
        \quad \text{and} \quad
        \lim_{n \to \infty} \frac{w_{n+1}(t)}{w_n(t)} = 1 \quad \text{for almost all } t \in [0,\omega].
    \]
	If there is $s \in [0,\omega]$ such that
    \begin{equation}
        \label{eq:int:6}
		\lim_{n \to \infty} 
        \int_0^\omega 
        \bigg| 
        \frac{1}{\gamma_n(s)} \frac{w_n(s+t)}{p_n(s+t)} 
	   - \frac{1}{\gamma} \frac{\frakw(s+t)}{\frakp(s+t)} 
       \bigg| \ud t
       =
       0,
    \end{equation}
    where
    \[
        \gamma_n(t) = \int_0^\omega \frac{w_n(t+t')}{p_n(t+t')} \ud t', \quad\text{and}\quad
        \gamma = \int_0^\omega \frac{\frakw(t')}{\frakp(t')} \ud t',
    \]
    and \eqref{eq:26} is satisfied, then there is a positive continuous function $g$ such that
    \begin{equation} 
        \label{eq:int:4}
        \lim_{L \to \infty} 
        \frac{1}{\rho_L} K_{L}(\lambda,\lambda;\eta) = g(\lambda;\eta)
    \end{equation}
	locally uniformly with respect to $(\lambda,\eta) \in \RR \times \sS^1$. Moreover, the measure $\mu_\eta$ is 
	absolutely continuous on $\RR$, and its density $\mu_\eta'$ can be expressed as
    \begin{equation}
        \label{eq:int:5}
        \mu_\eta'(\lambda) = 
        \frac{1}{\pi} 
        \frac{|\partial_z \tr \frakT(\omega;0)|}{\gamma \sqrt{-\discr \frakT(\omega;0)}} 
        \frac{1}{g(\lambda;\eta)}, \quad \lambda \in \RR.
    \end{equation}
\end{main_theorem}
Let us observe that the formula~\eqref{eq:int:5} implies that the density $\mu'_\eta$ is a continuous and positive function on 
the real line. To our knowledge, this conclusion has been established only in a few instances (cf. \cite{Hinton1986, Behncke1991, Behncke1991a} and 
\cite[Theorem~1.10]{Lukic2024}). We next notice that \eqref{eq:int:4} implies \eqref{eq:int:7}, however the statement \eqref{eq:int:5} 
is stronger than \eqref{eq:int:8}. Let us recall that continuity and positivity of $\mu'_\eta$ together 
with \eqref{eq:int:4} yields important information on the sine-kernel universality of the scaling limits of Christoffel--Darboux kernels 
(see \cite{Eichinger2021,Eichinger2024}). 

The proof of Theorem~\ref{thm:A} consists of several steps. The main one is to determine the asymptotic behavior of
$\mathsf{s}_\eta(\cdot;z)$ for $z \in \RR$. To do so, we let $u(\cdot;z)$ to be any solution of $\tau u = z u$. We write
\begin{equation}
    \label{eq:int:11}
    \mathbf{u}(t;z) = T(t;z) \mathbf{u}(0;z) 
    \quad \text{where} \quad
    \mathbf{u}(t;z) =
    \begin{pmatrix}
        u(t;z) \\
        \partial_t u(t;z)
    \end{pmatrix}, \quad t \in [0,\infty), z \in \CC,
\end{equation}
and whereas the function $T: [0,\infty) \times \CC \to \GL(2, \CC)$, called \emph{transfer matrix},
is the unique solution to 
\[
    T(t;z) = \Id + \int_0^t b(t';z) T(t';z) \ud t',
    \quad \text{where} \quad
    b(t;z) = 
    \begin{pmatrix}
        0 & 1 \\
        \frac{q(t)-zw(t)}{p(t)} & -\frac{p'(t)}{p(t)}
    \end{pmatrix}.
\]
Let us define
\begin{equation}
    \label{eq:178}
    \tsS^1 = 
    \big\{ \eta \in \RR^2 : |\eta_1|^2 + |p(0) \eta_2|^2 = 1 \big\}.
\end{equation}
For $\eta \in \tsS$, we let $\mathbf{u}$ be the solution to \eqref{eq:int:11} satisfying $\mathbf{u}(0;z) = \eta$. Then we set
\begin{equation}
    \label{eq:int:10}
     \mathbf{u}_n(t,\eta;z) = \mathbf{u}(t+n\omega;z), \quad t \in [0,\omega], n \in \NN_0.
\end{equation}
Setting 
\[
    X_n(t;z) = T(t+(n+1)\omega;z) T(t+n\omega;z)^{-1}, \quad n \geq 0,
\]
we easily see that
\begin{equation}
    \label{eq:int:9}
    \mathbf{u}_{n+1}(t,\eta;z) = X_{n}(t;z) \mathbf{u}_n(t,\eta;z), \quad n \geq 0.
\end{equation}
In the following theorem we compute the asymptotic behavior of $u_n(t,\eta;z)$ which is the first 
component of $\mathbf{u}_n(t,\eta;z)$. For its proof see Theorem \ref{thm:2}.
\begin{main_theorem}
    \label{thm:B}
    Let $\omega > 0$. Suppose that $(p, q, w)$ are $\omega$-periodically modulated Sturm--Liouville parameters such that
    the monodromy matrix $\frakT$ corresponding to $(\frakp, \frakq, \frakw)$ satisfies $\abs{\tr \frakT(\omega; 0)} < 2$. 
    Assume that \eqref{eq:D1-reg} holds true. Then for each compact set $K \subset \RR$, there are $M \geq 1$ and a non-vanishing
	continuous function $\varphi : [0,\omega] \times \tsS^1 \times K \to \CC$ such that
    \begin{equation}
		\label{eq:int:18}
		\frac{u_n(t,\eta;z)}{\prod_{k=M}^{n-1} \lambda_k^+(t;z)}
		=
		\frac{|\varphi(t,\eta;z)|}{\sqrt{4-|\tr \frakT(\omega;0)|^2}}
		\sin \bigg( \sum_{k=M}^{n-1} \theta_k(t;z) + \arg \varphi(t, \eta; z) \bigg) + E_n(t, \eta; z)
	\end{equation}
    where $\lambda^+_k(t;z)$ is the eigenvalue of $X_k(t;z)$ with the positive imaginary part,
    \[
		\theta_k(t;z) = \arccos \bigg( \frac{\tr X_k(t;z)}{2 \sqrt{\det X_k(t;z)}} \bigg),
    \]
    and
    \[
        \lim_{n \to \infty} \sup_{(t, \eta, z) \in [0,\omega] \times \tsS^1 \times K} |E_n(t, \eta; z)| = 0.
    \]
\end{main_theorem}
Now, to get the asymptotic behavior of $\mathsf{s}_\eta$ it is enough to use Theorem~\ref{thm:B}.
Moreover, by applying an averaging procedure, we obtained a discrete variant of \eqref{eq:int:4}, see Theorem~\ref{thm:3} 
for details. Next, thanks to the uniformity of the asymptotics with respect to $t$ and $\eta$, under the condition \eqref{eq:int:6} we
to get \eqref{eq:int:4}, see Corollary~\ref{cor:4} for details. However, to obtain a genuine asymptotic behavior,  one needs 
to show that $\varphi$ is non-vanishing. To do so we introduced a new object which, by analogy with Jacobi matrices, will be called 
\emph{Tur\'{a}n determinant}. It is defined by the formula
\[
    \scrD_n(t,\eta;z) =
    \det
    \begin{pmatrix}
        u_{n+1}(t, \eta; z) & u_n(t, \eta; z) \\
        \partial_t u_{n+1}(t, \eta; z) & \partial_t u_n(t, \eta; z)
    \end{pmatrix}, \quad n \geq 0, t \in [0,\omega],z \in \CC.
\]
The classical Tur\'{a}n determinants play a key r\^ole in approximating the density of the scalar spectral measure for various 
classes of Jacobi matrices (see, e.g., \cite{Nevai1983, GeronimoVanAssche1991, Nevai1992, PeriodicII, PeriodicIII, 
SwiderskiTrojan2019, jordan}). 
\begin{example} 
	\label{ex:1}
	Let $p(t) \equiv 1, q(t) \equiv 0, w(t) \equiv 1$. By direct computations, it can be shown that for 
	$\eta \in \{(1,0)^t, (0,1)^t \}$ and any $\omega > 0$,
	\begin{equation}
		\label{eq:int:12}
		\mathds{1}_{(0,\infty)}(\lambda)
		\det
		\begin{pmatrix}
		\mathsf{s}_\eta(t+\omega;\lambda) & \mathsf{s}_\eta(t;\lambda) \\
		\partial_t \mathsf{s}_\eta(t+\omega;\lambda) & \partial_t \mathsf{s}_\eta(t;\lambda)
		\end{pmatrix}
		\ud \mu_\eta(\lambda) 
		= 
		\mathds{1}_{(0,\infty)}(\lambda) \frac{1}{\pi} \sin(\sqrt{\lambda} \omega) \ud \lambda, \quad t > 0.
	\end{equation}
	Notice that the right-hand side of \eqref{eq:int:12} is independent of both $t$ and $\eta$.
\end{example}
Example~\ref{ex:1} suggests a conjecture analogous to the case of Jacobi matrices, that the function $h$ in 
\eqref{eq:int:13} is related to the measure $\mu_\eta$. As the proof of the following theorem is less involved than the proof
of Theorem~\ref{thm:A}, this conjecture would show the continuity and positivity of the measure $\mu_\eta$ in a more direct way.
\begin{main_theorem}
    \label{thm:C}
	Suppose that all hypothesis of Theorem~\ref{thm:B} are satisfied. Then there is a positive 
	continuous function $h$ such that
    \begin{equation}
        \label{eq:int:13}
        \lim_{n \to \infty} p_n(t) |\scrD_n(t,\eta;z)| = h(t,\eta;z)
    \end{equation}
    locally uniformly with respect to $(t,\eta,z) \in [0,\omega] \times \tsS^1 \times \RR$.
\end{main_theorem}
For the proof of Theorem~\ref{thm:C} see Theorem~\ref{thm:1}. An important consequence of Theorem~\ref{thm:C} is 
Corollary~\ref{cor:1}, which is instrumental in showing that the function $\varphi$ is non-vanishing. It also implies, 
see Corollary~\ref{cor:7}, that $\tau$ is in the limit circle case if the Carleman condition~\eqref{eq:26} is violated.

Now, having established the asymptotic behavior of $\mathsf{s}_\eta$, we can show \eqref{eq:int:4}. In order to express the value of $g$ 
in terms of \eqref{eq:int:5} we study a well-known object called \emph{density of states}. More precisely, for any $L>0$ one can restrict 
the operator $H_\eta$ to $L^2([0,L], w)$ by imposing the Dirichlet boundary condition at $L$, see \eqref{eq:129} for details. This restriction, 
$H_\eta^L$, is a self-adjoint operator with the pure point spectrum. We set
\[
    \nu_\eta^L = \sum_{\lambda \in \sigma(H_\eta^L)} \delta_\lambda
\]
where $\delta_\lambda$ denotes the Dirac measure at $\lambda$. Then $\nu_\eta^L$ is an infinite Radon measure. We show the following theorem.
\begin{main_theorem}
    \label{thm:D}
    Suppose that all hypothesis of Theorem~\ref{thm:B} are satisfied. Assume that \eqref{eq:26} and \eqref{eq:int:6} hold true. If 
    $f \in \calC_b(\RR)$ is such that
    \footnote{By $\calC_b(\RR)$ (resp. $\calC_c(\RR)$) we denote the Banach space of bounded (resp. compactly supported) continuous functions 
    on the real line with the supremum norm.}
    \begin{equation}
        \label{eq:int:14}
        \sup_{\lambda \in \RR} (1+\lambda^2) |f(\lambda)| < \infty,
    \end{equation}
    then for each $t \in [0,\omega]$,
    \begin{equation}
        \label{eq:int:15}
        \lim_{n \to \infty} \frac{1}{\rho_{t+n\omega}} \int_\RR f(\lambda) \nu_\eta^{t+n\omega}(\ud \lambda) =
        \int_\RR f(\lambda) \nu_\infty( \ud \lambda)
    \end{equation}
    where the measure $\nu_\infty$ is absolutely continuous with the density
    \[
        \frac{\ud \nu_\infty}{\ud \lambda} \equiv
        \frac{1}{\pi}
        \frac{|\partial_z \tr \frakT(\omega;0)|}{\gamma \sqrt{-\discr \frakT(\omega;0)}}.
    \]
\end{main_theorem}
The convergence of the form \eqref{eq:int:15} is usually established for $\calC_c(\RR)$, 
and $\rho_L = L$ (see, e.g. \cite[Section 2.4]{Brown2013}, \cite[Chapter 2.8.3]{Berezin1991}, \cite{Eichinger2025}). However, in view of 
\eqref{eq:int:17} and \eqref{eq:int:16} we have $\lim_{L \to \infty} \rho_L/L = 0$. Next, let us recall that Cauchy transform of a positive 
measure $\nu$ on the real line is defined by
\[
    \calC[\nu](z) = \int_\RR \frac{1}{\lambda-z} \nu(\ud \lambda), \quad z \in \CC \setminus \RR
\]
provided the integral exists. It can be shown that $\calC[\nu_\eta^{L}]$ exists and is equal to $-\partial_z \log \mathsf{s}_\eta(L;z)$, 
see \eqref{eq:74}. In view of our recent result \cite[Lemma 4.1]{SwiderskiTrojan2023}, to prove \eqref{eq:int:15} it suffices to understand 
the asymptotic behavior of $\mathsf{s}_\eta(t + n\omega; z)$ for any $z \in \mathbb{C} \setminus \mathbb{R}$ as $n \to \infty$. To this end, 
we derive in Theorem~\ref{thm:4} a basis of solutions to \eqref{eq:int:11} with prescribed asymptotic behavior. Then, formula~\eqref{eq:int:5} 
follows easily from \eqref{eq:int:4} and \eqref{eq:int:15}; see the proof of Corollary~\ref{cor:6} for details. This completes the proof of 
Theorem \ref{thm:A}.

Let us now turn to Case \ref{eq:PIII}.
\begin{main_theorem}
    \label{thm:E}
    Let $\omega > 0$. Suppose that $(p, q, w)$ are $\omega$-periodically modulated Sturm--Liouville parameters, such that the monodromy
    matrix $\frakT$ corresponding to $(\frakp, \frakq, \frakw)$ satisfies $\abs{\tr \frakT(\omega; 0)} > 2$. Assume that \eqref{eq:D1-reg}
    holds true. If $\tau$ is in the limit-point case, then for any $\eta \in \sS^1$ the operator satisfies 
    $\sigmaEss(H_\eta) = \emptyset$.
\end{main_theorem}
The proof of Theorem~\ref{thm:E} is based on a general consequence of subordinacy method, see Theorem~\ref{thm:11}. Namely, in the proof 
of Theorem~\ref{thm:10} for any compact $K \subset \RR$ we construct a family $\{ u(\cdot,\lambda) : \lambda \in K \}$ of non-zero solutions 
of $\tau u (\cdot;\lambda) = \lambda u(\cdot;\lambda)$ which decay exponentially fast and which are continuous in the $L^2_\mathrm{loc}([0,\infty),w)$ sense.

The main concern of this paper is the investigation of periodically modulated Sturm--Liouville parameters. However, in view of recent interest of asymptotically periodic Sturm--Liouville parameters, see \cite{Behrndt2023}, we show in Appendix~\ref{sec:11} how our techniques can be adapted to this case. The resulting proofs are sometimes easier in this setup, and at the same time, we obtained stronger and more general results than \cite[Theorem 1.1]{Behrndt2023}, see, e.g. Corollary~\ref{cor:9} and Theorem~\ref{thm:13}.

The paper is organized as follows: In Section \ref{sec:7} we introduce Sturm--Liouville, describe the class of $\omega$-periodically modulated
parameters and prove basic properties of the corresponding solution. The Stolz class is defined in Section \ref{sec:2}. In the following section
we describe how to diagonalize $X_n$. In Section \ref{sec:4} we study the asymptotic behavior of Tur\'an determinants, which in 
Section \ref{sec:5} is used to derive the asymptotic behavior of generalized eigenvectors. The asymptotic behavior of Christoffel--Darboux
kernels is studied in Section \ref{sec:6}. The density of states is studied in Section \ref{sec:1}. In Section \ref{sec:9} we consider
Case \ref{eq:PIII} showing that it leads to the empty essential spectrum. Section \ref{sec:10} provides a list of examples. Finally, in Appendix
\ref{sec:11} we describe how to apply our techniques to study asymptotically periodic Sturm--Liouville parameters.

\subsection*{Acknowledgments}
The first author thanks Mateusz Piorkowski for some useful comments.
The authors would like to thank Mateusz Kwaśnicki for sharing his knowledge about regulary varing functions, from which Example~\ref{ex:5} originated.

\section{Sturm--Liouville equations}
\label{sec:7}
In order to understand spectral properties of $H_\eta$ we need to study the eigenvalue problem
\begin{equation}
	\label{eq:2}
	\tau u = z u
\end{equation}
where $\tau$ is defined in \eqref{eq:int:1} and the equality is almost everywhere. In this paper we assume that Sturm–Liouville 
parameters $(p, q, w)$ are such that $p \in \ACloc([0,\infty))$ is positive everywhere.
In view of \cite[Theorem D.2]{Bennewitz2020}, the solution to \eqref{eq:2} has absolutely continuous derivative. Therefore, setting
\[
	\mathbf{u}(t) = 
	\begin{pmatrix}
		u(t) \\
		\partial_t u(t)
	\end{pmatrix}
\]
for given $z \in \CC$ and $t_0 \in \RR_+$, the equation \eqref{eq:2} can be written in the following form
\begin{equation}
	\label{eq:10}
	 \mathbf{u}(t) = \mathbf{u}(t_0) + \int_{t_0}^t b(t'; z) \mathbf{u}(t') {\: \rm d} t', \qquad t \in \RR_+,
\end{equation}
where
\begin{equation}
	\label{eq:5}
	b(t; z) = 
	\begin{pmatrix}
		0 & 1 \\
		\frac{q(t) - z w(t)}{p(t)} & -\frac{p'(t)}{p(t)}
	\end{pmatrix}, \quad t \in \RR_+, z \in \CC.
\end{equation}
By \cite[Theorem D.2]{Bennewitz2020}, there is the unique solution to \eqref{eq:10}, and 
$\mathbf{u}(\cdot; z) \in \ACloc\big(\RR_+; \CC^2\big)$. Moreover, by \cite[Theorem D.5]{Bennewitz2020},
the mapping $\CC \ni z \mapsto \mathbf{u}(\cdot; z)$ is entire with values in $\ACloc\big(\RR_+; \CC^2\big)$.
Let us observe that $u(t; z) = \sprod{\mathbf{u}(t; z)}{e_1}$ satisfies \eqref{eq:2}. Here $\{e_1, e_2\}$ denotes 
the standard base in $\RR^2$, that is
\[
    e_1 = 
    \begin{pmatrix}
        1 \\
        0
    \end{pmatrix}, \quad
    e_2 = 
    \begin{pmatrix}
        0 \\
        1
    \end{pmatrix}.
\]
For any solutions $\mathbf{u},\mathbf{v}$ of \eqref{eq:10} we define their \emph{Wronskian} by the formula
\begin{equation}
	\label{eq:23a}
	\Wrk(\mathbf{u}, \mathbf{v})(t; z) = 
	p(t) \det \big( \mathbf{u}(t; z), \mathbf{v}(t; z) \big), \quad t \in \RR_+, z \in \CC.
\end{equation}
By \eqref{eq:10}, Liouville's formula (see e.g. \cite[formula (1.2.15)]{Zettl2005}) and \eqref{eq:5} we have
\begin{align}
	\label{eq:23}
	\det \big( \mathbf{u}(t; z), \mathbf{v}(t; z) \big) 
	&=
	\det \big( \mathbf{u}(t_0; z), \mathbf{v}(t_0; z) \big) \cdot 
	\exp \bigg( \int_{t_0}^t \tr b(s; z) \ud s \bigg) \\
	\nonumber
	&=
	\det \big( \mathbf{u}(t_0; z), \mathbf{v}(t_0; z) \big) \cdot 
	\exp \bigg( -\int_{t_0}^t \frac{p'(s)}{p(s)} \ud s \bigg) \\
	\nonumber
	&=
	\det \big( \mathbf{u}(t_0; z), \mathbf{v}(t_0; z) \big) \cdot 
	\frac{p(t_0)}{p(t)}.
\end{align}
Thus,
\begin{equation}
	\label{eq:23b}
	\Wrk(\mathbf{u}, \mathbf{v})(t; z) =
	p(t_0) \det \big( \mathbf{u}(t_0; z), \mathbf{v}(t_0; z) \big).
\end{equation}
Taking $t_0 = 0$, for $z \in \CC$, by $\mathscr{S}(\cdot; z)$ and $\mathscr{C}(\cdot; z)$ we denote two solutions of 
\eqref{eq:2}, with boundary conditions
\begin{equation}
	\label{eq:69}
	\begin{cases}
		\mathscr{C}(0; z) = 1, \\
		\partial_x \mathscr{C}(0; z) = 0,
	\end{cases}
	\qquad
	\begin{cases}
		\mathscr{S}(0; z) = 0, \\
		\partial_x \mathscr{S}(0; z) = 1.
	\end{cases}
\end{equation}
Then $\mathscr{C}$ and $\mathscr{S}$ are jointly continuous on $[0, \infty) \times \CC$ and entire with respect to $z$. 
We define the \emph{transfer matrix} as
\begin{equation}
	\label{eq:12}
	T(t; z) = 
	\begin{pmatrix}
		\mathscr{C}(t; z) & \mathscr{S}(t; z) \\
		\partial_x \mathscr{C}(t; z) & \partial_x \mathscr{S}(t; z)
	\end{pmatrix}.
\end{equation}
Then it satisfies the matrix equation
\begin{equation}
	\label{eq:47}
	T(t; z) = \Id + \int_0^t b(t'; z) T(t'; z) \ud t', \qquad t \in \RR_+.
\end{equation}
If $\mathbf{u}(\cdot; z)$ satisfies \eqref{eq:10} then for each $t, s \in \RR_+$,
\begin{equation}
	\label{eq:6}
	\mathbf{u}(t; z) = T(t; z) T(s; z)^{-1} \mathbf{u}(s; z).
\end{equation}
By \eqref{eq:23}, we get
\begin{equation}
	\label{eq:16}
	\frac{ \det T(t; z) } {\det T(s; z)} = \frac{p(s)}{p(t)}, \qquad \text{for all } s, t \in \RR_+.
\end{equation}
A list of explicit examples of Sturm--Liouville equations can be found in \cite[Chapter 14]{Zettl2005} and \cite{Everitt2005}.

\subsection{Classes of Sturm--Liouville parameters}
Since the periodic case is well-understand, our aim is to study a class which is its certain perturbation. For $\omega > 0$, we 
say that $(\frakp, \frakq, \frakw)$ are \emph{$\omega$-periodic Sturm--Liouville parameters} if $(\frakp, \frakq, \frakw)$ 
are Sturm--Liouville parameters satisfying
\[
	\frakp(t+\omega) = \frakp(t), \quad
	\frakq(t+\omega) = \frakq(t), \quad
	\frakw(t+\omega) = \frakw(t), \quad\text{ for all } t \in \RR_+.
\]
We assume that $\frakp \in \ACloc([0,\infty))$ is positive everywhere.
Let
\begin{equation}
	\label{eq:15}
	\mathfrak{b}(t; z) = 
	\begin{pmatrix}
		0 & 1 \\
		\frac{\frakq(t) - z \frakw(t)}{\frakp(t)} & -\frac{\frakp'(t)}{\frakp(t)}
	\end{pmatrix},
	\quad t \in \RR_+, z \in \CC.
\end{equation}
By $\frakT(\cdot; z)$ we denote the corresponding transfer matrix, that is
\begin{equation}
	\label{eq:50}
	\frakT(t; z) = \Id + \int_0^t \mathfrak{b}(t'; z) \frakT(t'; z) \ud t', \qquad t \in \RR_+.
\end{equation}

\begin{lemma}
	\label{lem:8}
	Let $z \in \CC$.
	If $\tr \frakT(\omega; z) \in (-2, 2)$ then $\partial_z \tr \frakT(\omega; z) \neq 0$.
\end{lemma}
\begin{proof}
	By \eqref{eq:50}, we have
	\[
		\partial_z \frakT(t; z) = \int_0^t 
		\frakb(s; z) \partial_z \frakT(s; z) + \partial_z \frakb(s; z) \frakT(s; z) {\: \rm d} s.
	\]
	Hence, by variation of parameters, we get
	\[
		\partial_z \frakT(t; z) = \frakT(t; z)
		\int_0^t \frakT(s; z)^{-1} \partial_z b(s; z)  \frakT(s; z) {\: \rm d} s.
	\]
	Suppose, contrary to our claim, that
	\begin{equation}
		\label{eq:18}
		\tr \partial_z \frakT(\omega; z) = 0.
	\end{equation}
	Setting
	\[
		\frakT(s; z) = 
		\begin{pmatrix}
			a_{11}(s; z) & a_{12}(s; z) \\
			a_{21}(s; z) & a_{22}(s; z)
		\end{pmatrix},
	\]
	we can write
	\begin{align*}
		\tr \partial_z \frakT(\omega; z) 
		=
		-
		\int_0^\omega
		\big(
		&a_{11}(\omega; z) a_{11}(s; z) a_{12}(s; z) - a_{12}(\omega; z) a_{11}(s; z)^2 \\
		&+ a_{21}(\omega; z) a_{12}(s; z)^2 - a_{22}(\omega; z) a_{11}(s; z) a_{12}(s; z)
		\big) \frac{\frakw(s)}{\frakp(s)} {\: \rm d} s.
	\end{align*}
	Since $\frakT(\cdot; z)$ is continuous, by \eqref{eq:18} we conclude that there is $s \in [0, \omega]$ such that
	\[
		\big(a_{11}(\omega; z) - a_{22}(\omega; z)\big) 
		a_{11}(s; z) a_{12}(s; z) 
		= a_{12}(\omega; z) a_{11}(s; z)^2 - a_{21}(\omega; z) a_{12}(s; z)^2.
	\]
	Thus
	\[
		\big(a_{11}(\omega; z) - a_{22}(\omega; z)\big)^2 
		a_{11}(s; z)^2 a_{12}(s; z)^2 
		= \big( a_{12}(\omega; z) a_{11}(s; z)^2 - a_{21}(\omega; z) a_{12}(s; z)^2\big)^2.
	\]
	Since $\det \frakT(\omega; z) = 1$, we have
	\[
		\big(a_{11}(\omega; z) - a_{22}(\omega; z)\big)^2
		= 
		\big( a_{11}(\omega; z) + a_{22}(\omega; z)\big)^2 
		- 4 - 4 a_{12}(\omega; z) a_{21}(\omega; z),
	\]
	therefore
	\[
		\big((a_{11}(\omega; z) + a_{22}(\omega; z))^2 - 4\big) a_{11}(s; z)^2 a_{12}(s; z)^2 \\
		=\big(a_{12}(\omega; z) a_{11}(s; z)^2 + a_{21}(\omega; z) a_{12}(s; z)^2\big)^2.
	\]
	Now, in view of $\tr \frakT(\omega; z) \in (-2, 2)$, we must have
	\[
		a_{11}(s; z)^2 a_{12}(s; z)^2 = 0
	\]
	and
	\[
		a_{12}(\omega; z) a_{11}(s; z)^2 + a_{21}(\omega; z) a_{12}(s; z)^2 = 0.
	\]
	Because $\det \frakT(s; z) = 1$, if $a_{11}(s; z) = 0$, then $a_{12}(s; z) \neq 0$, and so $a_{21}(\omega; z) = 0$. 
	Consequently,
	\[
		1 = \det \frakT(\omega; z) = a_{11}(\omega; z) a_{22}(\omega; z),
	\]
	which leads to contradiction
	\[
		\big(a_{11}(\omega; z) - a_{22}(\omega; z)\big)^2 
		=
		\big(a_{11}(\omega; z) + a_{22}(\omega; z)\big)^2 - 4 < 0.
	\]
	Analogously, we treat the case $a_{12}(s; z) = 0$.
\end{proof}
We say that $(p, q, w)$ are \emph{$\omega$-periodically modulated Sturm--Liouville parameters} if $(p, q, w)$ are
Sturm--Liouville parameters such that there are $\omega$-periodic Sturm--Liouville parameters $(\frakp, \frakq, \frakw)$,
with properties that $p, \frakp \in \ACloc([0,\infty))$ are positive everywhere
and
\begin{align}
	\label{eq:3a}
	&\lim_{n \to \infty} 
	\int_{0}^{\omega} \frac{w(n \omega + t)}{p(n\omega + t)} \ud t = 0, \\
	\label{eq:3b}
	&\lim_{n \to \infty} 
	\int_{0}^{\omega} \bigg|\frac{q(n \omega + t)}{p(n \omega + t)} - \frac{\frakq(t)}{\frakp(t)} \bigg| \ud t = 0, \\
	\label{eq:3c}
	&\lim_{n \to \infty} 
	\int_{0}^{\omega} \bigg|\frac{p'(n \omega + t)}{p(n\omega + t)} - \frac{\frakp'(t)}{\frakp(t)} \bigg| \ud t = 0.
\end{align}
We additionally assume that for each $t \in [0, \omega]$,
\begin{equation}
    \label{eq:3d}
    \lim_{t \to \infty} p(t+n\omega) = +\infty.
\end{equation}
For $\eta \in \tsS^1$ (see \eqref{eq:178}), we denote by $\mathbf{u}$ the solution to
\begin{equation}
	\label{eq:61}
	\mathbf{u}(t; z) = \eta + \int_0^t b(t'; z) \mathbf{u}(t'; z) \ud t', \quad t \geq 0,
\end{equation}
where $b(t; z)$ is given by the formula \eqref{eq:5}. Then we set
\begin{equation}
    \label{eq:179}
	\mathbf{u}_n(s, \eta; z) = \mathbf{u}(s + n \omega; z), \quad s \in [0,\omega], n \in \NN_0.
\end{equation}
We notice that by \eqref{eq:6} we have $\mathbf{u}_0(s,\eta;z) = T(s;z) \eta$ and
\begin{equation}
	\label{eq:14}
	\mathbf{u}_{n+1}(s, \eta; z) = X_n(s; z) \mathbf{u}_n(s, \eta; z)
\end{equation}
where
\begin{equation}
	\label{eq:99}
	X_n(s; z) = U_{s; n}(\omega; z)
\end{equation}
and whereas
\begin{equation}
	\label{eq:41}
	U_{s; n}(t; z) = T(s + t + n \omega; z) T(s+n\omega; z)^{-1}.
\end{equation}
Let us observe that for each $z \in \CC$, $U_{s; n}(\cdot; z)$ satisfies 
\begin{equation}
	\label{eq:52}
	U(t) = \Id + \int_0^t b_{s; n}(t; z) U(t') \ud t', \quad t \geq 0
\end{equation}
where
\[
	b_{s; n}(t; z) = b(s+t+n\omega; z).
\]
We set
\[
	p_{s; n}(t) = p(s + t + n \omega), \quad 
	q_{s; n}(t) = q(s + t + n \omega), \quad
	w_{s; n}(t) = w(s + t + n \omega),
\]
and
\[
	u_n(s, \eta; z) = \sprod{\mathbf{u}_n(s, \eta; z)}{e_1}.
\]
If $s = 0$ we just write $b_n(t) = b(t+n\omega)$, and
\[
	p_n(t) = p(t + n \omega), \quad
	q_n(t) = q(t + n \omega), \quad
	w_n(t) = w(t + n \omega).
\]
\begin{proposition}
	\label{prop:4}
	Let $\omega > 0$. Suppose that $(p, q, w)$ are $\omega$-periodically modulated Sturm--Liouville parameters,
	and let $\frakT$ be the transfer matrix corresponding to $(\frakp, \frakq, \frakw)$. Then
	\begin{equation}
		\label{eq:60}
		\lim_{n \to \infty} U_{s; n}(t; z) = \frakT_s(t; 0)
	\end{equation}
	locally uniformly with respect to $s, t \in [0, \omega]$ and $z \in \CC$, where
	\[
		\frakT_s(t; z) = \frakT(s+t; z) \frakT(s; z)^{-1}.
	\]
	Moreover, for every compact set $K \subset \CC$,
	\[
		\sup_{s,t \in [0, \omega]} \sup_{z \in K}
		\big\|
		U_{s; n+1}(s; z) - U_{s; n}(t; z)
		\big\|
		\leq
		C \int_0^{2 \omega} \big\| b_{n+1}(t') - b_n(t') \big\| {\: \rm d} t'.
	\]
\end{proposition}
\begin{proof}
	Since for each $z \in \CC$, $\frakT_s(\cdot; z)$ satisfies
	\[
		\mathfrak{U}(t) = \Id + \int_0^t \frakb_s(t'; z) \mathfrak{U}(t') \ud t', \quad t' \geq 0,
	\]
	by \eqref{eq:52}, the variation of parameters gives
	\begin{equation}	
		\label{eq:11}
		U_{s; n}(t; z) = 
		\frakT_s(t; 0) 
		+
		\frakT_s(t; 0)
		\int_0^t
		\frakT_s(t'; 0)^{-1} \big(b_{s; n}(t'; z) - \frakb_s(t'; 0)\big) U_{s;n} (t'; z) {\: \rm d} t'.
	\end{equation}
	Since $\|\frakT_s(t; 0)\|$ is uniformly bounded from above as well as from below, we have
	\[
		\big\| U_{s; n}(t; z) \big\| \leq 
		C \bigg(1 + \int_0^t \big\| b_{s; n}(t'; z) - \frakb_s(t'; 0)\big\|\cdot \big\|U_{s; n}(t'; z) \big\|
		{\: \rm d} t'\bigg).
	\]
	Thus, by the Gronwall's inequality
	\[
		\big\|U_{s; n}(t; z)\big\| \leq C \exp\bigg(\int_0^{\omega} \big\| b_{s; n}(t'; z) - \frakb_s(t'; 0)\big\| 
		{\: \rm d} t'\bigg).
	\]
	By \eqref{eq:3a}--\eqref{eq:3c}, $U_{s; n}$ is locally uniformly bounded. Hence, by \eqref{eq:11},
	\[
		\big\|U_{s;n}(t; z) - \frakT_s(t; 0)\big\|
		\leq
		C
		\int_0^t \big\| b_{s; n}(t'; z) - \frakb_s(t'; 0)\big\| {\: \rm d} t'.
	\]
	Moreover,
	\[
		\big\| b_{s; n}(t; z) - \frakb_s(t; 0)\big\|
		\leq
		\abs{z} \bigg|\frac{w_n(s+t)}{p_n(s+t)}\bigg|
		+
		\bigg| \frac{q_n(s+t)}{p_n(s+t)} - \frac{\frakq(s+t)}{\frakq(p+t)} \bigg|
		+
		\bigg| \frac{p'_n(s+t)}{p_n(s+t)} - \frac{\frakp'(s+t)}{\frakp(s+t)}\bigg|,
	\]
	thus 
	\begin{align*}
		\int_0^\omega \big\| b_{s; n}(t'; z) - \frakb_s(t'; 0)\big\| {\: \rm d} t'
		&=
		\int_s^{\omega+s}  \big\| b_n(t'; z) - \frakb(t'; 0)\big\| {\: \rm d} t' \\
		&=
		\bigg( \int_0^\omega + \int_\omega^{s+\omega} - \int_0^s\bigg) \big\| b_n(t'; z) - \frakb(t'; 0) \big\| {\: \rm d} t' \\
		&\leq
		\int_0^{2\omega} \big\| b_n(t'; z) - \frakb(t'; 0)\big\| {\: \rm d} t'.
	\end{align*}
	Consequently, by \eqref{eq:3a}--\eqref{eq:3c}, we obtain \eqref{eq:60}. Lastly, by \eqref{eq:11}
	\begin{align*}
		\big\|U_{s; n+1}(t; z) - U_{s; n}(t; z)\big\| 
		&\leq
		C \bigg(
		\int_0^t \big\|b_{s; n+1}(t'; z) - b_{s; n}(t'; z)\big\| \cdot \big\|U_{s; n}(t'; z)\big\| {\: \rm d} t' \\
		&\phantom{\leq C \bigg(}
		+ \int_0^t \big\|b_{s; n}(t'; z)\big\| \cdot \big\|U_{s; n+1}(t'; z) - U_{s; n}(t'; z)\big\| {\: \rm d} t'\bigg) \\
		&\leq
		C'\bigg(\int_0^t \big\|b_{s; n+1}(t'; z) - b_{s;n}(t'; z)\big\| {\: \rm d} t' \\
		&\phantom{\leq C'\bigg(}
		+ \int_0^t \big\|U_{s; n+1}(t'; z) - U_{s; n}(t'; z)\big\| {\: \rm d} t'\bigg).
	\end{align*}
	In view of the Gronwall's inequality,
	\begin{align*}
		\big\|U_{s; n+1}(t; z) - U_{s; n}(t; z)\big\|
		&\leq
		C'
		\int_0^\omega  \big\|b_{s; n+1}(t'; z) - b_{s;n}(t'; z)\big\| {\: \rm d} t' \\
		&\leq
		C' \int_0^{2\omega} \big\|b_{n+1}(t'; z) - b_{n}(t'; z)\big\| {\: \rm d} t'
	\end{align*}
	and the proposition follows.
\end{proof}

    \begin{corollary} \label{cor:8}
    Let $\omega > 0$. Suppose that $(p, q, w)$ are $\omega$-periodically modulated Sturm--Liouville parameters,
    and let $\frakT$ be the transfer matrix corresponding to $(\frakp, \frakq, \frakw)$. If $\tr \frakT(\omega; 0) \in (-2,2)$, then for any compact $K \subset \CC$ there are constants $c>0$ and $M \geq 1$ such that for any $z \in K$, $t \in [0,\omega]$ and $n \geq M$
    \begin{equation}
        \label{eq:150}
        |u_n(t,\eta;z)|^2
        \leq
        \| \mathbf{u}_n(t,\eta;z) \|^2 
        \leq 
        c \big( |u_n(t,\eta;z)|^2 + |u_{n+1}(t,\eta;z)|^2\big).
    \end{equation}
    \end{corollary}
    \begin{proof}
 	Let us fix $z \in K$. Since the sequence $(\mathbf{u}_n(t; z) : n \geq 0)$ satisfies \eqref{eq:14}, we have
	\[
		u_{n+1}(t; z) = [X_n(t; z)]_{11} \cdot u_n(t; z) + [X_n(t; z)]_{12} \cdot \partial_t u_n(t; z).
	\]
	In view of \eqref{eq:99} and Proposition \ref{prop:4},
	\[
		\lim_{n \to \infty} \discr X_n(t; z) = \discr \frakT(\omega; 0) < 0,
	\]
	hence there are $M \geq 1$ and $\delta > 0$ such that for all $n \geq M$,
	\[
		\abs{[X_n(t; z)]_{12}} \geq \delta > 0.
	\]
	Consequently, there is $c > 0$ such that for all $n \geq M_1$,
	\[
		\abs{\partial_t u_n(t; z)} \leq c (\abs{u_n(t; z)} + \abs{u_{n+1}(t; z)}\big),
	\]
	which easily implies \eqref{eq:150}.
    \end{proof}

\begin{remark}
	\label{rem:1}
	Let us observe that
	\[
		\frakT_s(\omega; z) = \frakT(s; z) \frakT(\omega; z) \frakT(s; z)^{-1}.
	\]
	Indeed, we notice that given $z \in \CC$, both $\frakT(\cdot ; z) \frakT(\omega; z)$ and $\frakT(\cdot + \omega; z)$ satisfies
	\[
		\mathfrak{U}(s) = \frakT(\omega; z) +
		\int_0^s \mathfrak{b}(s'; z) \mathfrak{U}(s') \ud s', \quad s \geq 0,
	\]
	which implies that
	\[
		\frakT(s; z) \frakT(\omega; z) = \frakT(s + \omega; z), \quad s \in [0, \omega], z \in \CC.
	\]
\end{remark}

\begin{lemma}
	\label{lem:3}
	Let $\omega > 0$. Let $p, \frakp \in \ACloc([0,\infty))$ be positive everywhere.
	Suppose that $\frakp$ is $\omega$-periodic and 
	\[
		\lim_{n \to \infty} 
		\int_0^\omega \bigg| \frac{p'_n(s)}{p_n(s)} - \frac{\frakp'(s)}{\frakp(s)}\bigg| {\: \rm d} s = 0,
	\]
	then
	\[
		\lim_{n \to \infty} \sup_{t \in [0, \omega]}
		{\bigg| \frac{p_n(t)}{p_n(0)}  -  \frac{\frakp(t)}{\frakp(0)} \bigg|} = 0.
	\]
\end{lemma}
\begin{proof}
	Since
	\[
		\log p_n(t) - \log p_n(0) = \int_0^t \frac{p'_n(s)}{p_n(s)}{\: \rm d} s,
	\]
	we have
	\begin{align*}
		\bigg|\log \bigg( \frac{p_n(t)}{p_n(0)}\bigg) - \log \bigg( \frac{\frakp(t)}{\frakp(0)} \bigg) \bigg|
		&\leq
		\int_0^t \bigg|\frac{p'_n(s)}{p_n(s)} - \frac{\frakp'(s)}{\frakp(s)} \bigg| {\: \rm d} s \\
		&\leq
		\int_0^\omega \bigg|\frac{p'_n(s)}{p_n(s)} - \frac{\frakp'(s)}{\frakp(s)} \bigg| {\: \rm d} s
	\end{align*}
	and the lemma follows.
\end{proof}

For each $s \in [0, \omega]$ and $n \in \NN$, we set
\[
	\gamma_n(s) = \int_0^\omega \frac{w_n(s + t)}{p_n(s +t)} \ud t,
\]
and
\[
	\gamma = \int_0^\omega \frac{\frakw(t)}{\frakp(t)} \ud t.
\]
\begin{lemma}
	\label{lem:5}
	Let $\omega > 0$. Suppose that $(p, q, w)$ are $\omega$-periodically modulated Sturm--Liouville parameters, and let 
	$\frakT$ be the transfer matrix corresponding to $(\frakp, \frakq, \frakw)$. Suppose that there is $s \in [0, \omega]$
	such that
	\begin{equation}
		\label{eq:119}
		\lim_{n \to \infty} \int_0^\omega \bigg| \frac{1}{\gamma_n(s)} \frac{w_n(s + t)}{p_n(s+t)} 
		- \frac{1}{\gamma} \frac{\frakw(s+t)}{\frakp(s+t)} \bigg| \ud t = 0,
	\end{equation}
	then
	\[
		\lim_{n \to \infty}
		\frac{1}{\gamma_n(s)}
		\partial_z X_n(s; z) =
		\frac{1}{\gamma}
		\partial_z \frakT_s(\omega; 0) 
	\]
	locally uniformly with respect to $z \in \CC$.
\end{lemma}
\begin{proof}
	Since $U_{s; n}$ satisfies \eqref{eq:52}, we get
	\[
		\partial_z U_{s; n}(t; z) = \int_0^t b_{s; n}(t'; z) \partial_z U_{s; n}(t'; z) 
		+ \partial_z b_{s; n}(t'; z) U_{s; n}(t'; z) \ud t', \quad t \geq 0.
	\]
	By variation of parameters, we obtain
	\[
		\partial_z U_{s; n}(t; z) = 
		U_{s; n}(t; z) \int_0^t U_{s; n}(t'; z)^{-1} \partial_z b_{s; n}(t'; z) U_{s; n}(t'; z) {\: \rm d} t'.
	\]
	Analogously, we have
	\begin{equation}
		\label{eq:53}
		\partial_z \frakT_s(t; 0) =
		\frakT_s(t; 0) \int_0^t \frakT_s(t'; 0)^{-1} \partial_z \frakb_s(t'; 0) \frakT_s(t'; 0) {\: \rm d} t'.
	\end{equation}
	Since
	\[
		\partial_z b_{s; n}(t'; z) = -\frac{w_n(s+t')}{p_{n}(s+t')} 
		\begin{pmatrix}
			0 & 0 \\
			1 & 0
		\end{pmatrix},
	\]
	and
	\[
		\partial_z b_{s; n}(t'; z) = -\frac{\frakw(s+t')}{\frakp(s+t')}
		\begin{pmatrix}
			0 & 0 \\
			1 & 0
		\end{pmatrix},
	\]
	by \eqref{eq:119}, we obtain
	\[
		\lim_{n \to \infty} 
		\int_0^\omega
		\bigg|
		\frac{1}{\gamma_n(s)}
		\partial_z b_{s; n}(t'; z)
		-
		\frac{1}{\gamma}
		\partial_z \frakb(t'; 0)
		\bigg|
		{\: \rm d} t' = 0.
	\]
	Moreover, by Proposition \ref{prop:4} and the dominated convergence theorem, we obtain
	\begin{align*}
		\lim_{n \to \infty}
		&U_{s; n}(\omega; z) 
		\int_0^\omega U_{s; n}(t', z)^{-1}
		\partial_z \frakb_s(t'; 0)
		U_{s; n}(t'; z) {\: \rm d} t' \\ 
		&=
		\frakT_s(\omega; 0) 
		\int_0^\omega \frakT_s(t'; 0)^{-1}
		\partial_z \frakb_s(t'; 0) \frakT_s(t'; 0)
		{\: \rm d} t',
	\end{align*}
	therefore
	\begin{align*}
		&
		\lim_{n \to \infty}
		\frac{1}{\gamma_n(s)} \partial_z X_{n}(s; z)
		-
		\frac{1}{\gamma}
		\partial_z \frakT_s(\omega; 0) \\
		&=
		\lim_{n \to \infty}
		U_{s; n}(\omega; z) 
		\int_0^\omega U_{s; n}(t', z)^{-1}
		\bigg(
		\frac{1}{\gamma_{n}(s)} 
		\partial_z b_{s; n}(t'; z)
		-
		\frac{1}{\gamma}
		\partial_z \frakb(t'; 0)
		\bigg)
		U_{s; n}(t'; z) {\: \rm d} t' = 0
	\end{align*}
	which in view of \eqref{eq:53} completes the proof.
\end{proof}

\begin{corollary}
	\label{cor:3}
	For each $s \in [0, \omega]$ and $z \in \CC$,
	\[
		\tr \partial_z \frakT_s(\omega; z) = \tr \partial_z \frakT(\omega; z).
	\]
\end{corollary}
\begin{proof}
	By Remark \ref{rem:1},
	\[
		\frakT_s(\omega; z) = \frakT(s; z) \frakT(\omega; z) \frakT(s; z)^{-1}.
	\]
	Hence,
	\begin{align*}
		\partial_z \frakT_s(\omega; z) 
		&= \partial_z \frakT(s; z) \frakT(\omega; z) \frakT(s; z)^{-1}
		+ \frakT(s; z) \partial_z \frakT(\omega; z) \frakT(s; z)^{-1} \\
		&\phantom{=}-
		\frakT(s; z) \frakT(\omega; z) \frakT(s; z)^{-1} \partial \frakT(s; z) \frakT(s; z)^{-1}.
	\end{align*}
	Since
	\begin{align*}
		\tr \Big(\frakT(s; z) \frakT(\omega; z) \frakT(s; z)^{-1} \partial \frakT(s; z) \frakT(s; z)^{-1}\Big)
		&=
		\tr \Big( \frakT(\omega; z) \frakT(s; z)^{-1} \partial \frakT(s; z) \Big) \\
		&=
		\tr \Big( \partial \frakT(s; z) \frakT(\omega; z) \frakT(s; z)^{-1} \Big),
	\end{align*}
	and
	\[
		\tr \Big( \frakT(s; z) \partial_z \frakT(\omega; z) \frakT(s; z)^{-1} \Big)
		=
		\tr \partial_z \frakT(\omega; z),
	\]
	the corollary follows.
\end{proof}

\begin{remark} 
	\label{rem:4}
	Certain \emph{discrete analogues} of Sturm--Liouville operators with unbounded parameters were studied in \cite{Clark1996}. 
	They are equivalent to Jacobi matrices acting on weighed $\ell^2$-spaces.
\end{remark}

\subsection{The Stolz class}
\label{sec:2}
Let $(X, \|\cdot\|_X)$ be a semi-normed space. A sequence of elements $(x_n : n \in \NN_0)$ of $X$ belongs to $\calD_1(X)$ if
\[
	\sup_{n \in \NN} \|x_n\|_X < \infty,
	\quad\text{and}\quad
	\sum_{n = 0}^\infty \|\Delta x_n\|_X < \infty
\]
where
\[
	\Delta x_n = x_{n+1} - x_n, \qquad n \in \NN_0.
\]
For example $\calD_1(L^1(I, X)))$ consists of sequences $(f_n) \subset L^1(I; X)$ such that
\[
	\sup_{n \in \NN_0} \int_I \|f_n(s)\|_X {\: \rm d} s < \infty, \quad\text{and}\quad
	\sum_{n = 0}^\infty \int_I \|\Delta f_n(s)\|_X {\: \rm d} s.
\]
Given $\omega > 0$, by $\calD_1^\omega(X)$ we denote the set of functions $f: [0, \infty) \rightarrow X$, such that
$(f(n\omega) : n \in \NN)$ belongs to $\calD_1(X)$. The set $\calD_1^\omega(L^1; X)$ consists of functions
$f: [0, \infty) \rightarrow X$ such that the sequence $(f(n\omega + s) : n \in \NN)$ belongs to 
$\calD_1\big(L^1_s( [0, \omega]; X)\big)$. Next, we say that $f$ belongs to $\calD_1^\omega\big(L^\infty L^1; X)\big)$ if
$(f(n\omega + t + s) : n \in \NN_0)$ belongs to $\calD_1\big(L_t^\infty([0,\omega]; L^1_s([0, \omega]; X))\big)$, that is
\[
	\sup_{t \in [0, \omega]} \sup_{n \in \NN_0} \int_0^\omega \big\|f(n \omega + t + s)\big\|_X {\: \rm d} s < \infty,
	\quad\text{and}\quad
	\sum_{n = 0}^\infty \sup_{t \in [0, \omega]} \int_0^\omega 
	\big\|\Delta f(n\omega + t+ s)\big\|_X {\: \rm d} s < \infty.
\]
Finally, we say that a collection $(f_t : t \in \calT)$ belongs to $\calD_1^\omega(L^1; X)$ uniformly with respect to
$t \in \calT$, if and only if
\[
	\sup_{t \in \calT} \sup_{n \in \NN_0} \int_0^\omega \big\|f_t(n\omega+s)\big\|_X {\: \rm d} s <\infty
\]
and the series
\[
	\sum_{n = 0}^\infty \int_0^\omega \big\|\Delta f_t(n\omega+s)\big\| {\: \rm d} s
\]
converges uniformly with respect to $t \in \calT$.

Let us notice that for $f : [0, \infty) \rightarrow X$, we have
\begin{equation}
	\label{eq:24}
	\begin{aligned}
	\sum_{n = 0}^\infty \int_0^\omega \big\|\Delta f(n \omega + s) \big\|_X {\: \rm d} s
	&=
	\sum_{n = 0}^\infty \int_{n\omega}^{(n+1)\omega} \big\|\Delta_\omega f(s) \big\|_X {\: \rm d} s \\
	&=
	\int_0^\infty \big\|\Delta_\omega f(s) \big\|_X {\: \rm d} s
	\end{aligned}
\end{equation}
where 
\[
	\Delta_\omega f(x) = f(x+\omega) - f(x).
\]
Hence, $f \in \calD_1^\omega(L^1; X)$ if and only if $\Delta_\omega f \in L^1([0, \infty); X)$ and
\[
	\sup_{n \in \NN_0} \int_0^\omega \|f(n\omega +s)\big\|_X {\: \rm d }s < \infty.
\]
For $f \in \calD_1^\omega(L^1; X)$ we also have
\begin{align*}
	\sum_{n = 0}^\infty \sup_{t \in [0, \omega]} \int_0^\omega \| \Delta f(n\omega+s+t)\|_X {\: \rm d} s 
	&=
	\sum_{n = 0}^\infty \sup_{t \in [0, \omega]} \int_{t+n\omega}^{t+(n+1)\omega} \| \Delta_\omega f(s)\|_X {\: \rm d} s \\
	&\leq
	\sum_{n = 0}^\infty \int_{n\omega}^{(n+2)\omega} \| \Delta_\omega f(s)\|_X {\: \rm d} s \\
	&\leq
	2 \int_0^\infty \| \Delta_\omega f(s)\|_X {\: \rm d} s,
\end{align*}
and
\begin{align*}
	\sup_{t \in [0, \omega]} \sup_{n \in \NN_0} \int_0^\omega \|f(n\omega + t + s)\|_X {\: \rm d } s
	&=
	\sup_{t \in [0, \omega]} \sup_{n \in \NN_0} \int_{n\omega+t}^{(n+1)\omega+t} \|f(s)\|_X {\: \rm d } s \\
	&\leq
	2 \sup_{n \in \NN_0} \int_0^\omega \|f(n\omega + s)\|_X {\: \rm d } s.
\end{align*}
Therefore, $f \in \calD^\omega_1(L^\infty L^1; X)$. In particular, $f_t$, where
\[
	f_t(s) = f(t + s), \quad s > 0,
\]
belongs to $\calD^\omega_1(L^1; X)$ uniformly with respect to $t \in [0, \omega]$.

\begin{proposition}
	\label{prop:2}
	Let $\omega > 0$. Suppose that $(p, q, w)$ are $\omega$-periodically modulated Sturm--Liouville parameters.
	Let $K$ be a compact subset of $\CC$. If
	\[
		\frac{q}{p}, \frac{w}{p}, \frac{p'}{p} \in \calD^\omega_1\big(L^1; \RR\big),
	\]
	then the sequence $(X_n : n \in \NN)$ belongs to $\calD_1\big([0, \omega]\times K; \GL(2, \CC))$.
\end{proposition}
\begin{proof}
	Since
	\begin{align*}
		\big\| b_{n+1}(t ; z) - b_{n} (t; z) \big\|
		\leq
		\bigg|\Delta_\omega \bigg( \frac{q}{p} \bigg) (t+n \omega)\bigg|
		+
		\bigg|\Delta_\omega \bigg(\frac{w}{p}\bigg) (t+n \omega)\bigg|
		\abs{z}
		+
		 \bigg|\Delta_\omega \bigg(\frac{p'}{p}\bigg) (t+n \omega)\bigg|,
	\end{align*}
	by Proposition \ref{prop:4}, we obtain
	\begin{align*}
		\sup_{t \in [0, \omega]} \sup_{z \in K}
		\big\|\Delta X_n(t; z) \big\|
		&\leq
		c\int_{n\omega}^{(n+2)\omega}
		\bigg|\Delta_\omega \bigg(\frac{q}{p}\bigg)(s)\bigg| 
		+
		\bigg|\Delta_\omega \bigg(\frac{w}{p}\bigg)(s)\bigg| 
		+
		\bigg|\Delta_\omega \bigg(\frac{p'}{p}\bigg)(s)\bigg| {\: \rm d} s,\\
	\end{align*}
	thus
	\begin{align*}
		\sum_{n = 0}^\infty 
		\sup_{t \in [0, \omega]} \sup_{z \in K}
		\big\|\Delta X_n(t; z) \big\|
		&\leq
		2c
		\int_0^\infty \bigg|\Delta_\omega \bigg(\frac{q}{p}\bigg)(s)\bigg| +
		\bigg|\Delta_\omega \bigg(\frac{w}{p}\bigg)(s)\bigg| +
		\bigg|\Delta_\omega \bigg(\frac{p'}{p}\bigg)(s)\bigg| {\: \rm d} s
	\end{align*}
	which completes the proof.
\end{proof}

\section{Diagonalization of transfer matrix}
\label{sec:3} 
\subsection{Analytic tools}
\label{sec:2.1}
Given $z \in \CC$, let us consider the equation
\[
	\xi^2 - 2 z \xi + 1 = 0.
\]
If $z \notin [-1, 1]$, it has two distinct solutions
\begin{equation}
	\label{eq:56}
	\xi_+(x) = z + \exp f(z), \quad\text{and}\quad \xi_-(x) = z - \exp f(z)
\end{equation}
where $\exp f(z)$ is the holomorphic function in $\CC \setminus [-1, 1]$ such that
\[
	\exp f(z) = \sqrt{z-1} \sqrt{z+1}, \quad z \in \CC \setminus (-\infty, 1]
\]
whereas $\sqrt{.}$ denotes the principal branch of the square root. Let us recall that
\begin{equation}
	\label{eq:19}
	\xi_+(\CC_{\varepsilon}) \subset \CC_{\varepsilon}, \quad \varepsilon \in \{-1, 1\}.
\end{equation}
We extend $\xi_+$ to $\CC$ by setting 
\[
	\xi_+(z) = \lim_{n \to +\infty} \xi_+(z_n)
\]
where $(z_n : n \in \NN) \subset \CC_+$ converges to $z \in (-1, 1)$. We set $\xi_-(z) = 2z - \xi_+(z)$.

\subsection{Diagonalization}
\label{sec:2.2}
Let us describe diagonalization procedure of transfer matrices. Given $\omega > 0$, let $(p, q, w)$ be
$\omega$-periodically modulated Sturm--Liouville parameters with differentiable $p$. Suppose that the transfer matrix $\frakT$ 
corresponding to $(\frakp, \frakq, \frakw)$ satisfies $\abs{\tr \frakT(\omega; 0)} < 2$. In view of Remark \ref{rem:1},
for each $s \in [0, \omega]$,
\[
	\tr \frakT_s(\omega; 0) = \tr \frakT(\omega; 0), \quad\text{and}\quad
	\det \frakT_s(\omega; 0) = \det \frakT(\omega; 0),
\]
and so
\[
	\discr \frakT_s(\omega; 0) = \discr \frakT(\omega; 0) < 0.
\]
In particular, we must have we must have $[\frakT_s(\omega; 0)]_{12} \neq 0$. Therefore, by Proposition \ref{prop:4}, for each 
compact set $K \subset \CC$ there are $\delta > 0$ and $L_0 \geq 1$ such that for all $n \geq L_0$, $t \in [0, \omega]$ and 
$z \in K$, 
\begin{equation}
	\label{eq:48}
	\big|\discr X_n(t; z)\big| \geq \delta, \quad \text{and} \quad 
	\abs{[X_n(t; z)]_{12}} \geq \delta.
\end{equation}
Consequently, $X_n(t; z)$ has two distinct eigenvalues
\begin{equation} \label{eq:106}
	\lambda^+_n(t; z) = \sqrt{\det X_n(t; z)} \: \xi_+\bigg(\frac{\tr X_n(t; z)}{2 \sqrt{\det X_n(t; z)}} \bigg),
	\quad\text{and}\quad
	\lambda^-_n(t; z) = \sqrt{\det X_n(t; z)} \: \xi_-\bigg(\frac{\tr X_n(t; z)}{2 \sqrt{\det X_n(t; z)}} \bigg)
\end{equation}
where $\xi_+$ and $\xi_-$ are constructed in Section \ref{sec:2.1}. For $z \in K$ we set
\[
	\lambda^+_\infty = \frac{\tr \frakT(\omega; 0) + i \sqrt{-\discr \frakT(\omega; 0)}}{2},
	\quad\text{and}\quad
	\lambda^-_\infty = \frac{\tr \frakT(\omega; 0) - i \sqrt{-\discr \frakT(\omega; 0)}}{2}.
\]
\begin{lemma}
	\label{lem:6}
	Let $\omega > 0$. Suppose that $(p, q, w)$ are $\omega$-periodically modulated Sturm--Liouville parameters
	such that the transfer matrix $\frakT$ corresponding to $(\frakp, \frakq, \frakw)$ satisfies 
	$\abs{\tr \frakT(\omega; 0)} < 2$. Let $K$ be a compact subset of $\CC_\varsigma$ with nonempty interior where 
	$\varsigma = \sign{\tr \partial_z \frakT(\omega; 0)}$. Suppose that there is $s \in [0, \omega]$ such that
	\begin{equation}
		\label{eq:51}
		\lim_{n \to \infty} \int_0^\omega \bigg| \frac{1}{\gamma_n(s)} \frac{w_n(s + t)}{p_n(s + t)} -
		\frac{1}{\gamma} \frac{\frakw(s+t)}{\frakp(s+t)} \bigg| \ud t = 0.
	\end{equation}
	Then there is $M \geq 1$ such that for all $n \geq M$ and $z \in K$,
	\[
		\Im \bigg(\frac{\tr X_n(s; z)}{2 \sqrt{\det X_n(s; z)}} \bigg) > 0.
	\]
\end{lemma}
\begin{proof}
	By Lemma \ref{lem:8} and Corollary \ref{cor:3}, 
	\[
		\partial_z \tr \frakT_s(\omega; 0) = \partial_z \tr \frakT(\omega; 0) \neq 0.
	\]
	Let $z = x + i y$. By the mean value theorem
	\[
		\tr X_n(s; x+iy) - \tr X_n(s; x) = i y \tr \big( \partial_z X_n(s; x+i\xi)\big)
	\]
	where $\xi$ depends on $x$, $y$, $t$ and $n$, and it belongs to $(0, y)$. Since $\tr X_n(s; x)$ is a real number we have
	\begin{equation}
		\label{eq:17}
		\Im \big( \tr X_n(s; x+iy) \big) = y \Re \Big( \tr \big(\partial_z X_n(s; x + i \xi) \big)\Big).
	\end{equation}
	By Lemma \ref{lem:5} and Corollary \ref{cor:1}, we have
	\[
		\lim_{n \to \infty} 
		\frac{1}{\gamma_n(s)} 
		\tr \partial_z X_n(s; z) = 
		\frac{1}{\gamma}
		\tr \partial_z \frakT(\omega; 0),
	\]
	locally uniformly with respect to $z \in \CC$. Since $\tr \partial_z \frakT(\omega; 0) \in \RR \setminus \{0\}$, for
	almost every $s \in [0, \omega]$ there is $M=M_s \geq 1$ such that for all $n \geq M$, and $z \in K$,
	\[
		\mathrm{sign} \Big( \Re \big( \tr \partial_z X_n(s; z) \big)\Big)
		= 
		\mathrm{sign} \Big( \tr \partial_z \frakT(\omega; 0)\Big)
	\]
	which completes the proof.
\end{proof}
In view of Lemma \ref{lem:6}, if $K \subset \CC_\varsigma$ where $\varsigma = \sign{\tr \partial_z \frakT(\omega; 0)}$, then
for each $s \in [0, \omega]$ satisfying \eqref{eq:51}, the functions $\lambda^+_n(s; \cdot)$ and $\lambda^-_n(s; \cdot)$
are continuous on $K$ and holomorphic in $\intr{K}$.

\begin{proposition}
	\label{prop:1}
	Let $\omega > 0$. Suppose that $(p, q, w)$ are $\omega$-periodically modulated Sturm--Liouville, such that the transfer 
	matrix $\frakT$ corresponding to $(\frakp, \frakq, \frakw)$ satisfies $\abs{\tr \frakT(\omega; 0)} < 2$. Let $K$ be a 
	compact subset of $\CC_\varsigma$ with nonempty interior where $\varsigma = \sign{\tr \partial_z \frakT(\omega; 0)}$. 
	Suppose that there is $s \in [0, \omega]$ such that
	\[
		\lim_{n \to \infty} \int_0^\omega \bigg|\frac{1}{\gamma_n(s)} \frac{w_n(s+t)}{p_n(s+t)} 
		- \frac{1}{\gamma} \frac{\frakw(s+t)}{\frakp(s+t)} \bigg|{\: \rm d }t = 0.
	\]
	If the Carleman's condition \eqref{eq:26} is satisfied, then there is $M \geq 1$ such that
	\begin{equation} 
		\label{eq:77}
		\inf_{z \in K} \prod_{j \geq M} \bigg| \frac{\lambda^+_j(s; z)}{\lambda^-_j(s; z)}\bigg| =
		+\infty.
	\end{equation}
\end{proposition}
\begin{proof}
	By \cite[Proposition 5.4]{SwiderskiTrojan2023}, it is enough to show that
	\[
		\sum_{n \geq m} \inf_{z \in K}
		\bigg(\bigg|\xi_+\bigg(\frac{\tr X_n(s; z)}{2 \sqrt{\det X_n(s; z)}}\bigg)\bigg| - 1\bigg) = \infty.
	\]
	We set
	\[
		y_j(s; z) = \frac{\tr X_j(s; z)}{2 \sqrt{\det X_n(s; z)}}, \quad z \in K.
	\]
	Observe that
	\[
		\lim_{j \to \infty} y_j(t; z) = \frac{\tr \frakT(\omega; 0)}{2} \in (-1, 1)
	\]
	uniformly with respect to $z \in K$. Moreover, by \eqref{eq:16}
	\[
		\det X_j(s; z) = \frac{p_j(s)}{p_{j+1}(s)}
	\]
	thus by Lemma \ref{lem:3},
	\[
		\lim_{j \to \infty} \det X_j(s; z) = 1
	\]
	uniformly with respect to $z \in K$. Since
	\[
		\abs{\xi_+(v)} -1 \geq \frac{\abs{\Im v}}{2\sqrt{\abs{v-1} \abs{v+1}}}
	\]
	(see \cite[Proposition 5.3]{SwiderskiTrojan2023}), it is enough to show that
	\begin{equation}
		\label{eq:27}
		\sum_{j \geq M} \inf_{z \in K} \big| \Im \tr X_j(s; z) \big| = + \infty.
	\end{equation}
	By \eqref{eq:17} and Lemma \ref{lem:5}, there are $M \geq 1$ and $c > 0$ such that for all $n \geq M$
	and $z \in K$,
	\[
		\big|\Im \tr X_n(s; z) \big| \geq c \gamma_n(s),
	\]
	thus
	\begin{align*}
		\sum_{n \geq M} \inf_{z \in K} \big| \Im \tr X_j(s; z) \big|
		&\geq
		c \sum_{n \geq M} \gamma_n(s).
	\end{align*}
	Since
	\[
		\sum_{N \geq n \geq M} \gamma_n(s) = \int_{s+M \omega}^{s+N\omega} \frac{w(t)}{p(t)} \ud t
		\geq \int_{(M+1)\omega}^{N\omega} \frac{w(t)}{p(t)} \ud t,
	\]
	by \eqref{eq:26}, we conclude \eqref{eq:27} and the proposition follows.
\end{proof}

\begin{proposition}
	\label{prop:3}
	Let $\omega > 0$. Suppose that $(p, q, w)$ are $\omega$-periodically modulated Sturm--Liouville parameters, such that 
	the transfer matrix $\frakT$ corresponding to $(\frakp, \frakq, \frakw)$ satisfies $\abs{\tr \frakT(\omega; 0)} < 2$. 
	Let $K$ be a compact subset of $\CC_\varsigma$ with nonempty interior where $\varsigma 
	= \sign{\tr \partial_z \frakT(\omega; 0)}$. Suppose that there is $s \in [0, \omega]$, such that
	\[
		\lim_{n \to \infty} 
		\int_0^\omega \bigg| \frac{1}{\gamma_n(s)} \frac{w_n(s+t)}{p_n(s+t)} -
		\frac{1}{\gamma} \frac{\frakw(s+t)}{\frakp(s+t)} \bigg| \ud t = 0,
	\]
	then
	\[
		\lim_{n \to \infty} \lambda_n^+(s; z) = \lambda_\infty^+
	\]
	uniformly with respect to $z \in K$. Additionally,
	\[
		\lim_{n \to \infty}
		\frac{1}{\gamma_n(s)}
		\frac{\partial_z \lambda^+_n(s; z)}{\lambda_n^+(s; z)} =
		\frac{1}{\gamma}
		\frac{\partial_z \tr \frakT(\omega; 0)}{i \sqrt{-\discr \frakT(\omega; 0)}} 
	\]
	uniformly with respect to $z \in K$.
\end{proposition}
\begin{proof}
	We set
	\[
		y_n(s; z) = \frac{\tr X_n(s; z)}{2 \sqrt{\det X_n(s; z)}}.
	\]
	By Lemma \ref{lem:6} there is $M \geq 1$ such that for all $n \geq M$ and $z \in K$, we have $y_n \in \CC_+$, thus by 
	\eqref{eq:19} $\xi_+(y_n) \in \CC_+$. Since by \eqref{eq:16}
	\[
		\det X_n(s; z) = \frac{p_n(s)}{p_{n+1}(s)},
	\]
	we deduce that 
	\begin{equation}
		\label{eq:28}
		\Im\big(\lambda_n^+(s; z) \big) > 0.
	\end{equation}
	We next observe that the limit
	\begin{equation}
		\label{eq:29}
		\lim_{n \to \infty} \lambda^+_n(s; z) = \lim_{n \to \infty} \xi_+\big(y_n(s; z)\big)
	\end{equation}
	exists and satisfies 
	\[
		\lambda^2 - (\tr \frakT(\omega; 0)) \lambda + 1 = 0,
	\]
	which is the characteristic equation of $\frakT(\omega; 0)$. The equation has two solutions $\lambda_\infty^+$ and 
	$\lambda_\infty^-$. By \eqref{eq:28}, we must have
	\begin{equation}
		\label{eq:46}
		\lim_{n \to \infty} \lambda^+_n(s; z) = \lambda_\infty^+.
	\end{equation}
	We now compute
	\begin{align*}
		\frac{\partial_z \lambda_n^+(s; z)}{\lambda_n^+(s; z)} 
		&= 
		\frac{\xi_+'(y_n(s; z))}{\xi_+(y_n(s; z))} \partial_z y_n(s; z) \\
		&=
		\frac{\partial_z y_n(s; z)}{\xi_+(y_n(s; z)) - y_n(s; z)}.
	\end{align*}
	By \eqref{eq:29} and \eqref{eq:46}, we get
	\[
		\lim_{n \to \infty} \xi_+\big(y_n(s; z)\big) - y_n(s; z) 
		= 
		\frac{i}{2} \sqrt{-\discr \frakT(\omega; 0)}.
	\]
	Lastly, by Lemma \ref{lem:5}  
	\begin{align*}
		\lim_{n \to \infty}
		\frac{1}{\gamma(s)}
		\partial_z y_n(s; z) 
		&=
		\frac{1}{2}
		\lim_{n \to \infty}
		\frac{1}{\gamma_n(s)}
		\tr \partial_z X_n(s; z) \\
		&=
		\frac{1}{2} \cdot 
		\frac{1}{\gamma}
		\partial_z \tr \frakT(\omega; 0)
	\end{align*}
	which easily completes the proof.
\end{proof}

By \eqref{eq:48}, for each $t \in [0, \omega]$, $z \in K$ and $n \geq L_0$, the matrix $X_n(t; z)$ can be diagonalized
\begin{equation}
	\label{eq:31a}
	X_n(t; z) = C_n(t; z) D_n(t; z) C_n(t; z)^{-1}
\end{equation}
where
\begin{equation}
	\label{eq:31b}
	C_n = 
	\begin{pmatrix}
		1 & 1 \\
		\frac{\lambda_n^+ - [X_n]_{11}}{[X_n]_{12}} & \frac{\lambda_n^- - [X_n]_{11}}{[X_n]_{12}}
	\end{pmatrix}
	\quad\text{and}\quad
	D_n = \begin{pmatrix}
	\lambda_n^+ & 0 \\
	0 & \lambda_n^-
	\end{pmatrix}.
\end{equation}
By restricting to real line, we have the following two statements.
\begin{claim}
	\label{clm:3}
	For all $n > m \geq L_0$, $t \in [0, \omega]$ and $z \in \RR \cap K$, we have
	\begin{equation}
		\label{eq:33}
		p_n(t) \prod_{k = m}^{n-1} \abs{\lambda_k^+(t; z)}^2
		=
		p_m(t).
	\end{equation}
\end{claim}
To see this, let us observe that for all $t \in [0, \omega]$ and $z \in \RR$ we have
\[
	\lambda_n^-(t; z) = \overline{\lambda_n^+(t; z)}.
\]
Hence, for $k \geq L_0$,
\[
	|\lambda_k^+(t; z)|^2 
	= \lambda_k^+(t; z) \overline{\lambda_k^+(t; z)} \\
	= \det D_k(t; z) 
	= \det X_k(t; z) \\
	= \frac{p_{k}(t)}{p_{k+1}(t)},
\]
which easily leads to \eqref{eq:33}.

\begin{claim}
	\label{clm:1}
	There is $c > 0$ such that for all $n > m \geq L_0$, $\eta \in \sS^1$, $s \in [0, \omega]$ and $z \in \RR \cap K$,
	\[
		\|\mathbf{u}_n(s, \eta; z) \| \leq c \Big(\prod_{k = m}^{n-1} \abs{\lambda_k^+(t; z)} \Big)
		\|\mathbf{u}_m(s, \eta; z)\|.
	\]
\end{claim}
For the proof, we observe that by \eqref{eq:14}
\[
	\mathbf{u}_n = X_n \cdots X_m \mathbf{u}_{m}.
\]
Since by \eqref{eq:31a},
\[
	X_n X_{n-1} \cdots X_m = C_n \bigg(\prod_{k = m}^n D_k C_k C_k^{-1} \bigg) C_{m-1}^{-1},
\]
by \cite[Proposition 1]{SwiderskiTrojan2019}, we get
\[
	\big\|\mathbf{u}_n \big\| \leq c \Big( \prod_{k = m}^{n-1} \|D_k\| \Big) \|\mathbf{u}_{m}\|
\]
which proves the claim since $\|D_j\| = \abs{\lambda_j}$.

Next, in view of \eqref{eq:24}, Proposition \ref{prop:2} and \cite[Lemma 2]{SwiderskiTrojan2019}, we obtain the following 
corollary.
\begin{corollary}
	\label{cor:2}
	Let $\omega > 0$. Suppose that $(p, q, w)$ are $\omega$-periodically modulated Sturm--Liouville parameters, such that 
	the transfer matrix $\frakT$ corresponding to $(\frakp, \frakq, \frakw)$ satisfies $\abs{\tr \frakT(\omega; 0)} < 2$. 
	Let $K$ be a compact subset of $\CC_\varsigma$ with nonempty interior where 
	$\varsigma = \sign{\tr \partial_z \frakT(\omega; 0)}$. If
	\[
		\frac{q}{p}, \frac{w}{p}, \frac{p'}{p} \in \calD^\omega_1\big(L^1; \RR\big),
	\]
	then there is $M \geq 1$ such that both sequences $(C_n : n \geq M)$ and $(D_n : n \geq M)$ belong to 
	$\calD_1\big([0,\omega]\times K; \GL(2, \CC)\big)$.
\end{corollary}

\section{Tur\'an determinants}
\label{sec:4}
In this section we introduce generalized Tur\'an determinants, and study their asymptotic behavior in the case of 
$\omega$-periodically modulated Sturm--Liouville parameters. To be precise, assume that $(p, q, w)$ are $\omega$-periodically
modulated Sturm--Liouville parameters. If $u$ is a non-trivial solution of \eqref{eq:2},
then \emph{generalized $\omega$-shifted Tur\'an determinants} is defined as
\begin{align*}
	\scrD(t; z) 
	&=
	\det
	\begin{pmatrix}
		u(t+\omega;z) & u(t;z) \\
		\partial_t u(t+\omega;z) & \partial_t u(t; z)
	\end{pmatrix} \\
	&=
	u(t+\omega; z) \partial_t u(t; z) - u(t; z) \partial_t u(t+\omega; z).
\end{align*}
Given $\eta \in \tsS^1$ (cf. \eqref{eq:178}) and $s \in [0, \omega]$, $n \in \NN_0$ we set
\[
	\scrD_n(s, \eta; z) = \scrD(s + n \omega; z), \quad z \in \RR
\]
where $u = \sprod{\mathbf{u}}{e_1}$ and $\mathbf{u}$ satisfies \eqref{eq:61}.
\begin{theorem}
	\label{thm:1}
	Let $\omega > 0$. Suppose that $(p, q, w)$ are $\omega$-periodically modulated Sturm--Liouville parameters, such that 
	the transfer matrix $\frakT$ corresponding to $(\frakp, \frakq, \frakw)$ satisfies $\abs{\tr \frakT(\omega; 0)} < 2$. 
	Assume that 
	\[
		\frac{q}{p}, \frac{w}{p}, \frac{p'}{p} \in \calD_1^\omega(L^1; \RR).
	\]
	Then for each solution $\mathbf{u}$ of \eqref{eq:61} the limit 
	\[
		\lim_{n \to \infty} p_n(t) |\scrD_n(t, \eta; z)|
	\]
	exists locally uniformly with respect to $(t, \eta, z) \in [0, \omega] \times \tsS^1 \times \RR$, and it is a positive 
	continuous function.
\end{theorem}
\begin{proof}
	We write
	\[
		S_n = p_{n+1} \scrD_n = \sprod{E \mathbf{u}_{n+1}}{\mathbf{u}_{n}}
	\]
	where
	\[
		E = 
		\begin{pmatrix}
			0 & -1\\
			1 & 0
		\end{pmatrix}.
	\]
	Since for each $X \in \GL(2, \RR)$,
	\begin{equation}
		\label{eq:62}
		X^t E = (\det X) E X^{-1},
	\end{equation}
	by \eqref{eq:14}, we get
	\begin{align*}
		S_n 
		&= p_{n+1} \sprod{E X_{n+1}^{-1} \mathbf{u}_{n+2}}{X_n^{-1} \mathbf{u}_{n+1}} \\
		&= p_{n+1} \sprod{(X_n^{-1})^t E X_{n+1}^{-1} \mathbf{u}_{n+2}} {\mathbf{u}_{n+1}} \\
		&= p_{n+1} \det X_n^{-1} \sprod{ E X_n X_{n+1}^{-1} \mathbf{u}_{n+2}} {\mathbf{u}_{n+1}}.
	\end{align*}
	Recall that by \eqref{eq:16}
	\begin{equation}
		\label{eq:44}
		\det X_n(t; z) = \frac{p_n(t)}{p_{n+1}(t)}.
	\end{equation}
	Consequently, for each $k \geq 1$,
	\begin{align*}
		S_n 
		&= p_{n+1} \big(\det X_n^{-1} \cdots \det X_{n+k}^{-1}\big)
		\bigg\langle E  \bigg(\prod_{j = 0}^{k-1} X_{n+j}\bigg) \bigg(\prod_{j=1}^k X_{n+j}\bigg)^{-1} 
		\mathbf{u}_{n+k}, \mathbf{u}_{n+k-1} \bigg\rangle \\
		&= p_{n+1} \frac{p_{n+k+1}}{p_n} 
		\bigg\langle E  \bigg(\prod_{j = 0}^{k-1} X_{n+j}\bigg) \bigg(\prod_{j=1}^k X_{n+j}\bigg)^{-1}
		\mathbf{u}_{n+k}, \mathbf{u}_{n+k-1} \bigg\rangle.
	\end{align*}
	Therefore,
	\[
		S_n - S_{n+k}
		=
		p_{n+k+1} \bigg\langle E Y_{n,k} \bigg(\prod_{j=1}^k X_{n+j} \bigg)^{-1} 
		\mathbf{u}_{n+k}, \mathbf{u}_{n+k-1} \bigg\rangle
	\]
	where
	\[
		Y_{n, k}= \bigg(\prod_{j=0}^{k-1} X_{n+j} \bigg) - \frac{p_{n+1}}{p_n} \bigg(\prod_{j=1}^k X_{n+j} \bigg),
	\]
	which by the Cauchy--Schwarz inequality leads to
	\[
		\big|S_n - S_{n+k} \big|
		\leq
		c
		p_{n+k+1} \bigg\| \bigg(\prod_{j=1}^k X_{n+j}\bigg)^{-1} \bigg\| \cdot \big\|Y_{n, k} \big\| 
		\cdot \big\|\mathbf{u}_{n+k}\big\|^2.
	\]
	Let $K$ be a compact set in $\RR$. By Proposition \ref{prop:4}
	\begin{equation}
		\label{eq:21}
		\lim_{n \to \infty} \sup_{z \in K} \sup_{t \in [0, \omega]} 
		\big\| X_n(t; z) - \frakT_t(\omega; 0) \big\| = 0,
	\end{equation}
	thus there are $n_0 \geq 1$ and $\delta > 0$ such that for all $n \geq n_0$, $t \in [0, \omega]$ and $z \in K$,
	\begin{equation}
		\label{eq:20}
		\discr X_n(t; z) \leq -\delta.
	\end{equation}
	Since
	\[
		\det \big( \sym (E X_n)\big) = - \frac{1}{4} \discr X_n,
	\]
	there is $c > 0$ such that for all $n \geq n_0$, $t \in [0, \omega]$, $\eta \in \tsS^1$, and $z \in K$,
	\[
		\big|\scrD_n(t, \eta; z)\big| = 
		\big| \sprod{E X_n(t; z) {\mathbf{u}_n(t, \eta; z)}}{\mathbf{u}_{n}(t, \eta; z)} \big|
		\geq
		c \| \mathbf{u}_n(t, \eta; z) \|^2.
	\]
	Hence,
	\begin{align*}
		\bigg|
		\frac{S_{n}}{S_{n+k}} - 1
		\bigg|
		\leq
		c \bigg\| \bigg( \prod_{j=1}^k X_{n+j} \bigg)^{-1} \bigg\| \cdot
		\Bigg(
		\bigg\|\prod_{j=0}^{k-1} X_{n+j}  - \prod_{j=1}^k X_{n+j} \bigg\|
		+
		\bigg|1 - \frac{p_{n+1}}{p_n}\bigg| \cdot \bigg\| \prod_{j=1}^k X_{n+j}\bigg\|\Bigg).
	\end{align*}
	In view of Lemma \ref{lem:3}, 
	\[
		\lim_{n \to \infty} \frac{p_{n+1}(t)}{p_{n}(t)} = 1
	\]
	uniformly with respect to $t \in [0, \omega]$.
	Since by Section \ref{sec:2.2}, $(X_n : n \in \NN_0)$ is uniformly diagonalizable on $[0, \omega] \times K$, 
	\cite[Proposition 1 \& 2]{SwiderskiTrojan2019} implies that there is $c > 0$ such that for each $\epsilon > 0$ there is 
	$n_1 \geq n_0$, such that for all $n \geq n_1$, $k \in \NN$, 
	\begin{align*}
		\bigg| \frac{S_{n}}{S_{n+k}} - 1 \bigg|
		&\leq
		c \bigg(\prod_{j = 1}^k \|D_{n+j}\|^{-1} \bigg)\cdot
		\bigg(\epsilon \prod_{j = 1}^k \|D_{n+j}\| + \bigg|1 - \frac{p_{n+1}}{p_n} \bigg| \prod_{j = 0}^k \|D_{n+j}\|\bigg) \\
		&=
		c \bigg( \epsilon + \bigg|1 - \frac{p_{n+1}}{p_n} \bigg| \|D_n\|\bigg)
	\end{align*}
	uniformly on $[0, \omega] \times \tsS^1 \times K$. Consequently, for each $\epsilon > 0$ there is $n_2 \geq n_1$ such that 
	for all $m \geq n \geq n_2$,
	\begin{equation}
		\label{eq:22}
		\sup_{(t, \eta, z) \in [0, \omega] \times \tsS^1 \times K} \bigg|\frac{S_n(t, \eta; z)}{S_m(t, \eta; z)} - 1\bigg| 
		\leq \epsilon.
	\end{equation}
	Since for all $u \in \RR$,
	\[
		\log \abs{u} \leq \abs{1-\abs{u}},
	\]
	we obtain
	\[
		\big| \log \abs{S_n} - \log \abs{S_m} \big| 
		\leq
		\bigg|1 - \frac{S_n}{S_m} \bigg|
		\leq \epsilon,
	\]
	thus $(\log \abs{S_n} : n \in \NN)$ is a uniform Cauchy sequence of continuous functions on 
	$[0, \omega] \times \tsS^1 \times K$. Therefore, it uniformly converges to continuous function on
	$[0, \omega] \times \tsS^1 \times K$. In particular, $(\abs{S_n} : n \in \NN)$ is uniformly bounded on 
	$[0, \omega] \times \tsS^1 \times K$. Thus, by \eqref{eq:22}, $(S_n : n \in \NN)$ is a uniform Cauchy sequence 
	of continuous functions on $[0, \omega] \times \tsS^1 \times K$. This completes the proof of the theorem.
\end{proof}

\begin{corollary}
	\label{cor:1}
	Suppose that the hypotheses of Theorem \ref{thm:1} are satisfied. Then for any compact subset $K \subset \RR$, there is
	$c > 0$ such that for each solution $\mathbf{u}$ of \eqref{eq:10}, every $n \in \NN$, $t \in [0, \omega]$, $\eta \in \tsS^1$,
	and $z \in K$, we have
	\[
		c^{-1} \leq \sqrt{p_n(t)} \big\|\mathbf{u}_n(t, \eta; z) \big\| \leq c. 
	\]
\end{corollary}
\begin{proof}
	We have
	\[
		p_{n+1}(t) \scrD_n(t, \eta; z) \leq c p_{n+1}(t) \big\|\mathbf{u}_n(t, \eta; z)\big\|^2
	\]
	thus the lower estimate is a consequence of Theorem \ref{thm:1} and Lemma \ref{lem:3}. For the upper estimate
	it is enough to use Claim \ref{clm:1} together with Claim \ref{clm:3}.
\end{proof}

\begin{corollary}
	\label{cor:7}
	Suppose that the hypotheses of Theorem~\ref{thm:1} are satisfied. Then $\tau$ is in the circle case if the Carleman condition \eqref{eq:26} is \emph{violated}. Moreover, if $w \in \calD_1^\omega(L^1;\RR)$ and the Carleman's condition is satisfied, then $\tau$ is in the limit point case and for any $\eta \in \sS^1$ and compact $K \subset \RR$ there are constants $c_1,c_2>0$ such that
    \[
        c_1 < \mu_\eta'(\lambda) < c_2
    \]
    for almost all $\lambda \in K$.
\end{corollary}
\begin{proof}
    By Corollary~\ref{cor:8} we have that for any compact $K \subset \RR$ there are constants $c_3>0$ and $M \geq 1$ such that for any $z \in K$, $t \in [0,\omega]$, $\eta \in \tsS^1$ and $n \geq M$,
    \begin{equation}
        \label{eq:156}
        |u_n(t,\eta;z)|^2
        \leq
        \| \mathbf{u}_n(t,\eta;z) \|^2 
        \leq 
        c_3 \big( |u_n(t,\eta;z)|^2 + |u_{n+1}(t,\eta;z)|^2\big).
    \end{equation}
    In particular, by \eqref{eq:156} and Corollary~\ref{cor:1}
    \begin{align*}
        \int_{M\omega}^\infty |u(x,\eta;z)|^2 w(x) \ud x 
        &\leq 
        c_3 \int_{M\omega}^\infty \| \mathbf{u}(x,\eta;z) \|^2 w(x) \ud x \\
        &\leq
        c_3 c \int_{M\omega}^\infty \frac{w(x)}{p(x)} \ud x.
    \end{align*}
    It shows that if the Carleman condition is violated, then $\tau$ is in the limit circle case.
    
    Suppose now that $w \in \calD_1^\omega(L^1;\RR)$ and the Carleman's condition is satisfied. We are going to show that 
    \begin{equation}
        \label{eq:157}
        \limsup_{n \to \infty}
        \sup_{z \in K}
        \sup_{\eta,\tilde{\eta} \in \tsS^1} \frac{\int_{M\omega}^{n\omega} |u(x,\eta;z)|^2 w(x) \ud x}{\int_{M\omega}^{n\omega} |u(x,\tilde{\eta};z)|^2 w(x) \ud x} < \infty.
    \end{equation}
    By \eqref{eq:156} and Corollary~\ref{cor:1} we have
    \begin{equation}
        \label{eq:160}
        \int_{M\omega}^{n\omega} |u(x,\eta;z)|^2 w(x) \ud x 
        =
        \sum_{k=M}^{n-1}
        \int_0^\omega |u_k(t,\eta;z)|^2 w_k(t) \ud t
        \leq 
        c \int_{M\omega}^{n\omega} \frac{w(x)}{p(x)} \ud x.
    \end{equation}
    Next, 
    \begin{multline*}
        \int_{M\omega}^{n\omega} |u(x,\tilde{\eta};z)|^2 w(x) \ud x 
        \geq 
        \frac{1}{2} 
        \sum_{k=M}^{n-2} 
        \int_0^\omega 
        \big( |u_k(t,\eta;z)|^2 w_k(t) + |u_{k+1}(t,\eta;z)|^2 w_{k+1}(t) \big) \ud t \\
        \geq 
        \frac{1}{2} 
        \sum_{k=M}^{n-2} 
        \int_0^\omega
        \big( |u_k(t,\eta;z)|^2 + |u_{k+1}(t,\eta;z)|^2 \big) w_k(t) \ud t \\
        -
        \frac{1}{2}
        \sum_{k=M}^{n-2}
        \int_0^\omega |u_{k+1}(t,\eta;z)|^2 |w_{k+1}(t) - w_k(t) | \ud t.
    \end{multline*}
    Since $w \in \calD_1^\omega(L^1;\RR)$, by Corollary~\ref{cor:1}, \eqref{eq:156} and \eqref{eq:3d} we get for some constant $c_4 > 0$,
    \begin{equation}
        \label{eq:159}
        \int_{M\omega}^{n\omega} |u(x,\tilde{\eta};z)|^2 w(x) \ud x  \geq 
        \frac{1}{2 c_3 c} \int_{M\omega}^{(n-1)\omega} \frac{w(x)}{p(x)} \ud x - \frac{c_4 c}{2}.
    \end{equation}
    By the Carleman condition, it implies
    \begin{equation}
        \label{eq:161}
        \int_{M\omega}^{\infty} |u(x,\tilde{\eta};z)|^2 w(x) \ud x = \infty.
    \end{equation}
    In particular, $\tau$ is in the limit point case.
    Therefore, by combining \eqref{eq:160}, \eqref{eq:159} and \eqref{eq:161}, the formula \eqref{eq:157} easily follows. 
    
    Finally, if \eqref{eq:157} holds, then for any strictly increasing sequence $(L_j:j \geq 1)$ to infinity, let $L_j \in [n_j \omega, n_{j+1} \omega]$. Then for sufficiently large $j$ we have
    \begin{align*}
        \frac{\int_{M\omega}^{L_j} |u(x,\eta;z)|^2 w(x) \ud x}{\int_{M\omega}^{L_j} |u(x,\tilde{\eta};z)|^2 w(x) \ud x} 
        &\leq
        \frac{\int_{M\omega}^{n_{j+1}\omega} |u(x,\eta;z)|^2 w(x) \ud x}{\int_{M\omega}^{n_j\omega} |u(x,\tilde{\eta};z)|^2 w(x) \ud x} \\
        &\leq
        \frac{\int_{M\omega}^{n_{j}\omega} |u(x,\eta;z)|^2 w(x) \ud x}{\int_{M\omega}^{n_j\omega} |u(x,\tilde{\eta};z)|^2 w(x) \ud x} + \frac{c \int_{n_j\omega}^{n_{j+1}\omega} \frac{w(x)}{p(x)} \ud x}{\int_{M\omega}^{n_j\omega} |u(x,\tilde{\eta};z)|^2 w(x) \ud x},
    \end{align*}
    which in view of \eqref{eq:157}, \eqref{eq:3a} and \eqref{eq:161} stays bounded. Therefore, we have proved that
    \[
        \limsup_{L \to \infty}
        \sup_{z \in K}
        \sup_{\eta,\tilde{\eta} \in \tsS^1} \frac{\int_{M\omega}^{L} |u(x,\eta;z)|^2 w(x) \ud x}{\int_{M\omega}^{L} |u(x,\tilde{\eta};z)|^2 w(x) \ud x} < \infty,
    \]
    from which our conclusion follows.
\end{proof}

\section{Asymptotic of generalized eigenvectors}
\label{sec:5}
In this section we determine the asymptotic behavior of generalized solutions of \eqref{eq:61} for $\omega$-periodically
modulated Sturm--Liouville parameters.
\begin{theorem}
	\label{thm:2}
	Let $\omega > 0$. Suppose that $(p, q, w)$ are $\omega$-periodically modulated Sturm--Liouville parameters, such that 
	the transfer matrix $\frakT$ corresponding to $(\frakp, \frakq, \frakw)$ satisfies $\abs{\tr \frakT(\omega; 0)} < 2$. 
	Assume that
	\[
		\frac{q}{p}, \frac{w}{p}, \frac{p'}{p} \in \calD_1^\omega(L^1; \RR).
	\]
	Then for each $K$ a compact subset of $\RR$ there is $M \geq 1$ and a non-vanishing function
	$\vphi: [0, \omega] \times \tsS^1 \times K \rightarrow \CC$ such that for every solution $\mathbf{u}$ of \eqref{eq:61},
	\begin{equation}
		\label{eq:36}
		\lim_{n \to \infty} \sup_{(t, \eta, z) \in [0, \omega] \times \tsS^1 \times K}
		\bigg|\frac{u_{n+1}(t, \eta ; z) - \lambda_n^-(t; z) u_n(t, \eta; z)}
		{\prod_{k = M}^{n-1} \lambda_k^+(t, z)}
		-\vphi(t, \eta; z)\bigg| = 0
	\end{equation}
	where $u_n = \sprod{\mathbf{u}_n}{e_1}$. Furthermore,
	\begin{equation}
		\label{eq:39}
		\frac{u_n(t, \eta ; z)}{\prod_{k = M}^{n-1} \lambda_k^+(t; z)}
		=
		\frac{\abs{\vphi(t, \eta; z)}}{\sqrt{4 - \abs{\tr \frakT(\omega; 0)}^2}}
		\sin\Big(\sum_{k = M}^{n-1} \theta_k(t; z) + \arg \vphi(t, \eta; z) \Big)
		+
		E_n(t, \eta; z)
	\end{equation}
	where
	\[
		\theta_k(t; z) = \arccos\bigg(\frac{\tr X_k(t; z)}{2 \sqrt{\det X_k(t; z)}}\bigg),
	\]
	and
	\[
		\lim_{n \to \infty} \sup_{(t, \eta, z) \in [0, \omega] \times \tsS^1 \times K} {\abs{E_n(t, \eta; z)}} = 0.
	\]
\end{theorem}
\begin{proof}
	Let us observe that $\discr \frakT_t(\omega; 0) < 0$ and $[\frakT_t(\omega; 0)]_{12} \neq 0$. Given $K$ a compact subset of 
	$\RR$, by \eqref{eq:99} and Proposition \ref{prop:4}, there are $N_0 \geq 1$ and $\delta > 0$ such that for all 
	$n \geq N_0$, $t \in [0, \omega]$ and $z \in K$,
	\[
		\discr X_n(t; z) \leq - \delta,
		\quad\text{and}\quad
		| [X_n(t; z)]_{12} | \geq \delta.
	\]
	Hence, the matrix $X_n(t; z)$ has two eigenvalues
	\[
		\lambda^+_n(t; z) = \frac{\tr X_n(t; z) + i\sqrt{-\discr X_n(t; z)}}{2},
		\quad\text{and}\quad
		\lambda^-_n(t; z) = \frac{\tr X_n(t; z) - i\sqrt{-\discr X_n(t; z)}}{2}.
	\]
	Furthermore, we have
	\begin{equation}
		\label{eq:95}
		\lim_{n \to \infty} \lambda^+_n(t; z) = \lambda_\infty^+, \qquad 
		\lim_{n \to \infty} \lambda^-_n(t; z) = \lambda_\infty^-
	\end{equation}
	uniformly with respect to $t \in [0, \omega]$ and $z \in K$. In particular, by \eqref{eq:31a} for all $n_1 > n_0$, we have
	\begin{equation}
		\label{eq:30}
		X_{n_1} \cdots X_{n_0} = C_{n_1} \Big(\prod_{k = n_0}^{n_1} D_j C_j^{-1} C_{j-1} \Big) 
		C_{n_0-1}^{-1}.
	\end{equation}
	where $C_k$ and $D_k$ are matrices defined by \eqref{eq:31b}. 

	Given $\epsilon > 0$, there is $L_1 \geq L_0$ be such that for all $L \geq L_1$,
	\[
		\sum_{k = L}^\infty \sup_{[0, \omega] \times K} \big\|\Delta C_{k-1} \big\| \leq \epsilon.
	\]
	For each $L \geq L_1$, and $n \geq L_0$ we set
	\[
		\phi_n = \frac{u_{n+1} - \lambda_n^- u_n}{\prod_{k = L_0}^n \lambda_k^+}.
	\]
	Next, for each $n \geq L$, we define a function on $[0, \omega] \times \tsS^1 \times K$ by the formula
	\[
		\psi_{n; L} = \frac{q_{n+1; L} - \lambda_n^- q_{n; L}}{\prod_{k = L_0}^n \lambda_k^+}
	\]
	where $q_{n; L}$ are functions of $[0, \omega] \times \tsS^1 \times K$ given as
	\begin{equation}
		\label{eq:37}
		q_{n; L} = \left\langle C_\infty \Big(\prod_{k = L}^{n-1} D_j \Big) C_{L-1}^{-1} \mathbf{u}_L, e_1\right\rangle
	\end{equation}
	with
	\[
		C_\infty(t; z) = \lim_{n \to \infty} C_n(t; z) = 
		\begin{pmatrix}
			1 & 1 \\
			\frac{\lambda_\infty^+ - [\frakT_t(\omega; 0)]_{11}}{[\frakT_t(\omega; 0)]_{12}} &
			\frac{\lambda_\infty^- - [\frakT_t(\omega; 0)]_{11}}{[\frakT_t(\omega; 0)]_{12}}
		\end{pmatrix}.
	\]
	The last limit is a consequence of Proposition \ref{prop:4}, \eqref{eq:99} and \eqref{eq:95}. Let us observe that 
	\begin{align*}
		u_n &= \sprod{X_{n-1} \cdots X_{L} \mathbf{u}_{L}}{e_1} \\
		&= \left\langle C_{n-1} \Big(\prod_{k = L}^{n-1} D_j C_j^{-1} C_{j-1} \Big) C_{L-1}^{-1} \mathbf{u}_L,
		e_1\right\rangle.
	\end{align*}
	Next, by the proof of \cite[Proposition 1]{SwiderskiTrojan2019}, for all 
	$(t, \eta, z) \in [0, \omega] \times \tsS^1 \times K$,
	\begin{align*}
		\abs{u_n(t, \eta; z) - q_{n; L}(t, \eta; z)}
		&\leq
		c
		\Big(\prod_{k = L}^{n-1} \|D_k(t; z)\|\Big) 
		\sum_{k = L}^n 
		\big\|\Delta C_{k-1}(t; z) \big\| \|\mathbf{u}_L(t, \eta; z)\| \\
		&\leq
		c \big( \sup_{[0,\omega] \times \tsS^1 \times K} \|\mathbf{u}_{L_0}\|\big)
		\Big(\prod_{k = L_0}^{n-1} \abs{\lambda_k^+(t; z)} \Big)
		\sum_{k = L}^n \big\|\Delta C_{k-1}(t; z) \big\|
	\end{align*}
	where the last estimate is a consequence of Claim \ref{clm:1}. Hence, 
	\begin{align}
		\nonumber
		\big|\phi_n(t, \eta; z) - \psi_{n; L}(t, \eta; z)\big|
		&\leq
		\Big|\frac{u_{n+1}(t, \eta; z) - q_{n+1; L}(t, \eta; z)}{\prod_{k = L_0}^{n} \lambda_k^+(t; z)}\Big|
		+
		\Big|\lambda_{n}^-(t; z) \cdot \frac{u_{n}(t, \eta; z) - q_{n;L}(t, \eta; z)}
		{\prod_{k = L_0}^{n} \lambda_k^+(t; z)}\Big| \\
		\label{eq:32}
		&\leq
		c \sum_{k = L}^n \|\Delta C_{j-1}\| \leq c \epsilon.
	\end{align}
	Therefore, for all $n \geq m \geq L$,
	\[
		\big|\phi_n(t, \eta; z) - \phi_m(t, \eta; z)\big|
		\leq
		2 c \epsilon + \big|\psi_{n; L}(t, \eta; z) - \psi_{m; L}(t, \eta; z)\big|.
	\]
	In particular, to conclude that $(\phi_n)$ is a Cauchy sequence it is enough to show that $(\psi_{n; L} : n \geq L)$
	converges. To achieve this, we first notice that
	\begin{align*}
		D_n - \lambda_n^-\Id = 
		\begin{pmatrix}
			\lambda_n^+ - \lambda_n^- & 0 \\
			0 & 0
		\end{pmatrix},
	\end{align*}
	thus
	\begin{align*}
		\frac{1}{\prod_{k = L_0}^{n-1} \lambda_k^+} 
		\big(D_n - \lambda_n^-\Id\big) \Big(\prod_{k = L_0}^{n-1} D_k\Big)
		=i \frac{ \sqrt{-\discr X_n} }{\prod_{k = L_0}^{L-1} \lambda_k^+} 
		\begin{pmatrix}
			1 & 0 \\
			0 & 0
		\end{pmatrix}.
	\end{align*}
	Therefore,
	\begin{align*}
		\psi_{n; L} 
		&=
		\frac{1}{\prod_{k = L_0}^{n-1} \lambda_k^+}
		\left\langle C_\infty \big(D_n - \lambda_n^-\Id\big) \Big( \prod_{k = L}^{n-1} D_k\Big) C_{L-1}^{-1} 
		\mathbf{u}_{L}, e_1\right\rangle \\
		&=
		i \frac{\sqrt{-\discr X_n}}{\prod_{k = L_0}^{L-1} \lambda_k^+}
		\left\langle C_\infty \begin{pmatrix} 1& 0 \\ 0& 0 \end{pmatrix} C_{L-1}^{-1} \mathbf{u}_{L}, e_1\right\rangle.
	\end{align*}
	Since
	\[
		\begin{pmatrix}
			1 & 0\\
			0 & 0
		\end{pmatrix}
		C_\infty^t
		e_1
		=e_1
	\]
	we get
	\begin{align*}
		\lim_{n \to \infty} \sup_{[0, \omega] \times \tsS^1 \times K} \big|\psi_{n; L} - \psi_{\infty; L}\big| = 0
	\end{align*}
	where
	\[
		\psi_{\infty; L}(t, \eta; z)
		=
		i \frac{\sqrt{-\discr \frakT(\omega; 0)}}{\prod_{k = L_0}^{L-1} \lambda_k^+(t; z)}
		\sprod{C_{L-1}^{-1}(t; z) \mathbf{u}_{L}(t, \eta; z)}{e_1}, 
		\quad t \in [0, \omega], \eta \in \tsS^1, z \in K.
	\]
	We claim that the limit is nontrivial.
	\begin{claim}
		\label{clm:2}
		For all $t \in [0, \omega]$, $\eta \in \tsS^1$, and $z \in K$,
		\[
			\liminf_{L \to \infty}{\abs{\psi_{\infty; L}(t, \eta; z)}} > 0.
		\]
	\end{claim}
	On the contrary, let us suppose that there are $t_0 \in [0, \omega]$, $\eta_0 \in \tsS^1$ and $z_0 \in K$, and a sequence 
	$(L_j : j \in \NN)$ such that
	\[
		\lim_{j \to \infty} L_j = +\infty
	\]
	and
	\[
		\lim_{j \to \infty} \abs{\psi_{\infty; L_j}(t_0, \eta_0; z_0)} = 0.
	\]
	Moreover, we have
	\begin{equation}
		\label{eq:35}
		\sprod{C_{L-1}^{-1} \mathbf{u}_L}{e_1} = 
		i \frac{[X_{L-1}]_{12}}{2 \sqrt{-\discr X_{L-1}}}
		\bigg(
		\frac{\lambda_{L-1}^- - [X_{L-1}]_{11}}{[X_{L-1}]_{12}}
		u_L - u_L'\bigg).
	\end{equation}
	Hence,
	\begin{equation}
		\label{eq:38}
		\psi_{\infty; L}
		=
		-
		\frac{\sqrt{-\discr \frakT(\omega; 0)}}{\prod_{k = L_0}^{L-1} \lambda_k^+}
		\frac{[X_{L-1}]_{12}}{2 \sqrt{-\discr X_{L-1}}}
		\bigg(
		\frac{\lambda_{L-1}^- - [X_{L-1}]_{11}}{[X_{L-1}]_{12}}
		u_L - u_L'\bigg).	
	\end{equation}
	Because $\mathbf{u}_L(t_0, \eta_0; z_0) \in \RR^2$, by taking the imaginary part of \eqref{eq:38}, we conclude that
	\[
		\lim_{j \to \infty} \frac{u_{L_j}(t_0, \eta_0; z_0)}{\prod_{k = L_0}^{L_j-1} \abs{\lambda_k^+(t_0; z_0)}} = 0,
	\]
	which by \eqref{eq:38} allows to deduce that
	\begin{equation}
		\label{eq:34}
		\lim_{j \to \infty} \frac{\big\| \mathbf{u}_{L_j}(t_0, \eta_0; z_0)\big\|}
		{\prod_{k = L_0}^{L_j-1} \abs{\lambda_k^+(t_0; z_0)}} = 0.
	\end{equation}
	On the other hand, by Theorem \ref{thm:1}, there are $c_1, c_2 > 0$ and $J > 0$ such that for all $j \geq J$, 
	\[
		p_{L_j+1}(t) \scrD_{L_j}(t, \eta; z) \geq c_1,
	\]
	and
	\[
		\scrD_{L_j}(t; z) 
		\leq c_2 \big\|\mathbf{u}_{L_j}(t, \eta; z)\big\|^2.
	\]
	Hence,
	\begin{align*}
		\big\|\mathbf{u}_{L_j}(t, \eta; z)\big\|^2 
		\geq \frac{c_1}{c_2} \frac{1}{p_{L_j+1} (t)}.
	\end{align*}
	Now, by Claim \ref{clm:3} we get
	\[
		\frac{\big\| \mathbf{u}_{L_j}(t, \eta; z)\big\|}{\prod_{k = L_0}^{L_j-1} \abs{\lambda_k^+(t; z)}} 
		\geq
		\sqrt{\frac{c_1}{c_2} \frac{1}{p_{L_0}(t)}}
	\]
	which contradicts \eqref{eq:34}, proving the claim.

	Since $(\phi_n)$ converges, by \eqref{eq:32}, its limit $\phi_\infty$ satisfies
	\[
		\sup_{[0, \omega] \times \tsS^1 \times K} \big|\phi_\infty - \psi_{\infty; L}\big| \leq c \epsilon
	\]
	for all $L \geq L_0$. Consequently, taking 
	\[
		\epsilon = \frac{1}{2c} \liminf_{L \to \infty}{\abs{\psi_{\infty; L}(t, \eta; z)}}
	\]
	we get
	\[
		\abs{\phi_{\infty}(t, \eta; z)} \geq \abs{\psi_{\infty; L}(t, \eta; z)} - c\epsilon 
		\geq \frac{1}{2} \liminf_{L \to \infty}{\abs{\psi_{\infty; L}(t, \eta; z)}}.
	\]
	Hence, $\phi_\infty \neq 0$ on $[0, \infty] \times \tsS^1 \times K$, proving \eqref{eq:36} with
	\[
		\vphi(t, \eta; z) = \phi_\infty(t, \eta; z).
	\]
	Now, using \eqref{eq:36}, we can write
	\[
		\lim_{n \to \infty}
		\sup_{[0, \omega] \times \tsS^1 \times K}
		\bigg|
		\frac{u_n - \lambda_n^- u_{n-1}}{\prod_{k=L_0}^{n-1} \abs{\lambda_k^+}} - 
		\vphi \prod_{k = L_0}^{n-1} \frac{\lambda_k^+}{\overline{\lambda_k^+}}
		\bigg| = 0.
	\]
	Since $u_n(t, \eta ; z) \in \RR$, by taking imaginary part we get
	\[
		\lim_{n \to \infty}
		\sup_{[0, \omega] \times \sS^1 \times K}
		\bigg|
		\frac{1}{2} \sqrt{-\discr X_n} \frac{u_n}{\prod_{k = L_0}^{n-1} \abs{\lambda_k^+}}
		-
		\abs{\vphi} \sin\Big(\sum_{k = L_0}^{n-1} \theta_k + \arg \vphi \Big)
		\bigg| = 0
	\]
	where we have used that
	\[
		\Im(\lambda_n^+(t; z)) = \frac{1}{2} \sqrt{-\discr X_n(t; z)}.
	\]
	Lastly, we observe that
	\[
		\bigg|\frac{1}{\sqrt{-\discr X_n(t; z)}} - \frac{1}{\sqrt{-\discr \frakT(\omega; 0)}} \bigg|
		\leq
		c\sum_{k = n}^\infty \sup_{[0, \omega] \times K} \big\|\Delta X_n \big\|,
	\]
	and the proof is completed.
\end{proof}

\begin{remark}
    \label{rem:2}
	Under the hypotheses of Theorem \ref{thm:2} we also obtain the asymptotic behavior of 
	$(\partial_t u_n(t, \eta; z) : n \in \NN)$ in terms of the function
	\[
		\frac{\lambda_\infty^+ - [\frakT_t(\omega; 0)]_{11}}{[\frakT_t(\omega; 0)]_{12}}
		\vphi(t, \eta ; z), \quad t \in [0, \omega], \eta \in \tsS^1, z \in K.
	\]
	Indeed, in the proof of Theorem \ref{thm:2} one needs to replace $e_1$ by $e_2$. Since
	\[
		\psi_{\infty; L}
		=
		i
		\frac{\sqrt{-\discr \frakT(\omega; 0)}}{\prod_{k = L_0}^{L-1} \lambda_k^+}
		\left\langle
		C_\infty \begin{pmatrix} 1 & 0 \\ 0 & 0 \end{pmatrix} C_{L-1}^{-1} \mathbf{u}_L, e_2
		\right\rangle
	\]
	is the limit of $(\psi_{n; L} : n \in \NN)$, the only place which needs a separate argument is Claim \ref{clm:2}.
	Because
	\[
		\begin{pmatrix} 1 & 0 \\ 0 & 0 \end{pmatrix} C_\infty^t e_2
		=
		\frac{\lambda_\infty^+ - [\frakT_t(\omega; 0)]_{11}}{[\frakT_t(\omega; 0)]_{12}} e_1
	\]
	by \eqref{eq:35} we obtain
	\begin{align*}
		\psi_{\infty; L}
		&=
		i \frac{\sqrt{-\discr \frakT(\omega; 0)}}{\prod_{k = L_0}^{L-1} \lambda_k^+} \cdot
		\frac{\lambda_\infty^+ - [\frakT_t(\omega; 0)]_{11}}{[\frakT_t(\omega; 0)]_{12}}
		\sprod{C_{L-1}^{-1} \mathbf{u}_{L-1}}{e_1}\\
		&=
		- \frac{\sqrt{-\discr \frakT(\omega; 0)}}{\prod_{k = L_0}^{L-1} \lambda_k^+} \cdot
		\frac{\lambda_\infty^+- [\frakT_t(\omega; 0)]_{11}}{[\frakT_t(\omega; 0)]_{12}}
		\frac{[X_{L-1}]_{12}}{2 \sqrt{-\discr X_{L-1}}}
		\bigg(
		\frac{\lambda_{L-1}^- - [X_{L-1}]_{11}}{[X_{L-1}]_{12}}
		u_L - u_L'\bigg).
	\end{align*}
	Now, by the same line of reasoning we can prove Claim \ref{clm:2} also in this case.
\end{remark}

In Section \ref{sec:1}, we also need asymptotic behavior of solutions to \eqref{eq:10} on the complex plane.
\begin{theorem}
	\label{thm:4}
	Let $\omega > 0$. Suppose that $(p, q, w)$ are $\omega$-periodically modulated Sturm--Liouville parameters, such that
	the transfer matrix $\frakT$ corresponding to $(\frakp, \frakq, \frakw)$ satisfies $\abs{\tr \frakT(\omega; 0)} < 2$. 
	Assume that
	\[
		\frac{q}{p}, \frac{w}{p}, \frac{p'}{p} \in \calD_1^\omega(L^1; \RR).
	\]
	Let $K$ be a compact subset of $\CC_\varsigma$ with non-empty interior where 
	$\varsigma = \sign{\tr \partial_z \frakT(\omega; 0)}$. Suppose that there is $s \in [0, \omega]$, such that
	\begin{equation}
		\label{eq:118}
		\lim_{n \to \infty} 
		\int_0^\omega \bigg|\frac{1}{\gamma_n(s)} \frac{w_n(s+t)}{p_n(s+t)} 
		- \frac{1}{\gamma} \frac{\frakw(s+t)}{\frakp(s+t)} \bigg| {\: \rm d} t = 0.
	\end{equation}
	If the Carleman's condition \eqref{eq:26} is satisfied then for each $z \in K$, there are two linearly independent mappings $\mathbf{u}^+(\cdot;z)$ and $\mathbf{u}^-(\cdot;z)$ solving \eqref{eq:10}. Moreover, for any $x \geq 0$ the mappings $\mathbf{u}^+(x;\cdot)$ and $\mathbf{u}^-(x;\cdot)$
	are continuous on $K$ and holomorphic in $\intr{K}$. Furthermore, there is $M \geq 1$ such that the limits
	\[
		\phi^+(t;z) = \lim_{k \to \infty} \frac{u_k^+(t;z)}{\prod_{M \leq j < k} \lambda_j^+(s; z)},
		\quad\text{and}\quad
		\phi^-(t;z) = \lim_{k \to \infty} \frac{u_k^-(t;z)}{\prod_{M \leq j < k} \lambda_j^-(s; z)}
	\]
	exist uniformly with respect to $(t,z) \in [0,\omega] \times K$ and define non-vanishing functions.
\end{theorem}
\begin{proof}
	Let $M \geq 1$, be determined in Corollary \ref{cor:2}. In view of Section~\ref{sec:2.2} and Proposition~\ref{prop:1}, 
	we can use \cite[Theorem 6.1]{SwiderskiTrojan2023} to get sequences of maps 
	$(\mathbf{u}^-_n(s;\cdot) : n \geq M)$ and $(\mathbf{u}^+_n(s;\cdot) : n \geq M)$, solving \eqref{eq:14}, continuous on $K$ and holomorphic in $\intr{K}$, and satisfying
	\begin{equation}
		\label{eq:117}
		\lim_{k \to \infty} 
		\bigg\|
		\frac{\mathbf{u}^+_n(s;z)}{\prod_{M \leq j < k} \lambda_j^+(s; z)} - C_\infty e_1
		\bigg\| = 0,
		\quad\text{and}\quad
		\lim_{k \to \infty} 
		\bigg\|
		\frac{\mathbf{u}^-_n(s;z)}{\prod_{M \leq j < k} \lambda_j^-(s; z)} - C_\infty e_2
		\bigg\| = 0,
	\end{equation}
	uniformly with respect to $z \in K$, where
	\[
		C_\infty = \lim_{n \to \infty} C_n(s; z) =
		\begin{pmatrix}
			1 & 1 \\
			\frac{\lambda_\infty^+ - [\frakT_s(\omega; 0)]_{11}}{[\frakT_s(\omega; 0)]_{12}} &
			\frac{\lambda_\infty^- - [\frakT_s(\omega; 0)]_{11}}{[\frakT_s(\omega; 0)]_{12}}
		\end{pmatrix}.
	\]
	Moreover, for each $z \in K$, $(\mathbf{u}^-_n(s;z) : n \geq M)$ and $(\mathbf{u}^+_n(s;z) : n \geq M)$ are linearly independent.
    For any $x \geq 0$ define
    \begin{equation}
        \label{eq:170}
        \mathbf{u}^+(x;z) = T(x;z) T(s+M\omega;z)^{-1} \mathbf{u}^+_M(s;z), \quad \text{and} \quad
        \mathbf{u}^-(x;z) = T(x;z) T(s+M\omega;z)^{-1} \mathbf{u}^-_M(s;z)
    \end{equation}
    Notice that by \eqref{eq:6} functions $\mathbf{u}^+(\cdot;z)$ and $\mathbf{u}^-(\cdot;z)$ satisfy \eqref{eq:10}. Since $T(x;\cdot)$ is invertible and entire for any $x \in [0,\infty)$, we get that $\mathbf{u}^+(x;\cdot)$ and $\mathbf{u}^-(x;\cdot)$ are continuous on $K$ and holomorphic in $\intr{K}$. Moreover, using \eqref{eq:14}, \eqref{eq:99} and \eqref{eq:41} we easily get
    \[
        \mathbf{u}^+(s+n\omega;z) = \mathbf{u}_n^+(s;z), \quad \text{and} \quad
        \mathbf{u}^-(s+n\omega;z) = \mathbf{u}_n^-(s;z), \quad n \geq M.
    \]
    Therefore, we shall define
    \begin{equation}
        \label{eq:147}
        \mathbf{u}_n^+(t;z) = \mathbf{u}^+(t+n\omega;z), \quad \text{and} \quad
        \mathbf{u}_n^-(t;z) = \mathbf{u}^-(t+n\omega;z), \quad t \in [0,\infty), n \geq 0.
    \end{equation}
    In view of \eqref{eq:23b},
	\begin{equation}
            \label{eq:91}
		\Wrk \big( \mathbf{u}^+(x; z), \mathbf{u}^-(x;z) \big) 
		\equiv
		p_M(s)
		\det
		\big( \mathbf{u}^+_M(s;z), \mathbf{u}^-_M(s;z) \big), \quad x \geq 0.
	\end{equation}
    By the linear independence of $(\mathbf{u}^-_n(s;z) : n \geq M)$ and $(\mathbf{u}^+_n(s;z) : n \geq M)$, the right-hand side of \eqref{eq:91} is non-zero (see, e.g. \cite[Corollary 3.12]{Elaydi2005}). Therefore, the mappings $\mathbf{u}^+(\cdot; z)$ and $\mathbf{u}^-(\cdot;z)$ are linearly independent.

    Let $t \in [0,\omega]$. Then by \eqref{eq:147}, \eqref{eq:41} and \eqref{eq:6} we have
    \[
        \mathbf{u}_n^+(s+t;z) = U_{s;n}(t;z) \mathbf{u}_n^+(s+t;z), \quad
        \mathbf{u}_n^-(s+t;z) = U_{s;n}(t;z) \mathbf{u}_n^-(s+t;z), \quad
        t \in [0,\omega], n \geq 0.
    \]
    By \eqref{eq:117} and Proposition~\ref{prop:4} we get
	\begin{equation}
		\label{eq:117a}
		\lim_{k \to \infty} 
		\bigg\|
		\frac{\mathbf{u}^+_k(t+s;z)}{\prod_{M \leq j < k} \lambda_j^+(s; z)} - \frakT_s(t;z) C_\infty e_1
		\bigg\| = 0
    \end{equation}
    and
    \begin{equation}
        \label{eq:117b}
		\lim_{k \to \infty} 
		\bigg\|
		\frac{\mathbf{u}^-_k(t+s;z)}{\prod_{M \leq j < k} \lambda_j^-(s; z)} - \frakT_s(t;z) C_\infty e_2
		\bigg\| = 0,
	\end{equation}
    uniformly with respect to $(t,z) \in [0,\omega] \times K$. For all $t \in [0,\omega]$, and $z \in K$ and $k \geq M$, we set
	\[
		u_k^+(t;z) = \sprod{\mathbf{u}^+_k(t;z)}{e_1}, 
		\quad\text{and}\quad
		u_k^-(t;z) = \sprod{\mathbf{u}^-_k(t;z)}{e_1},
	\]
	and
	\begin{equation}
            \label{eq:171}
		\phi^+(t;z) = \sprod{\frakT_s(t;z) C_\infty e_1}{e_1}, 
		\quad\text{and}\quad
		\phi^-(t;z) = \sprod{\frakT_s(t;z) C_\infty e_2}{e_1}.
	\end{equation}
	Since the matrix $\frakT_s(t; 0)$ is real and $\lambda_\infty^+, \lambda_\infty^- \in \CC \setminus \RR$ we easily get for all $i, j \in \{1, 2\}$, 
	\[
		\sprod{\frakT_s(t;z) C_\infty e_i}{e_j} \neq 0.
	\]
	In particular, for all $z \in K$, and $t \in [0,\omega]$, $\phi^+(t;z) \neq 0$ and $\phi^-(t;z) \neq 0$, and the theorem follows.
\end{proof}

\section{Diagonal behavior of Christoffel--Darboux kernels}
\label{sec:6}
Given $\eta \in \tsS^1$, we denote by $\mathbf{u}$ the solution to \eqref{eq:61}. In this section we study
diagonal behavior of Christoffel--Darboux kernels which are defined by
\[
    K_L(z_1, z_2; \eta) = 
    \int_0^L u(t, \eta; z_1) \overline{u(t, \eta; z_2)} w(t) {\: \rm d} t, \quad z_1, z_2 \in \CC,
\]
where $u = \sprod{\mathbf{u}}{e_1}$.

For each $n \in \NN$ and $s \in [0, \omega]$, we set
\begin{equation}
	\label{eq:67}
	\rho_n(s) = \sum_{j = 0}^n \gamma_j(s). 
\end{equation}
Let us observe that
\begin{equation}
    \label{eq:134}
	\rho_n(s) = \int_s^{s+n\omega} \frac{w(t)}{p(t)} \ud t = \int_0^\omega \varrho_n(s+t) \ud t
\end{equation}
where
\[
	\varrho_n(t) = \sum_{j = 0}^n \frac{w_j(t)}{p_j(t)}.
\]

\begin{theorem}
	\label{thm:3}
	Let $\omega > 0$. Suppose that $(p, q, w)$ are $\omega$-periodically modulated Sturm--Liouville parameters, such that 
	the transfer matrix $\frakT$ corresponding to $(\frakp, \frakq, \frakw)$ satisfies $\abs{\tr \frakT(\omega; 0)} < 2$.
	Assume that
	\begin{equation}
		\label{eq:43}
		\frac{q}{p}, \frac{w}{p}, \frac{p'}{p} \in \calD_1^\omega(L^1; \RR).
	\end{equation}
	If for almost all $t \in [0, \omega]$,
	\begin{equation}
		\label{eq:42}
		\lim_{n \to \infty} \varrho_n(t) = \infty,
	\end{equation}
	and
	\begin{equation}
		\label{eq:103}
		\lim_{n \to \infty} \frac{w_{n+1}(t)}{w_n(t)} = 1,
	\end{equation}
	then for each solution $\mathbf{u}$ of \eqref{eq:61} there is $M \geq 1$, such that almost all $t \in [0, \omega]$,
	\[
		\lim_{n \to \infty} 
		\frac{1}{\varrho_n(t)}
		\sum_{m = M}^n \abs{\sprod{\mathbf{u}_m(t, \eta; z)}{e_1}}^2 w_m(t)
		=
		\frac{\abs{\vphi(t, \eta; z)}^2 p_{M}(t)}{4 - \abs{\tr \frakT(\omega; 0)}^2}
	\]
	locally uniformly with respect to $(\eta, z) \in \tsS^1 \times \RR$.
\end{theorem}
\begin{proof}
	Let $K$ be a compact subset of $\RR$. We set $u_n = \sprod{\mathbf{u}_n}{e_1}$. By Theorem \ref{thm:2}, there are 
	$M \geq 1$, such that for all $n \geq M$, $(t, \eta, z) \in [0, \omega] \times \tsS^1 \times K$,
	\begin{align*}
		\frac{\abs{u_n(t, \eta; z)}^2}{\prod_{k = M}^{n-1} \abs{\lambda_k^+(t; z)}^2}
		=
		\frac{\abs{\vphi(t, \eta; z)}^2}{4 - \abs{\tr \frakT(\omega; 0))}^2} 
		\sin^2\Big(\sum_{k = M}^{n-1} \theta_k(t; z) + \arg \vphi(t, \eta; z) \Big) + E_n(t, \eta; z)
	\end{align*}
	where
	\[
		\lim_{n \to \infty} \sup_{(t, \eta, z) \in [0, \omega] \times \tsS^1 \times K} \abs{E_n(t, \eta; z)} =0.
	\]
	By Claim \ref{clm:3} we obtain
	\begin{align*}
		\sum_{m = M}^n
		\abs{u_m(t, \eta; z)}^2 w_m(t)
		&=
		\frac{\abs{\vphi(t, \eta; z)}^2 p_{M}(t)}{4 - \abs{\tr \frakT(\omega; 0))}^2}
		\sum_{m = M}^n \frac{w_m(t)}{p_m(t)} 
		\Big\{1 - \cos\Big(2 \sum_{k = M}^{m-1} \theta_k(t; z) + 2 \arg \vphi(t, \eta; z) \Big)\Big\} \\
		&\phantom{=}+
		\sum_{m = M}^n \frac{w_m(t)}{p_m(t)} E_m(t, \eta; z).
	\end{align*}
	We are going to show that for almost all $t \in [0, \omega]$,
	\begin{equation}	
		\label{eq:40}
		\lim_{n \to \infty}
		\frac{1}{\varrho(t)}
		\sum_{m = 0}^{n-1}
		\abs{u_m(t, \eta; z)}^2 w_m(t)
		=
		\frac{\abs{\vphi(t, \eta; z)}^2 p_M(t)}{4 - \abs{\tr \frakT(\omega; 0))}^2}
	\end{equation}
	uniformly with respect to $(\eta, z) \in \tsS^1 \times K$. For the proof, let us observe that by the Stolz--Ces\`aro theorem,
	\[
		\lim_{n \to \infty} \frac{1}{\varrho_n(t)} \sum_{m=M}^n \frac{w_m(t)}{p_m(t)} E_m(t, \eta; z)
		=
		\lim_{n \to \infty} E_n(t, \eta; z) = 0,
	\]
	for almost all $t \in [0, \omega]$ and uniformly with respect to $(\eta, z) \in \tsS^1 \times K$. 
	Next, by Corollary~\ref{cor:1}, we have
	\[
		\lim_{n \to \infty}
		\sup_{(\eta, z) \in \tsS^1 \times K}{
		\frac{1}{\varrho_n(t)}
		\sum_{m = 0}^{M-1} \abs{u_m(t, \eta; z)}^2} w_m(t) {\: \rm d } t \\
		\leq
		\lim_{n \to \infty}
		c \frac{1}{\varrho_n(t)} \sum_{m = 0}^{M-1} \frac{w_m(t)}{p_m(t)} = 0, 
	\]
	thus to complete the proof of \eqref{eq:40} it is enough to show that
	\begin{align*}
		\lim_{n \to \infty} 
		\frac{1}{\varrho_n(t)} \sum_{m = M}^n
		\frac{w_m(t)}{p_m(t)}
		\cos\Big(2 \sum_{k = M}^{m-1} \theta_k(t; z) + 2 \varphi(t, \eta; z)\Big)
		=0
	\end{align*}
	uniformly with respect to $(\eta, z) \in \tsS^1 \times K$. First, by Lemma \ref{lem:3}, and \eqref{eq:103}, for almost all 
	$t \in [0, \omega]$,
	\[
		\lim_{n \to \infty}
		\bigg| \frac{w_{n+1}(t)}{p_{n+1}(t)} \frac{p_n(t)}{w_n(t)} - 1 \bigg| = 0.
	\]
	Hence, by \cite[Lemma 4.1]{ChristoffelI}, for almost all $t \in [0, \omega]$,
	\[
		\lim_{n \to \infty}
		\frac{1}{\varrho_n(t)} \sum_{m = M}^n
		\frac{w_m(t)}{p_m(t)}
		\cos\Big(2 \sum_{k = M}^{m-1} \theta_k(t; z) + 2 \varphi(t, \eta; z)\Big) = 0,
	\]
	uniformly with respect to $(\eta, z) \in \tsS^1 \times K$, which completes the proof.
\end{proof}

For $L > 0$, we set
\begin{equation}
    \label{eq:135}
	\rho_L = \int_0^L \frac{w(t)}{p(t)} \ud t.
\end{equation}
\begin{corollary} 
	\label{cor:4}
	Suppose that the hypotheses of Theorem \ref{thm:3} are satisfied. Suppose that there is $s \in [0, \omega]$ such that
	\begin{equation}
		\label{eq:104}
		\lim_{n \to \infty} \int_0^\omega \bigg| \frac{1}{\gamma_n(s)} \frac{w_n(s+t)}{p_n(s+t)} 
		- \frac{1}{\gamma} \frac{\frakw(s+t)}{\frakp(s+t)} \bigg| \ud t.
	\end{equation}
	Then 
	\begin{equation}
            \label{eq:130}
		\lim_{L \to \infty}
		\frac{1}{\rho_L} K_L(z, z; \eta)
		=
		\int_0^\omega \frac{\abs{\vphi(t, \eta; z)}^2 p_M(t)}{4 - |\tr \frakT(\omega; 0)|^2}
		\frac{\frakw(t)}{\frakp(t)} \ud t,
        \end{equation}
	locally uniformly with respect to $z \in \RR$ and $\eta \in \tsS^1$. In particular, $\tau$ is in the limit-point case at 
	$\infty$.
\end{corollary}
\begin{proof}
    We shall start by proving the convergence 
    \begin{equation}
        \label{eq:131}
	\lim_{n \to \infty}
	\frac{1}{\rho_n(s)} K_{s + n\omega}(z, z; \eta)
	=
	\int_0^\omega 
	\frac{\abs{\vphi(t, \eta; z)}^2 p_M(t)}{4 - |\tr \frakT(\omega; 0)|^2} 
	\frac{\frakw(t)}{\frakp(t)}
	{\: \rm d}t
    \end{equation}
    locally uniformly with respect to $(\eta, z) \in \tsS^1 \times \RR$.
    
    We apply Theorem \ref{thm:3} on $[0, \omega - s]$ and $[\omega-s, \omega]$, together with the Lebesgue's dominated theorem, 
	to get
	\begin{equation}
		\label{eq:68}
		\begin{aligned}
		&
		\lim_{n \to \infty}
		\int_0^\omega \frac{1}{\gamma} \frac{\frakw(s+t)}{\frakp(s+t)}
		\frac{1}{\varrho_n(s+t)} 
		\sum_{m = M}^n \abs{\sprod{\mathbf{u}_m(s+t, \eta; z)}{e_1}}^2 w_m(s+t)
		\ud t \\
		&\qquad=
		\int_0^\omega
		\frac{\abs{\vphi(t, \eta; z)}^2 p_{M}(t)}{4 - \abs{\tr \frakT(\omega; 0)}^2}
		\frac{\frakw(t)}{\frakp(t)}
		{\: \rm d} t
		\end{aligned}
	\end{equation}
	locally uniformly with respect to $(\eta, z) \in \tsS^1 \times \RR$. Let us next observe that by Lemma \ref{lem:3} and 
	\eqref{eq:104}
	\begin{align*}
		\int_0^\omega \bigg|\frac{1}{\gamma} \frac{\frakw(s+t)}{\frakp(s+t)} - \frac{\varrho_n(t+s)}{\rho_n(s)}\bigg| \ud t
		\leq
		\frac{1}{\rho_n(s)} \sum_{j = 0}^n \gamma_j(s)
		\int_0^\omega \bigg| \frac{1}{\gamma} \frac{\frakw(s+t)}{\frakp(s+t)} - \frac{1}{\gamma_j(s)}
		\frac{w_j(s+t)}{p_j(s+t)} \bigg| \ud t,
	\end{align*}
	thus by the Stolz--Ces\`aro theorem we obtain
	\[
		\lim_{n \to \infty}
		\int_0^\omega \bigg|\frac{1}{\gamma} \frac{\frakw(s+t)}{\frakp(s+t)} - \frac{\varrho_n(t+s)}{\rho_n(s)}\bigg| \ud t
		=0.
	\]
	Next, we write
	\begin{align*}
		&\bigg|
		\int_0^\omega 
		\frac{1}{\gamma} \frac{\frakw(s+t)}{\frakp(s+t)}
		\frac{1}{\varrho_n(s+t)} \sum_{k = 0}^n
		\abs{u_k(s+t, \eta; z)}^2 w_k(s+t)
		{\: \rm d} t \\
		&\phantom{\int_0^\omega}
		\qquad-
		\int_0^\omega \frac{1}{\rho_n(s)} \sum_{k = 0}^n
		\abs{u_k(s+t, \eta; z)}^2 w_k(s+t) {\: \rm d} t
		\bigg|\\
		&\qquad\leq
		\int_0^\omega
		\bigg|
		\frac{1}{\gamma} \frac{\frakw(s+t)}{\frakp(s+t)} 
		-
		\frac{\varrho_n(s+t)}{\rho_n(s)}
		\bigg|
		\cdot
		\frac{1}{\varrho_n(s+t)}
		\bigg(
		\sum_{k = 0}^n \abs{u_k(s+t, \eta; z)}^2 w_k(t)
		\bigg)
		{\: \rm d} t \\
		&\qquad\leq
		c
		\int_0^\omega
		\bigg|
		\frac{1}{\gamma} \frac{\frakw(s+t)}{\frakp(s+t)} - \frac{\varrho_n(s+t)}{\rho_n(s)}
		\bigg|
		{\: \rm d} t.
	\end{align*}
	Therefore, by \eqref{eq:68},
	\[
		\lim_{n \to \infty}
		\frac{1}{\rho_n(s)}
		\int_0^\omega
		\sum_{k=0}^n
		\abs{u_k(s + t, \eta; z)}^2 w_k(s + t)
		{\: \rm d} t
		=
		\int_0^\omega
		\frac{\abs{\vphi(t, \eta; z)}^2 p_{M}(t)}{4 - |\tr \frakT(\omega; 0)|^2} 
		\frac{\frakw(t)}{\frakp(t)}
		{\: \rm d} t
	\]
	uniformly with respect to $(\eta, z) \in \tsS^1 \times K$. Since
	\begin{align*}
		K_{s+n\omega}(z, z; \eta) = 
		\int_0^{s + n\omega} \abs{u(t, \eta; z)}^2 w(t) {\: \rm d} t
		&=
		\sum_{k=0}^{n-1} \int_0^\omega \abs{u_k(s + t, \eta; z)}^2 w_k(t) \ud t  \\
		&\phantom{=}- \int_0^s \abs{u(t, \eta; z)}^2 w(t) \ud t
	\end{align*}
	we obtain \eqref{eq:131}
	uniformly with respect to $(\eta, z) \in \tsS^1 \times K$. Notice that, in view of \eqref{eq:104}, the function $u(\cdot, \eta; z)$ does not belong to $L^2(w)$, which implies that $\tau$ is in the limit-point case at $\infty$.
	
    Let us turn to the proof of \eqref{eq:130}. Let $(L_j : j \geq 1)$ be a strictly increasing sequence of positive integers. Then there is a sequence of positive integers $(n_j : j \geq 1)$ such that $L_j \in [n_j \omega, (n_{j}+1) \omega)$. We obviously have $\lim_{j \to \infty} n_j = \infty$. Notice that by Corollary~\ref{cor:1} there is a constant $c>0$ such that
    \begin{align}
        \nonumber
        \sup_{(\eta,z) \in \tsS^1 \times K} \big| K_{L_j}(z,z;\eta) - K_{s+n_j \omega}(z,z;\eta) \big| 
        &\leq
        \int_{n_j \omega}^{(n_j + 1) \omega} |u(t,\eta;z)|^2 w(t) \ud t \\
        \label{eq:132}
        &\leq
        c \int_{n_j \omega}^{(n_j + 1) \omega} \frac{w(t)}{p(t)} \ud t.
    \end{align}
    Next, by \eqref{eq:134} and \eqref{eq:135}
    \begin{equation}
        \label{eq:133}
        |\rho_{L_j} - \rho_{n_j}(s)| \leq 
        \int_{0}^s \frac{w(t)}{p(t)} \ud t +
        \int_{n_j \omega}^{(n_j + 1) \omega} \frac{w(t)}{p(t)} \ud t
    \end{equation}
    Since
    \begin{align*}
        \bigg|
        \frac{1}{\rho_{L_j}} K_{L_j}(z,z;\eta) - 
        \frac{1}{\rho_{n_j}(s)} K_{s+n_j \omega}(z,z;\eta) 
        \bigg| 
        &\leq
        \frac{1}{\rho_{L_j}}
        \big| K_{L_j}(z,z;\eta) - K_{s+n_j \omega}(z,z;\eta) \big| \\
        &+
        \frac{|\rho_{L_j} - \rho_{n_j}(s)|}{\rho_{L_j}} \frac{1}{\rho_{n_j}(s)} K_{s+n_j \omega}(z,z;\eta)
    \end{align*}
    by \eqref{eq:132}, \eqref{eq:133}, \eqref{eq:131}, \eqref{eq:3a} and the Carleman condition, the conclusion \eqref{eq:130} follows.
\end{proof}

\section{The density of states}
\label{sec:1}
Given $\eta \in \tsS^1$ and $L > 0$ let us consider an operator $H^L_{\eta} :\Dom \big( H^L_{\eta} \big) \to L^2([0,L], w)$ 
defined by $H^L_{\eta} f = \tau f$ with the domain
\begin{align}
    \label{eq:129}
    \Dom \big( H^L_{\eta} \big) = 
    \big\{ 
        f \in L^2([0,L], w) \colon
        &f, pf' \in \AC([0,L]), \tau f \in L^2([0,L],w), \\
    \nonumber
        &p(0) \eta_2 f(0) - \eta_1 pf'(0) = 0, f(L) = 0 
    \big\}.
\end{align}
Notice that according to \cite[Example, p.225]{Teschl2014} the operator $H_\eta^L$ is self-adjoint with the spectrum consisting
of simple eigenvalues only. In this section we shall be interested in asymptotic behavior of the family of the eigenvalue 
counting measures
\[
	\nu_{\eta}^L = 
	\sum_{\lambda \in \sigma(H^L_\eta)} \delta_\lambda, \quad L \in (0, \infty)
\]
where $\delta_\lambda$ denotes the Dirac's delta at $\lambda$. 

Let us begin with an adaptation of \cite[Lemma 3.1]{SwiderskiTrojan2023}, which allows to link asymptotic behavior of Christoffel--Darboux kernel on the diagonal with the vague limit of $(\nu^L_\eta : L > 0)$.
\begin{lemma}
    \label{lem:1}
    Let $\eta \in \tsS^1$ be given.
    Suppose that there are an open subset $U$ of the real line, a continuous function $\rho_\cdot : (0,\infty) \to (0,\infty)$ such that $\lim_{L \to \infty} \rho_L = \infty$ and a non-zero function $g : U \to \RR$ such that
    \begin{equation}
        \label{eq:138}
        \lim_{L \to \infty} \frac{1}{\rho_L} K_L(\lambda, \lambda; \eta) = g(\lambda)
    \end{equation}
    locally uniformly with respect to $\lambda \in U$. Then for each $f \in \calC_c(U)$
    \begin{equation}
        \label{eq:137}
        \lim_{L \to \infty} \frac{1}{\rho_L} \int_\RR f(\lambda) \nu^L_\eta(\ud \lambda) = 
        \int_\RR f(\lambda) g(\lambda) \mu_\eta(\ud \lambda).
    \end{equation}
\end{lemma}
\begin{proof}
    Since $g(\lambda_0) \neq 0$ for certain $\lambda_0 \in U$, we have that $u(\cdot,\eta;\lambda_0)$ does not belong to $L^2(w)$, so $\tau$ is in the limit-point case at $+\infty$.
    
    Let us define a collection of positive measures
	\[
		\sigma_\eta^{L} =
		\sum_{\lambda \in \sigma(H^L_\eta)} \frac{\delta_\lambda}{K_L(\lambda, \lambda; \eta)}, \quad L \in (0, \infty).
	\]
	By \cite[Theorem 10.1]{Weidmann2005} (see also \cite[formula (3.13)]{Remling2018} for a stronger statement) we have
	\begin{equation}
		\label{eq:87}
		\lim_{L \to \infty} 
		\int_\RR h(\lambda) \: \sigma_\eta^{L}({\rm d}\lambda) =
		\int_\RR h(\lambda) \: \mu_\eta({\rm d} \lambda), \quad h \in \calC_c(\RR).
	\end{equation}
	Notice that
	\begin{equation}
		\label{eq:86}
		\sigma_\eta^{L}({\rm d} \lambda) = \frac{1}{K_L(\lambda,\lambda;\eta)} \nu_\eta^L({\rm d} \lambda).
	\end{equation}
	Let $f \in \calC_c(U)$ and let $K = \supp(f)$. By \eqref{eq:86} we have
	\begin{align*}
		\bigg|
		\frac{1}{\rho_L} \int_\RR f(\lambda) \: \nu_\eta^L({\rm d} \lambda)
		-
		\int_\RR f(\lambda) g(\lambda) \: \sigma_\eta^L({\rm d} \lambda)
		\bigg|
		&=
		\bigg| 
		\int_\RR f(\lambda) \Big( \frac{1}{\rho_L} K_L(\lambda,\lambda;\eta) 
		- g(\lambda) \Big) \: \sigma_\eta^L({\rm d} \lambda) \bigg| \\
		&\leq
		\sup_{\lambda \in K} |f(\lambda)| \cdot 
		\sup_{\lambda \in K} 
		\bigg| \frac{1}{\rho_L} K_L(\lambda,\lambda;\eta) - g(\lambda) \bigg| 
		\cdot \sigma^L_\eta(K).
	\end{align*}
	By \eqref{eq:87} applied to $h \in \calC_c(\RR)$ such that $h(x) = 1$ for any $x \in K$ we get
	\[
		\limsup_{L \to \infty} \sigma^L_\eta(K) < \infty.
	\]
	Thus, by \eqref{eq:88} and Corollary~\ref{cor:4} we get
	\begin{equation}
		\label{eq:89}
		\lim_{L \to \infty}
		\frac{1}{\rho_L} \int_\RR f(\lambda) \: \nu_\eta^L ({\rm d} \lambda)
		=
		\lim_{L \to \infty}
		\int_\RR f(x) g(\lambda) \: \sigma_\eta^L ({\rm d} \lambda).
	\end{equation}
	Since $f \cdot g \in \calC_c(\RR)$, by \eqref{eq:87} we conclude that
	\begin{equation}
		\label{eq:90}
		\lim_{L \to \infty}
		\int_\RR f(\lambda) g(\lambda) \: \sigma_\eta^L({\rm d} \lambda)
		=
		\int_\RR f(\lambda) g(\lambda) \: \mu_\eta({\rm d} \lambda).
	\end{equation}
    Since $f \in \calC_c(U)$ was arbitrary, by combining \eqref{eq:89} with \eqref{eq:90}, the conclusion~\eqref{eq:137} follows.
\end{proof}

Recall that in Corollary~\ref{cor:4}, under some hypotheses, we proved \eqref{eq:138} (cf. \eqref{eq:130}). We would like to express the value of $g$ more explicitly. In view of \eqref{eq:137}, one way to do that is to compute the value of the left-hand side of \eqref{eq:137} in a different way.

%For $\eta \in \sS^1$, $z \in \CC$ we denote by $u(\cdot,\eta;z)$ 
%the solution of
%\[
%	\begin{cases}
%		-(p u')' + q u = z w u \quad \text{on } (0, \infty) \\
%		\eta_2 u(0) + \eta_1 u'(0) = 0.
%	\end{cases}
%\]
By \cite[Example, p.225]{Teschl2014},
\[
	\sigma(H^L_\eta) = \{ \lambda : u(L,\eta;\lambda) = 0 \}.
\]
For any $L \in (0, \infty)$ and any $\eta \in \tsS^1$, by \cite[Theorem D.6]{Bennewitz2020}, the function $\CC \ni z \mapsto u(L,\eta;z)$ is entire of 
the order less or equal to $1/2$. Thus the Hadamard factorization theorem (see e.g. \cite[Theorem E.2.7]{Bennewitz2020}) 
implies that there are $A \neq 0, k \in \NN_0$ and $(\lambda_j : j \geq 1)$, such that 
\[
	u(L, \eta; z) 
	= 
	A z^{k} 
	\prod_{j=1}^\infty 
	\bigg( 1 -\frac{z}{\lambda_j} \bigg).
\]
Since $u(L,\eta;\cdot)$ has simple zeros, we must have $k \in \{0, 1\}$. Thus,
\begin{equation}
	\label{eq:74}
	-\frac{\partial_z u(L,\eta;z)}{u(L,\eta;z)}
	=
	\sum_{\lambda \in \sigma(H^L_\eta)} \frac{1}{\lambda - z} 
	=
	\calC \big[ \nu^L_\eta \big](z), \quad z \in \CC \setminus \RR
\end{equation}
where for a positive measure $\nu$ on the real line, $\calC[\mu]$ denotes its \emph{Cauchy transform}, that is
\[
	\calC[\nu](z) = 
	\int_\RR \frac{1}{\lambda - z} \nu({\rm d} \lambda), \quad z \in \CC \setminus \RR
\]
provided the integral exists. Motivated by \cite[Lemma 4.1]{SwiderskiTrojan2023} we are going to analyze the asymptotic 
behavior of the left-hand side of \eqref{eq:74}.

\begin{theorem}
	\label{thm:9}
	Suppose that the hypotheses of Theorem \ref{thm:4} are satisfied. Then for any $t \in [0,\omega]$
	\begin{equation}
		\label{eq:85}
		\lim_{n \to \infty} 
		\calC \big[ \tfrac{1}{\rho_{t+n\omega}} \nu_\eta^{t+n\omega} \big](z) =
		\frac{1}{\gamma}
		\frac{\partial_z \tr \frakT(\omega;0)}{i \sqrt{-\discr \frakT(\omega;0)}}, 
		\quad z \in \CC_{\varsigma}.
	\end{equation}
\end{theorem}
\begin{proof}
    Let $K$ be a compact subset of $\CC_{\varsigma}$ with non-empty interior. By Theorem~\ref{thm:4} for any $z \in K$, there are two linearly independent mappings $\mathbf{u}^+(\cdot;z)$ and $\mathbf{u}^-(\cdot;z)$ solving \eqref{eq:10}. Moreover, for any $x \geq 0$ the mappings $\mathbf{u}^+(x;\cdot)$ and $\mathbf{u}^-(x;\cdot)$
	are continuous on $K$ and holomorphic in $\intr{K}$. Furthermore, there is $M \geq 1$ such that the limits
	\begin{equation}
            \label{eq:148}
		\phi^+(t;z) = \lim_{k \to \infty} \frac{u_k^+(t;z)}{\prod_{M \leq j < k} \lambda_j^+(s; z)},
		\quad\text{and}\quad
		\phi^-(t;z) = \lim_{k \to \infty} \frac{u_k^-(t;z)}{\prod_{M \leq j < k} \lambda_j^-(s; z)}
	\end{equation}
	exist uniformly with respect to $(t,z) \in [0,\omega] \times K$ and define non-vanishing functions.
    
	Let $\mathbf{v}(\cdot; z)$ be any solution to \eqref{eq:10}. By the linear independence of $\mathbf{u}^+(\cdot;z)$ and $\mathbf{u}^-(\cdot;z)$, there are functions $f, g: K \to \CC$,
	such that
	\begin{equation} 
		\label{eq:76}
		\mathbf{v}(\cdot;z) 
		= 
		f(z) \mathbf{u}^+(\cdot;z) + g(z) \mathbf{u}^-(\cdot;z).
	\end{equation}
	By computing the Wronskians we get
	\begin{align*}
		\Wrk \big( \mathbf{v}(\cdot; z), \mathbf{u}^-(\cdot; z) \big) 
		&=
		f(z) \Wrk \big( \mathbf{u}^+(\cdot; z), \mathbf{u}^-(\cdot; z) \big)
		\intertext{and}
		\Wrk \big( \mathbf{v}(\cdot;z), \mathbf{u}^+(\cdot;z) \big) &=
		-g(z) \Wrk \big( \mathbf{u}^+(\cdot;z), \mathbf{u}^-(\cdot;z) \big).
	\end{align*}
	Thus,
	\[
		f(z) = 
		\frac
		{\Wrk \big( \mathbf{v}(\cdot;z), \mathbf{u}^-(\cdot;z) \big)}
		{\Wrk \big( \mathbf{u}^+(\cdot;z), \mathbf{u}^-(\cdot;z) \big)}, 
		\quad \text{and} \quad
		g(z) = 
		-
		\frac
		{\Wrk \big( \mathbf{v}(\cdot;z), \mathbf{u}^+(\cdot;z) \big)}
		{\Wrk \big( \mathbf{u}^+(\cdot;z), \mathbf{u}^-(\cdot;z) \big)}.
	\]
    In particular, $\mathbf{v}$ is continuous on $K$ and holomorphic on $\intr{K}$ if and only if functions $f$ and $g$ are continuous on $K$ and holomorphic on $\intr{K}$.
	
    Next, by \eqref{eq:76} we have
	\begin{equation}
		\label{eq:79}
		v_n(t; z) = f(z) u_n^+(t;z) + g(z) u_n^-(t;z), 
            \quad t \in [0,\omega], n \geq 0.
	\end{equation}
	Thus, for $n \geq M$
	\begin{align*}
		\frac{v_n(t; z)}{\prod_{M \leq j < n} \lambda_j^+(s; z)} 
		&=
		f(z) \frac{u_n^+(t;z)}{\prod_{M \leq j < n} \lambda_j^+(s; z)} +
		g(z) \frac{u_n^-(t;z)}{\prod_{M \leq j < n} \lambda_j^+(s; z)} \\
		&=
		f(z) \frac{u_n^+(t;z)}{\prod_{M \leq j < n} \lambda_j^+(s; z)} +
		g(z) \frac{u_n^-(t;z)}{\prod_{M \leq j < n} \lambda_j^-(s; z)} \cdot
		\frac{\prod_{M \leq j \leq n} \lambda_j^-(s; z)}{\prod_{M \leq j < n} \lambda_j^+(s; z)}.
	\end{align*}
    Now, by \eqref{eq:77} and \eqref{eq:148} we obtain
	\begin{equation} 
		\label{eq:80}
		\lim_{n \to \infty}
		\frac{v_n(t; z)}{\prod_{M \leq j \leq n} \lambda_j^+(s; z)} 
		=
		f(z) \phi^+(t;z)
	\end{equation}
    uniformly with respect to $(t,z) \in [0,\omega] \times K$. Recall that 
    $\phi^+$ is a continuous non-vanishing function. Therefore, if $f(z) \neq 0$, then there is $\delta > 0$, $M' \geq M$,
	such that for all $n \geq M'$ and all $t \in [0,\omega]$
	\[
		\delta
		\prod_{M \leq j < n} |\lambda_j^+(s; z)|^2
		\leq
		|v_n(t; z)|^2
	\]
    By \eqref{eq:106}, \eqref{eq:99}, \eqref{eq:41} and \eqref{eq:16} we have
    \[
        |\lambda_j^+(s;z)|^2 \geq \det X_{j}(s;z) = \frac{p_j(s)}{p_{j+1}(s)}.
    \]
    Thus,
    \[
        \delta \frac{p_M(s)}{p_n(s)} w_n(t) \leq |v_n(t;z)|^2 w_n(t).
    \]
    By Lemma~\ref{lem:3} there is $\delta'>0$ such that 
    \[
        \delta' \frac{w_n(t)}{p_n(t)} \leq |v_n(t;z)|^2 w_n(t).
    \]
    Therefore,
    \[
        \int_{M' \omega}^\infty \frac{w(x)}{p(x)} \ud x \leq 
        \int_{M' \omega}^\infty |v(x;z)|^2 w(x) \ud x.
    \]
    Consequently, by the Carleman condition \eqref{eq:26}, if $v(\cdot; z) \in L^2(w)$, then $f(z) = 0$. Since $\tau$ is in the limit point case, it yields $u^-(\cdot;z) \in L^2(w)$ for any $z \in K$.

    Let us turn to the proof of \eqref{eq:85}. Let us fix $t \in [0,\omega]$ and define
    \begin{equation}
        \label{eq:174}
	\phi_n(z) = \frac{u_n(t, \eta; z)}{\prod_{M \leq j \leq n} \lambda_j^+(s; z)}, \quad z \in \intr(K).
    \end{equation}
    Then, by \eqref{eq:80},
	\[
		\lim_{n \to \infty}
		\phi_n(z)
		=
		\phi(z)
	\]
	uniformly with respect to $z \in K$, where
        \[
		\phi(z) = f(z) \phi^+(t; z), \quad z \in K.
	\]
    Let us show that $\phi$ is non-vanishing. 
    Since $\tau$ is in the limit point case at $+\infty$, the operator $H_\eta$ is self-adjoint. Thus, we have $\sigma(H_\eta) \subset \RR$. In particular, $K \cap \sigma(H_\eta) = \emptyset$. Notice that since $u(\cdot, \eta;z)$ satisfies the boundary conditions of $H_\eta$, it cannot belong to $L^2(w)$, otherwise it would be an eigenfunction of $H_\eta$. Consequently, $u(\cdot,\eta;z)$ is linearly independent with $u^-(\cdot;z)$, which implies that $f(z) \neq 0$. Since $(\phi_n : n \geq M)$ is a sequence of holomorphic functions on $\intr(K)$ 
	and $\phi$ is everywhere non-zero, we get
	\begin{equation}
		\label{eq:84}
		\lim_{n \to \infty} \frac{\partial_z \phi_n(z)}{\phi_n(z)} = \frac{\partial_z \phi(z)}{\phi(z)}
	\end{equation}
	uniformly with respect to $z \in \intr(K)$. Notice that
	\begin{equation}
		\label{eq:81}
		\frac{\partial_z \phi_n(z)}{\phi_n(z)} = 
		\frac{\partial_z u_n(t, \eta; z)}{u_n(t, \eta; z)}
		-
		\sum_{j=M}^n \frac{\partial_z \lambda_j^+(s; z)}{\lambda_j^+(s; z)}.
	\end{equation}
	Since the Carleman's condition \eqref{eq:26} is satisfied,
	\[
		\lim_{n \to \infty} \rho_n(s) = +\infty.
	\]
	Therefore, by \eqref{eq:84} and \eqref{eq:81} we obtain
	\begin{equation}
		\label{eq:54}
		\lim_{n \to \infty} 
		\frac{1}{\rho_n(s)} \frac{\partial_z u_n(t,\eta; z)}{u_n(t,\eta; z)} 
		=
		\lim_{n \to \infty}
		\frac{1}{\rho_n(s)}
		\sum_{j=M}^n \frac{\partial_z \lambda_j^+(s; z)}{\lambda_j^+(s; z)}.
	\end{equation}
	To compute the right-hand side of \eqref{eq:54}, we are going to use the Stolz--Ces\`{a}ro theorem. Namely,
	we have 
	\begin{align}
            \label{eq:176}
		\lim_{n \to \infty}
		\frac{1}{\rho_n(s)}
		\sum_{j=M}^n \frac{\partial_z \lambda_j^+(s; z)}{\lambda_j^+(s; z)}
		&=
		\lim_{n \to \infty}
		\frac{1}{\gamma_n(s)}
		\frac{\partial_z \lambda_n^+(s; z)}{\lambda_n^+(s; z)} \\
            \nonumber
		&=
		\frac{1}{\gamma}
		\frac{\partial_z \tr \frakT(\omega;0)}{i \sqrt{-\discr \frakT(\omega;0)}}
	\end{align}
	where the last equality follows by Proposition~\ref{prop:3}. By the Carleman condition we easily get
    \[
        \lim_{n \to \infty} \frac{\rho_n(s)}{\rho_{t+n\omega}} = 1
    \]
    (cf. \eqref{eq:135} and \eqref{eq:134}). Therefore, by \eqref{eq:74} we get \eqref{eq:85}. Since $t \in [0,\omega]$ was arbitrary, it completes the proof of the theorem.
\end{proof}

\begin{corollary} 
	\label{cor:5}
	Suppose that the hypotheses of Theorem~\ref{thm:4} are satisfied. Then for every $f \in \calC(\RR)$ such that 
	\[
		\sup_{\lambda \in \RR} (1+\lambda^2)|f(\lambda)| < \infty
	\]
	we have for any $t \in [0,\omega]$
	\begin{equation}
            \label{eq:136}
		\lim_{n \to \infty} 
		\frac{1}{\rho_{t+n\omega}}
		\int_\RR f(\lambda) \: \nu_\eta^{t+n\omega}({\rm d} \lambda)  
		=
		\int_\RR f(\lambda) \: \nu_\infty({\rm d} \lambda),
	\end{equation}
	where the measure $\nu_\infty$ is purely absolutely continuous with the density
	\begin{equation}
        \label{eq:141}
		\frac{\ud \nu_\infty}{\ud \lambda} \equiv 
            \frac{1}{\pi}
		\frac{|\partial_z \tr \frakT(\omega;0)|}{\gamma \sqrt{-\discr \frakT(\omega;0)}}.
	\end{equation}
\end{corollary}
\begin{proof}
    Let us recall that $\varsigma = \sign{\partial_z \tr \frakT(\omega;0)}$. Now, the conclusion easily follows by \cite[Lemma 4.1]{SwiderskiTrojan2023} together with Theorem~\ref{thm:9} and \cite[Theorem A.3 and Remark A.4]{SwiderskiTrojan2023}.
\end{proof}

\begin{corollary}
	\label{cor:6}
	Suppose that the hypotheses of Corollary \ref{cor:4} are satisfied. 
    Then the measure $\mu_\eta$ is purely absolutely 
	continuous on $\RR$ with the density
	\begin{equation}
            \label{eq:140}
		\mu_\eta'(\lambda) = 
		\frac{1}{\pi} 
		\frac{|\partial_z \tr \frakT(\omega;0)|}{\gamma \sqrt{-\discr \frakT(\omega;0)}} 
		\frac{1}{g(\lambda)}, \quad \lambda \in \RR,
	\end{equation}
	where 
	\begin{equation}
		\label{eq:88}
		g(\lambda) = 
		\lim_{L \to \infty} 
		\frac{1}{\rho_L} K_L(\lambda,\lambda;\eta), \quad \lambda \in \RR.
	\end{equation}
	In particular, $\mu_\eta'$ is a continuous and positive function on $\RR$.
\end{corollary}
\begin{proof}
    By Lemma~\ref{lem:1} we have
    \begin{equation}
        \label{eq:142}
        \lim_{L \to \infty} 
        \frac{1}{\rho_{L}} 
        \int_\RR f(\lambda) \nu_\eta^{L}(\ud \lambda)
        =
        \int_\RR f(\lambda) g(\lambda) \mu_\eta(\ud \lambda), \quad f \in \calC_c(\RR).
    \end{equation}
    Thus, by \eqref{eq:136} we have 
    \[
        \int_\RR f(\lambda) \nu_\infty(\ud \lambda) = 
        \int_\RR f(\lambda) g(\lambda) \mu_\eta(\ud \lambda), \quad f \in \calC_c(\RR).
    \]
    Since $g$ is positive everywhere, by \eqref{eq:141} the formula \eqref{eq:140} follows.
\end{proof}

\section{The case of empty essential spectrum}
\label{sec:9}
In this section we investigate the case when $|\tr \frakT(\omega; 0)| > 2$. We seek for assumptions under which the operator
$H_\eta$ has the essential spectrum empty.

The proof of \cite[Lemma 5.13]{Moszynski2009} together with the subordinacy theory for Sturm--Liouville operators
(see e.g. \cite{Clark1993}) yields the following lemma (cf. \cite[Theorem 8.1b]{Weidmann2005}).
\begin{proposition} 
	\label{prop:7}
	Suppose that $\tau$ is regular at $0$ and limit point at $+\infty$. Let $K \subset \RR$ be a compact interval with 
	non-empty interior. If for any $\lambda \in K$ there exists a non-zero $u \in L^2([0,\infty), w)$ such that 
	$\tau u = \lambda u$, then for any $\eta \in \sS^1$ the operator $H_\eta$ is pure point in $K$.
\end{proposition}
Taking into account Proposition~\ref{prop:7}, the proof of \cite[Theorem 5.3]{Silva2007} yields the following result.
\begin{theorem}
	\label{thm:11}
	Suppose that $\tau$ is regular at $0$ and limit point at $+\infty$. Let $K \subset \RR$ be a compact interval with 
	non-empty interior. If there exists a family $\{ u(\cdot; \lambda) : \lambda \in K \}$ of non-zero functions such that 
	$\tau u(\cdot;\lambda) = \lambda u(\cdot;\lambda)$, satisfying
	\begin{equation}
		\label{eq:114}
		\sup_{\lambda \in K} \int_0^\infty |u(x; \lambda)|^2 w(x) \ud x < \infty,
	\end{equation}
	and
	\begin{equation} 
		\label{eq:100}
		\lim_{L \to \infty} \sup_{\lambda \in K} \int_L^\infty |u(x; \lambda)|^2 w(x) \ud x = 0,
	\end{equation}
	and for all $L \in (0, \infty)$, and $\lambda \in \intr(K)$,
	\begin{equation}
		\label{eq:101}
		\lim_{\lambda' \to \lambda} 
		\int_0^L | u(x; \lambda) - u(x;\lambda') |^2 w(x) \ud x = 0.
	\end{equation}
	Then for every $\eta \in \tsS^1$, the operator $H_\eta$ satisfies $\sigmaEss(H_\eta) \cap \intr{K} = \emptyset$.
\end{theorem}
\begin{proof}
	Let $\eta \in \tsS^1$. In view of \eqref{eq:114} and Proposition~\ref{prop:7} the spectrum of $H_\eta$ is pure point on $K$. 
	Therefore, it is enough to prove that $\sigmaP(H_\eta)$ has no accumulation points in $\intr{K}$. 
	Let us define a family of functions $(f_L : L \in (0,\infty] )$ by the setting 
	\[
		f_L(\lambda, \lambda') =
		\int_0^L u(x;\lambda) \overline{u(x;\lambda')} w(x) \ud x, \quad \lambda,\lambda' \in K.
	\]
	By the Cauchy--Schwarz inequality, we have
	\[
		|f_L(\lambda, \lambda')| \leq
		\bigg( \int_0^L |u(x;\lambda)|^2 w(x) \ud x \bigg)^{1/2}
		\bigg( \int_0^L |u(x;\lambda')|^2 w(x) \ud x \bigg)^{1/2},
	\]
	thus
	\[
		\sup_{\lambda,\lambda' \in K} |f_L(\lambda, \lambda')| \leq
		\sup_{\lambda \in K} \int_0^\infty |u(x;\lambda)|^2 w(x) \ud x,
	\]
	which by \eqref{eq:114} is finite. Next, we claim that for each $L \in (0,\infty)$ the function $f_L$ is continuous. 
	Let us show that $f_L$ is continuous with respect to the first variable. By the Cauchy--Schwarz inequality we have
	\begin{align*}
  	  	|f_L(\lambda'', \lambda') - f_L(\lambda, \lambda')| 
  	  	&\leq
		\bigg( \int_0^L |u(x;\lambda'') - u(x;\lambda)|^2 w(x) \ud x \bigg)^{1/2}
		\bigg( \int_0^L |u(x;\lambda')|^2 w(x) \ud x \bigg)^{1/2}.
	\end{align*}
	In view of \eqref{eq:114} and \eqref{eq:101}, we obtain
	\[
  	  \lim_{\lambda'' \to \lambda} f_L(\lambda'',\lambda') = f_L(\lambda, \lambda').
	\]
	Similarly one can show continuity of $f_L$ with respect to the second variable. Finally, let us observe that
	\[
		|f_L(\lambda, \lambda') - f_\infty(\lambda, \lambda')| \leq 
		\sup_{\lambda \in K} \int_L^\infty |u(x;\lambda)|^2 w(x) \ud x,
	\]
	which by \eqref{eq:100} tends to $0$ as $L$ approaches infinity. Consequently, the function $f_\infty$ is also continuous.

	Suppose, on the contrary to our claim, that $\sigmaEss(H_\eta) \cap \intr(K) \neq \emptyset$. Let 
	$\lambda \in \sigmaEss(H_\eta) \cap \intr(K)$. Then there exists a sequence $(\lambda_n : n \geq 0) \subseteq K$ such that
	\begin{enumerate}[label=(\roman*), ref=\roman*, leftmargin=3em]
		\item $\lambda_n \in \intr{K} \cap \sigmaP(H_\eta)$ for any $n$;
		\item $\lim_{n \to \infty} \lambda_n = \lambda$;
		\item all $\lambda_n$ are different.
	\end{enumerate}
	Then by the continuity of $f_\infty$, we have
	\[
		0 < f_\infty(\lambda, \lambda) = \lim_{n \to \infty} f_\infty(\lambda_{n+1},\lambda_n) = 0
	\]
	where the last equality follows from the fact that for self-adjoint operators eigenspaces corresponding to different
	eigenvalues are orthogonal. This leads to contradiction and the theorem follows.
\end{proof}
Let us state the main result of this section.

%{ \color{red} \textbf{TODO}: formulate conditions on $w$ implying that $\tau$ is in the limit-point case.}

\begin{theorem} 
	\label{thm:10}
	Let $\omega > 0$. Suppose that $(p, q, w)$ are $\omega$-periodically modulated Sturm--Liouville parameters, such that 
	the transfer matrix $\frakT$ corresponding to $(\frakp, \frakq, \frakw)$ satisfies $\abs{\tr \frakT(\omega; 0)} > 2$. 
	Assume that
	\[
    	\frac{q}{p}, \frac{w}{p}, \frac{p'}{p} \in \calD_1^\omega(L^1; \RR).
	\]
	If $\tau$ is in the limit point case at $+\infty$, then for each $\eta \in \tsS^1$ the operator $H_\eta$ satisfies 
	$\sigmaEss(H_\eta) = \emptyset$.
\end{theorem}
\begin{proof}
	Let $z_0 \in \RR$. Let $K \subset \RR$ be a compact interval containing $z_0$ in its interior. We consider the equation 
	\eqref{eq:14} for $s=0$, that is
	\begin{equation} 
		\label{eq:94}
		\mathbf{u}_{n+1}(0;z) = X_n(0;z) \mathbf{u}_n(0;z), \quad n \geq 0, z \in K.
	\end{equation}
	Let us recall that $(X_n(0;\cdot) : n \geq 0)$ is a sequence of continuous mappings on $K$ with values in $\GL(2,\RR)$.
	Since
	\[
		\discr \frakT(\omega; 0) = (\tr \frakT(\omega; 0))^2 - 4 > 0,
	\]
	the matrix $\frakT(\omega; 0)$ has two distinct eigenvalues
	\begin{equation}
		\label{eq:108}
		\lambda^-_\infty = \xi_-\bigg(\frac{\tr \frakT(\omega; 0)}{2} \bigg),
		\quad \text{and} \quad 
		\lambda^+_\infty = \xi_+\bigg(\frac{\tr \frakT(\omega; 0)}{2} \bigg).
	\end{equation}
	Since $|\tr \frakT(\omega; 0)| > 2$, we have
	\begin{equation}
		\label{eq:111}
		0 < |\lambda^-_\infty| < 1 < |\lambda^+_\infty|. 
	\end{equation}
	Moreover, by \eqref{eq:99} and Proposition~\ref{prop:4}, 
	\begin{equation}
		\label{eq:107}
		\lim_{n \to \infty} X_n(0; z) = \frakT(\omega; 0)
	\end{equation}
	uniformly with respect to $z \in K$. Therefore, there are $M \geq 1$ and $\delta > 0$ such that for all $n \geq M$,
	\[
		\discr X_n(0; z) > \delta. 
	\]
	Hence, for all $n \geq M$ and $z \in K$, the matrix $X_n(0; z)$ has two distinct eigenvalues
	\begin{equation} 
		\label{eq:110}
		\lambda^-_n(0;z) = \sqrt{\det X_n(0;z)} 
		\xi_- \bigg( \frac{\tr X_n(0;z)}{2 \sqrt{\det X_n(0;z)}} \bigg), \quad
		\lambda^+_n(0;z) = \sqrt{\det X_n(0;z)} 
		\xi_+ \bigg( \frac{\tr X_n(0;z)}{2 \sqrt{\det X_n(0;z)}} \bigg)
	\end{equation}
	see the formula \eqref{eq:106}. Since $|\lambda^-_n(z) \leq |\lambda^+_n(z)|$ for all $n \in \NN$ and $z \in K$, we must have
	\begin{equation}
		\label{eq:115}
		\lim_{n \to \infty} \lambda^-_n(0; z) = \lambda^-_\infty,
		\quad
		\lim_{n \to \infty} \lambda^+_n(0; z) = \lambda^+_\infty,
	\end{equation}
	uniformly with respect to $z \in K$. In view of Proposition~\ref{prop:2}, 
	$(X_n(0;\cdot) : n \geq 0) \in \calD_1(K; \GL(2,\RR))$. By \cite[Lemma 4.2]{Discrete} there exists a compact interval 
	$K' \subset K$ containing $z_0$ in its interior such that the hypotheses of \cite[Theorem 4.1]{Discrete} are satisfied. 
	Moreover, by the proof of \cite[Theorem 4.1]{Discrete} for $r=1$ there are a constant $M \geq 1$ and a sequence 
	$(\mathbf{u}^-_n(0; \cdot) : n \geq 0)$ of continuous mappings on $K'$ satisfying \eqref{eq:94}, and such that
	\begin{equation}
		\label{eq:96}
		\lim_{n \to \infty} 
		\sup_{z \in K'}
		\bigg\| \frac{\mathbf{u}^-_n(0;z)}{\prod_{j=M}^{n-1} \lambda_j^-(0;z)} -   \nu^- \bigg\| = 0
	\end{equation}
	where $\nu^-$ is a non-zero vector. Hence, there is a constant $c>0$ such that
	\[
		\| \mathbf{u}^-_n(0;z) \| \leq c
		\prod_{j=M}^{n-1} |\lambda_j^-(0;z)|, \quad n \geq M, z \in K'.
	\]
	Notice that by \eqref{eq:110} and \eqref{eq:16} we have
	\[
		\| \mathbf{u}^-_n(0;z) \| \leq c
		\sqrt{\frac{p_M(0)}{p_n(0)}}
		\prod_{j=M}^{n-1} 
		\bigg| \xi_- \bigg( \frac{\tr X_j(0;z)}{2 \sqrt{\det X_n(0;z)}} \bigg) \bigg|, \quad n \geq M, z \in K'.
	\]
	By \eqref{eq:107} and \eqref{eq:111}, for every $\epsilon \in (0, 1- |\lambda_\infty^-|)$ there are $M' \geq M$ 
	and $c'>0$ such that
	\begin{equation}
		\label{eq:97}
		\| \mathbf{u}^-_n(0;z) \| \leq c'
		\sqrt{\frac{1}{p_n(0)}} 
		(1-\epsilon)^{n}, \quad n \geq M', z \in K'.
	\end{equation}
    For any $x \geq 0$ define
	\[
		\mathbf{u}^-(x;z) = T(x;z) \mathbf{u}^-_0(0;z).
	\]
    Notice that since $\mathbf{u}^-_0(0;\cdot)$ is continuous on $K'$, the mapping $\mathbf{u}^-$ is continuous on $[0,\infty) \times K'$.
    Using \eqref{eq:14}, \eqref{eq:99} and \eqref{eq:41} we easily get
    \[
        \mathbf{u}^-(n\omega;z) = \mathbf{u}^-_n(0;z), \quad n \geq 0.
    \]
    Then by \eqref{eq:6} we get for any $t \in [0,\omega]$
	\begin{equation}
		\label{eq:98}
		\mathbf{u}^-(t + n\omega;z) = U_{0;n}(t;z) \mathbf{u}^-_n(0;z).
	\end{equation}
	In view of Proposition~\ref{prop:4} we have
	\[
		\lim_{n \to \infty} U_{0;n}(t;z) = \frakT(t;0)
	\]
	uniformly with respect to $z \in K'$. Since $\frakT(\cdot;0)$ is continuous, it is uniformly bounded on $[0,\omega]$. 
	Therefore, in view of \eqref{eq:98} and \eqref{eq:97} there is a constant $c>0$ such that for any $t \in [0,\omega]$ 
	and any $z \in K'$
	\[
		\| \mathbf{u}^-(t + n\omega;z) \| \leq  c \sqrt{\frac{1}{p_n(0)}} (1-\epsilon)^n.
	\]
	Therefore,
	\begin{align*}
		\int_{M' \omega}^\infty \sup_{z \in K'} |u^-(x;z)|^2 w(x) \ud x 
		&=
		\sum_{n=M'}^\infty \int_0^\omega \sup_{z \in K'} |u^-_n(t;z)|^2 w_n(t) \ud t \\
		&=
		\int_0^\omega \sum_{n=M}^\infty \sup_{z \in K'} |u^-_n(t;z)|^2 w_n(t) \ud t \\
		&\leq
		c
		\sum_{n=M'}^\infty (1-\epsilon)^{2n} 
		\int_0^\omega \frac{w_n(t)}{p_n(0)} \ud t.
	\end{align*}
	Next, we notice that
	\[
		\frac{w_n(t)}{p_n(0)} = \frac{w_n(t)}{p_n(t)} \frac{p_n(t)}{p_n(0)},
	\]
	thus, by Lemma~\ref{lem:3} and \eqref{eq:3a}
	\[
		\lim_{n \to \infty} 
		\int_0^\omega \frac{w_n(t)}{p_n(0)} \ud t
		=
		\lim_{n \to \infty}
		\int_0^\omega \frac{w_n(t)}{p_n(t)} \frac{\frakp(t)}{\frakp(0)} \ud t.
	\]
	Because the function $\frakp$ is continuous, its uniformly bounded on $[0,\omega]$. Therefore, by \eqref{eq:3a} we obtain
	\[
		\lim_{n \to \infty} 
		\int_0^\omega \frac{w_n(t)}{p_n(0)} \ud t = 0.
	\]
	Consequently, we get
	\begin{equation}
            \label{eq:154}
		\int_{M' \omega}^\infty \sup_{z \in K'} |u^-(x;z)|^2 w(x) \ud x < \infty.
	\end{equation}
	Since $u^-$ is continuous function on $[0, \infty) \times K'$, we get
	\begin{equation}
	    \label{eq:112}
		\int_{0}^\infty \sup_{z \in K'} |u^-(x;z)|^2 w(x) \ud x < \infty.
	\end{equation}
	Observe that \eqref{eq:112} implies \eqref{eq:114} and \eqref{eq:100}. Since $u^-$ is continuous, we also have \eqref{eq:101}.
	Hence, if $\tau$ is in the limit point case at $+\infty$, then Theorem~\ref{thm:11} implies that for any $\eta \in \tsS^1$ 
	the operator $H_\eta$ satisfies $\sigmaEss(H_\eta) \cap \{z_0\} = \emptyset$. Since $z_0 \in \RR$ was arbitrary, 
	the theorem follows.
\end{proof}

\begin{remark}
	Let us recall that if $\tau$ is regular at $0$, and it is in the limit-circle case at $+\infty$, then all self-adjoint 
	extensions of $H_\eta$ have empty essential spectrum. Therefore, the conclusion of Theorem~\ref{thm:10} holds true also 
	for such extensions.
\end{remark}

\section{Examples}
\label{sec:10}
In this section we provide a collection of examples where our method applies. The first example demonstrates that the assumptions 
regarding the Sturm–Liouville parameters do not necessarily lead to them being regularly varying.
\begin{example}
    \label{ex:5}
    Let $\omega > 0$ and let $\frakq \in \Lloc([0,\infty))$ be a $\omega$-periodic real-valued function. For $0 < a < b < 1$ let $p$ be a strictly positive continuously differentiable function such that
    \[
        p(t) = \exp \Big( \frac{a+b}{2} \log{t} + \frac{b-a}{2} \log{t} \cdot \sin \big( \log{\log{t}} \big) \Big), \quad t > \exp(\exp(1)).
    \]
    Notice that
    \begin{equation}
     \label{eq:144}
         t^b \leq p(t) \leq t^a, \quad t > \exp(\exp(1)).
    \end{equation}
    Both bounds are attained infinitely many times. In particular, $p$ is not regularly varying. Let us consider
    \[
      q(t) = p(t) \frakq(t), \quad w(t) \equiv 1, \quad t \geq 0.
    \]
    Let $\frakp(t) \equiv 1$ and $\frakw(t) \equiv 1$. Notice that there is a constant $c>0$ such that for each $t > \exp(\exp(1))$,
    \[
        \frac{w(t)}{p(t)} \leq \frac{1}{t^b}, \quad
        \frac{q(t)}{p(t)} = \frac{\frakq(t)}{\frakp(t)}, \quad
        \big| (\log p)'(t) \big| \leq c \frac{1}{t},
    \]
    which easily implies that \eqref{eq:3a}, \eqref{eq:3b} and \eqref{eq:3c} are satisfied. Moreover, by \eqref{eq:144} 
    the condition~\eqref{eq:3d} holds true as well. Next, by direct computation with a help of \eqref{eq:144}, one can show that there 
    is a constant $c>0$ such that for all $t > \exp(\exp(1))$,
    \[
        \bigg| \bigg( \frac{w}{p} \bigg)'(t) \bigg| \leq \frac{c}{t^{1+b}}, \quad
        \big| (\log p)''(t) \big| \leq \frac{c}{t^2}.
    \]
    In particular, $(w/p)', (\log p)'' \in L^1([0,\infty))$. Hence, by \cite[p. 2]{Stolz1991a} it easily implies that \eqref{eq:D1-reg} 
    is satisfied.
\end{example}

\begin{example} 
    \label{ex:2}
    Let $\omega > 0$ and let $\frakq \in \Lloc([0,\infty))$ be a $\omega$-periodic real-valued function. For $0 < \kappa \leq 1$ we 
    consider the following Sturm--Liouville parameters
    \[
        p(t) = c_\kappa^2 (1+t)^{2 \kappa}, \quad
        q(t) = (1+t)^{2\kappa} \frakq(t) + q_{\mathrm{cor},\kappa}(t), \quad 
        w(t) \equiv 1, \quad t \geq 0,
    \]
    where
    \begin{equation}
        \label{eq:128}
        c_\kappa = 
        \begin{cases}
            \frac{1}{1-\kappa} & \text{if } 0<\kappa<1, \\
            1 & \text{otherwise } \kappa=1,
        \end{cases} \quad\text{and}\quad
        q_{\mathrm{cor},\kappa}(t) =
        -\frac{c_\kappa^2}{4} \cdot
        \begin{cases}
            \kappa (3\kappa-2) (1+t)^{2\kappa-2} & \text{if } 0<\kappa<1, \\
            1 & \text{otherwise }\kappa=1.
        \end{cases}
    \end{equation}
    Let $\frakp(t) \equiv c_\kappa^2$ and $\frakw(t) \equiv 1$. We notice that
    \begin{align} 
        \label{eq:121a}
        \frac{w(t)}{p(t)} &= \frac{1}{c_\kappa^2(1+t)^{2\kappa}}, \\
        \label{eq:121b}
        \frac{q(t)}{p(t)} &=
        \frac{\frakq(t)}{\frakp(t)} - \frac{1}{4(1+t)^{2}} \cdot
        \begin{cases}
            \kappa(3\kappa-2) & \text{if } 0<\kappa<1, \\
            1 & \text{otherwise } \kappa=1,
        \end{cases} \\
        \label{eq:121c}
        \frac{p'(t)}{p(t)} &= \frac{\frakp'(t)}{\frakp(t)} + \frac{2 \kappa}{1+t},
    \end{align}
    which easily implies \eqref{eq:3a}, \eqref{eq:3b}, \eqref{eq:3c}. We obviously have \eqref{eq:3d}. From \eqref{eq:121a}, 
    \eqref{eq:121b} and \eqref{eq:121c} it easily follows that the condition~\eqref{eq:D1-reg} is satisfied. 
    We are now going to apply the Liouville transformation to $(p,q,w)$ (see, e.g. \cite[Section 7]{Everitt2005}). It will 
    produce a unitary equivalent Schr\"{o}dinger operator on $[0,\infty)$ with a corresponding potential $V$. Define a function 
    by the formula
    \[
        x(t) = 
        \int_0^t \sqrt{\frac{w(t')}{p(t')}} \ud t', \quad t \geq 0.
    \]
    We immediately get
    \[
        x(t) = 
        \begin{cases}
            (1+t)^{1-\kappa} -1 & \text{if } 0<\kappa<1, \\
        \log(1+t) & \text{otherwise } \kappa=1.
        \end{cases}
    \]
    Consequently, its inverse is equal to
    \[
        t(x) = 
        \begin{cases}
            (1+x)^{1/(1-\kappa)} - 1 & \text{if } 0<\kappa<1, \\
            \ue^x -1 & \text{otherwise } \kappa=1,
        \end{cases} \quad x \geq 0.
    \]
    It can be computed that
    \[
        \bigg(\frac{p(t)}{(w(t))^3} \bigg)^{1/4} \cdot 
        \bigg( p \cdot \Big( \Big(\frac{p}{w} \Big)^{-1/4} \Big)' \bigg)'(t) = 
        q_{\mathrm{cor},\kappa}(t).
    \]
    Therefore, the potential $V$ is equal to
    \begin{align}
        \nonumber
        V(x) &= \big( 1+t(x) \big)^{2\kappa} \frakq \big( t(x) \big) \\
        \label{eq:123}
        &= 
        \begin{cases}
            (1+x)^{2\kappa/(1-\kappa)} \frakq \big( (1+x)^{1/(1-\kappa)} - 1 \big) & \text{if }0<\kappa<1, \\
            \ue^{2x} \frakq(\ue^{x} - 1) & \text{otherwise } \kappa=1,
        \end{cases}
        \quad x \geq 0.
    \end{align}
    In the case $0<\kappa<1$ let us define $a=\frac{2\kappa}{1-\kappa}, b = \frac{1}{1-\kappa}$. Then we easily get that 
    $a>0$ and
    \begin{equation} 
        \label{eq:122}
        V(x) = (1+x)^{a} \frakq \big( (1+x)^{(2+a)/2} - 1 \big), \quad x \geq 0.
    \end{equation}
\end{example}
In the following example we shall specify specific choices of $\frakq$ in Example~\ref{ex:2}.
\begin{example} 
    \label{ex:4}
    For any $c \in \RR$ consider the following $2\pi$-periodic Sturm--Liouville parameters
    \begin{equation} 
        \label{eq:127}
        \frakp(t) \equiv c_\kappa^2, \quad 
        \frakq(x) = -c + \sin(t),
        \quad \frakw(t) \equiv 1,
    \end{equation}
    where $c_\kappa$ is defined in \eqref{eq:128}. 

    Figures~\ref{img:1} and \ref{img:2} contain plots of a numerical computation of $\tr \frakT(2\pi;0)$. 
    In the first plot we have $\kappa=0.5$ and $c \in [-0.75,20]$, whereas in the second case we have $c=0$ 
    and $\kappa \in (0,1)$. The solution of $\tr \frakT(2\pi;0)=-2$ for $c=0$ satisfies $\kappa \approx 0.326$. 
    For concreteness, let us also mention that our numerical computations give: $\tr \frakT(2\pi;0) \approx 0.77$ for 
    $c=0,\kappa=0.5$ and $\tr \frakT(2\pi;0) \approx -2.61$ for $c=1,\kappa=0.5$. 

    \begin{figure}[!tbp]
    \centering
    \begin{minipage}[b]{0.45\textwidth}
        \includegraphics[width=\textwidth]{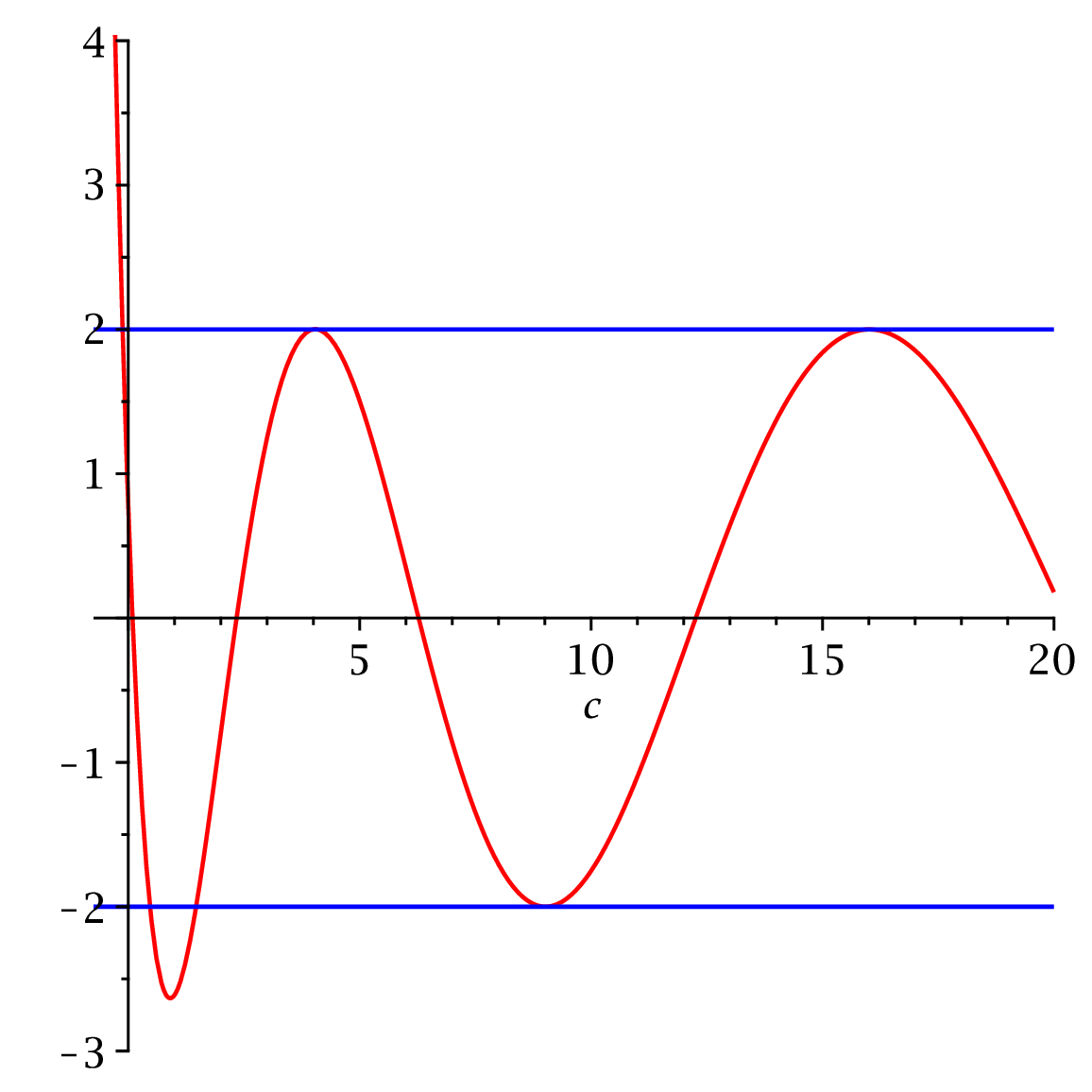}
        \caption{A plot of $\tr \frakT(2\pi;0)$ for \eqref{eq:127}, $c\in[-0.75,20]$ and $\kappa=0.5$.}
        \label{img:1}
    \end{minipage}
    \hfill
    \begin{minipage}[b]{0.45\textwidth}
        \includegraphics[width=\textwidth]{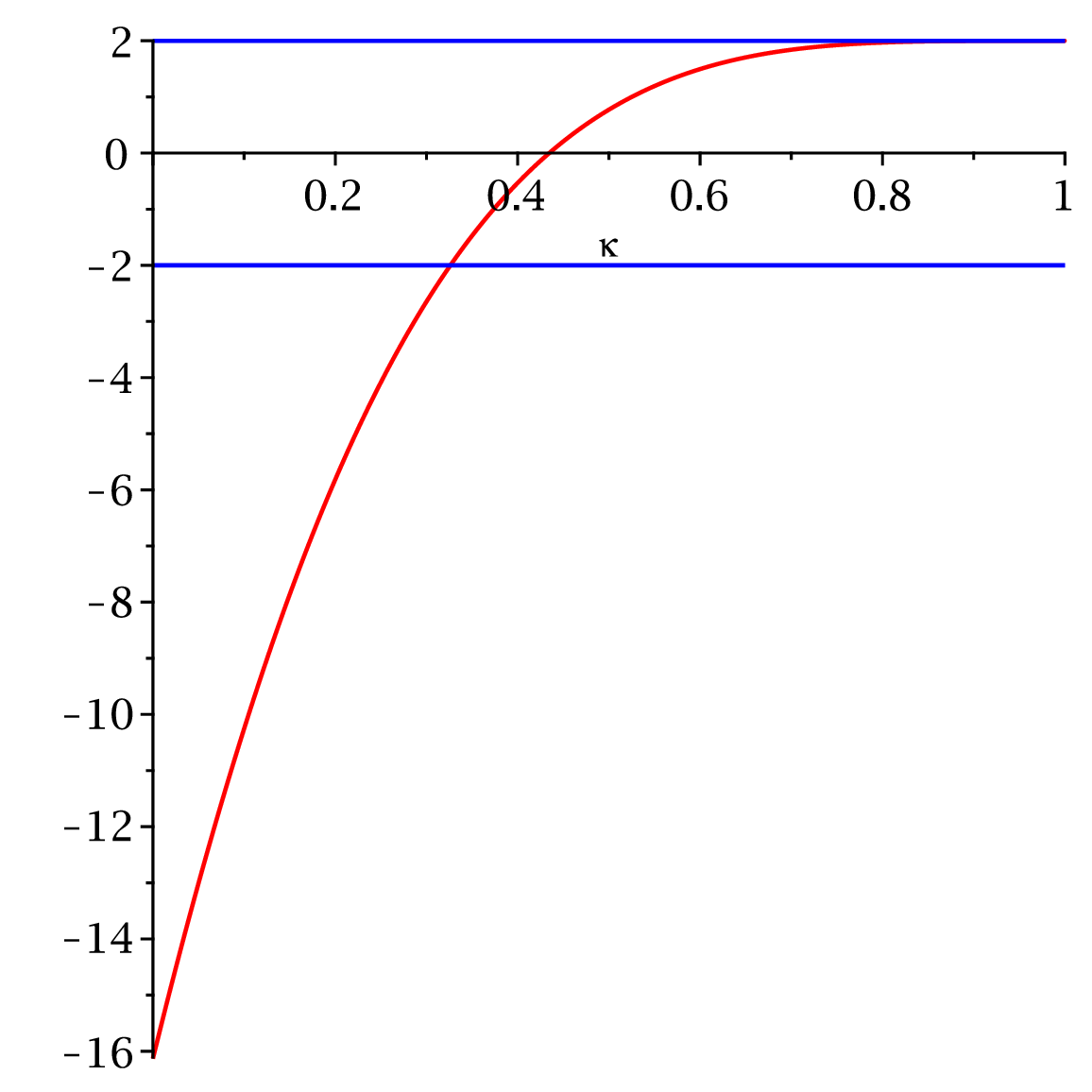}
        \caption{A plot of $\tr \frakT(2\pi;0)$ for \eqref{eq:127}, $c=0$ and $\kappa\in(0,1)$.}
        \label{img:2}
        \end{minipage}
    \end{figure}

Table~\ref{table:1} contains a summary of spectral properties of $H_\eta$ from Example~\ref{ex:2} depending on the trace of the monodromy matrix corresponding to \eqref{eq:127}. 

\begin{table}[h]
\caption{Spectral properties of $H_\eta$ depending on the trace of the monodromy matrix of \eqref{eq:127}.}
\centering
\renewcommand{\arraystretch}{1.75}
\begin{tabular}{|c|c|c|c|}
    \hline
    & $0 < \kappa \leq \frac{1}{2}$ & $\frac{1}{2} < \kappa  < 1$ & $\kappa=1$ \\
    \hline 
    $\big|\tr \frakT(2\pi;0)\big| < 2$ & $\sigmaAC(H_\eta) = \RR$ & 
    \multicolumn{2}{|c|}{$\tau$ is limit circle at $+\infty$} \\
    \hline
    $\big|\tr \frakT(2\pi;0)\big| > 2$ & 
    \multicolumn{3}{|c|}{all self-adjoint extensions of $H_\eta$ have no essential spectrum} \\
    \hline
\end{tabular}
\label{table:1}
\end{table}
\end{example}

\appendix
\section{Asymptotically periodic parameters}
\label{sec:11}
Recently in \cite{Behrndt2023} there has been studied another class of Sturm--Liouville parameters which is a perturbation of 
$\omega$-periodic case. We say that $(p, q, w)$ are \emph{$\omega$-asymptotically periodic Sturm--Liouville parameters} 
if $(p, q, w)$ are Sturm--Liouville parameters such that there are $\omega$-periodic Sturm--Liouville parameters 
$(\frakp, \frakq, \frakw)$ satisfying
\begin{equation}
    \label{eq:45}
    \lim_{n \to \infty} 
    \int_0^\omega 
    \big|w_n(t) - \frakw(t)\big| +
    \big|q_n(t) - \frakq(t)\big| + 
    \bigg|\frac{1}{p_n(t)} - \frac{1}{\frakp(t)} \bigg|
    \ud t = 0.
\end{equation}
Let us notice that \eqref{eq:45} implies
\begin{equation}
    \label{eq:165}
    \int_0^\infty w(x) \ud x = \infty.
\end{equation}

We aim to study solutions of \eqref{eq:2}. Let $u(\cdot;z)$ be any solution of \eqref{eq:2}. We write
\begin{equation}
    \label{eq:173}
    \mathbf{u}(t;z) = T(t;z) \mathbf{u}(0;z) 
    \quad \text{where} \quad
    \mathbf{u}(t;z) = 
    \begin{pmatrix}
        u(t;z) \\
        (p \partial_t u)(t;z)
    \end{pmatrix},
    \quad t \in [0,\infty), z \in \CC,
\end{equation}
and whereas the function $T : [0,\infty) \times \CC \to \SL(2,\CC)$, called here \emph{transfer matrix}, is the unique solution to
\[
    T(t; z) = \Id + \int_0^t b(t'; z) T(t'; z) \ud t' 
    \quad \text{where} \quad
	b(t; z) = 
	\begin{pmatrix}
		0 & \frac{1}{p(t)} \\
		q(t) - z w(t) & 0
	\end{pmatrix}.
\]
For $\omega$-periodic Sturm--Liouville parameters $(\frakp,\frakq,\frakw)$ we shall denote by $\frakT$ and $\frakb$ the mappings $T$ and $b$, respectively.

For any $\eta \in \sS^1$, we denote by $\mathbf{u}$ the solution to
\begin{equation}
    \label{eq:179'}
    \mathbf{u}(t;z) = \eta + \int_0^t b(t';z) \mathbf{u}(t';z) \ud t', \quad t \geq 0.
\end{equation}
Then we set
\[
    \mathbf{u}_n(s, \eta; z) = \mathbf{u}(s + n \omega; z), \quad s \in [0,\omega], n \in \NN_0.
\]
We notice that by \eqref{eq:173} we have $\mathbf{u}_0(s,\eta;z) = T(s;z) \eta$ and
\begin{equation}
    \label{eq:70'}
	\mathbf{u}_{n+1}(s, \eta; z) = X_n(s; z) \mathbf{u}_n(s, \eta; z),
\end{equation}
where $X_n(t; z) = U_{s; n}(\omega; z)$, whereas
\[
    U_{s; n}(t; z) = T(s+t+n\omega; z) T(s + n\omega; z)^{-1}.
\]

First, we investigate the result corresponding to Proposition \ref{prop:4}.
\begin{proposition}
	\label{prop:5}
	Let $\omega > 0$. Suppose that $(p, q, w)$ are $\omega$-asymptotically periodic Sturm--Liouville parameters,
	and let $\frakT$ be the transfer matrix corresponding to $(\frakp, \frakq, \frakw)$. Then
	\begin{equation}
		\label{eq:49}
		\lim_{n \to \infty} U_{s; n}(t; z) = \frakT_s(t; z)
	\end{equation}
	locally uniformly with respect to $s, t \in [0, \omega]$ and $z \in \CC$. Moreover, for every compact set $K \subset \CC$, 
	there is $C > 0$ such that
	\begin{equation}
		\label{eq:55}
		\sup_{s,t \in [0, \omega]} \sup_{z \in K}
		\big\|
		U_{s; n+1}(s; z) - U_{s; n}(t; z)
		\big\|
		\leq
		C \int_0^{2 \omega} \big\| b_{n+1}(t') - b_n(t') \big\| {\: \rm d} t'.
	\end{equation}
\end{proposition}
\begin{proof}
	Observe that for $s \in [0, \omega]$, $\frakT_s(\cdot; z)$ satisfies
	\[
		\frakT_s(t; z) = \Id + \int_0^t \frakb_s(t'; z) \frakT_s(t'; z) \ud t', \quad t' \geq 0,
	\]
	by the variation of parameters we have
	\begin{equation}
		\label{eq:72}
		U_{s; n}(t; z) = \frakT_s(t; z) 
		+ \frakT_s(t; z) \int_0^t \frakT_s(t'; z)^{-1} \big(b_{s; n}(t'; z) - \frakb_s(t'; z) \big) U_{s; n}(t'; z)
		{\: \rm d} t'.
	\end{equation}
	Let $K$ be a compact subset of $\CC$. Since $\|\frakT_s(t; z)\|$ is uniformly bounded from above and from below,
	we have
	\[
		\big\|U_{s; n}(t; z) \big\| \leq C \bigg(1 + 
		\int_0^t \big\|b_{s; n}(t'; z) - \frakb_s(t'; z) \big\| \cdot \|U_{s; n}(t'; z)\| {\: \rm d} t'
		\bigg).
	\]
	Thus by the Gronwall's inequality
	\[
		\|U_{s; n}(t; z)-\frakT_s(t; z) \|
		\leq
		C \int_0^t \big\|b_{s; n}(t'; z) - \frakb_s(t'; z) \big\| {\: \rm d} t'.
	\]
	Since
	\[
		\int_0^t \big\|b_{s; n}(t'; z) - \frakb_s(t'; z) \big\| {\: \rm d} t'
		\leq
		C
		\int_0^{2\omega}
		\big|q_n(t') - \frakq(t') \big| + \big|w_n(t') - \frakw(t') \big| + \bigg|\frac{1}{p_n(t')} - \frac{1}{\frakp(t')} \bigg| 
		{\: \rm d} t',
	\]
	by \eqref{eq:45} we get
	\[
		\lim_{n \to \infty} U_{s; n}(t; z) = \frakT_s(t; z)
	\]
	locally uniformly with respect to $s, t \in [0, \omega]$ and $z \in \CC$. 
\end{proof}

For $s \in [0, \omega]$ and $n \in \NN$, we set
\[
	\gamma_n(s) = \int_0^\omega w_n(s+t) \ud t, \quad\text{and}\quad
	\gamma = \int_0^\omega \frakw(t) \ud t.
\]
\begin{proposition}
    Let $\omega > 0$. Suppose that $(p, q, w)$ are $\omega$-asymptotically periodic Sturm--Liouville parameters. Then for any $s \in [0,\omega]$,
    \begin{equation}
            \label{eq:145}
		\lim_{n \to \infty} \int_0^\omega \bigg| \frac{1}{\gamma_n(s)} w_n(s+t) - \frac{1}{\gamma} \frakw(s+t)
		\bigg| \ud t
		=0.
\end{equation}
\end{proposition}
\begin{proof}
    We have
    \[
        \int_0^\omega 
        \bigg| \frac{1}{\gamma_n(s)} w_n(s+t) - \frac{1}{\gamma} \frakw(s+t)
		\bigg| \ud t
        \leq 
        \bigg| \frac{1}{\gamma_n(s)} - \frac{1}{\gamma} \bigg|
        \int_s^{s+\omega} w_n(t) \ud t 
        + 
        \frac{1}{\gamma} 
        \int_s^{s+\omega} |w_n(t) - \frakw(t)| \ud t.
    \]
    Since by \eqref{eq:45}
    \begin{equation}
        \label{eq:155}
        \lim_{n \to \infty}
        |\gamma_n(s) - \gamma| \leq 
        \lim_{n \to \infty}
        \int_s^{s+\omega} |w_n(t) - \frakw(t)| \ud t
        = 0,
    \end{equation}
    the formula~\eqref{eq:145} easily follows from \eqref{eq:45}.
\end{proof}

\begin{proposition} 
    \label{prop:8}
	Let $\omega > 0$. Suppose that $(p, q, w)$ are $\omega$-asymptotically periodic Sturm--Liouville parameters,
	and let $\frakT$ be the transfer matrix corresponding to $(\frakp, \frakq, \frakw)$. Then for any $s \in [0,\omega]$,
	\begin{equation}
            \label{eq:146}
		\lim_{n \to \infty} \frac{1}{\gamma_n(s)} \partial_z X_n(s; z) 
		=
		\frac{1}{\gamma} \partial_z \frakT_s(\omega; z)
	\end{equation}
	locally uniformly with respect to $z \in \CC$.
\end{proposition}
\begin{proof}
    Recall that $(X_n(s;\cdot) : n \geq 1)$ is a sequence of holomorphic mappings, which by Proposition~\ref{prop:5} converges locally uniformly on $\CC$ to $\frakT_s(\omega;\cdot)$. Therefore, 
    \[
        \lim_{n \to \infty} \partial_z X_n(s; z) =
        \partial_z \frakT_s(\omega;z)
    \]  
    locally uniformly with respect to $z \in \CC$. In view of \eqref{eq:155} the result follows.
\end{proof}

\begin{proposition}
	\label{prop:6}
	Let $\omega > 0$. Suppose that $(p, q, w)$ are $\omega$-asymptotically periodic Sturm--Liouville parameters
	such that the transfer matrix $\frakT$ corresponding to $(\frakp, \frakq, \frakw)$.
	Let $K$ be a compact subset of $\CC$. If
	\[
		q, w, \frac{1}{p} \in \calD^\omega_1\big(L^1; \RR\big)
	\]
	then the sequence $(X_n : n \in \NN)$ belongs to $\calD_1\big([0, \omega]\times K; \GL(2, \CC))$.
\end{proposition}
\begin{proof}
	We have
	\[
		\big\|\Delta b_n (t; z)\big\|
		\leq
		|q_{n+1}(t) - q_n(t)| + \abs{z} |w_{n+1}(t) - w_n(t)| + \bigg|\frac{1}{p_{n+1}(t)} - \frac{1}{p_n(t)} \bigg|,
	\]
	thus by Proposition \ref{prop:5},
	\[
		\big\|\Delta X_{n+1}(t; z) \big\|
		\leq
		c \int_{n\omega}^{(n+2)\omega} 
		|\Delta_\omega q(s)| + \abs{z} |\Delta_\omega w(s) | 
		+ \bigg|\Delta_\omega \bigg(\frac{1}{p}\bigg)(s) \bigg| {\: \rm d} s.
	\]
	Hence,
	\[
		\sum_{n = 0}^\infty \sup_{t \in [0, \omega]}  \sup_{z \in K}
		\big\|\Delta X_{n+1}(t; z) \big\|
		\leq
		2 c \int_0^\infty |\Delta_\omega q(s)| + \abs{z} |\Delta_\omega w(s) |
		+ \bigg|\Delta_\omega \bigg(\frac{1}{p}\bigg)(s) \bigg| {\: \rm d} s,
	\]
	which finishes the proof.
\end{proof}

Next let us describe diagonalization procedure. Fix $K \subset \RR$ a compact subset of
\[
	\Lambda_- = \big\{x \in \RR : \discr \frakT(\omega; x) < 0 \big\}.
\]
Since both $\frakT(s + \omega; z)$ and $\frakT(s; z) \frakT(\omega; z)$ satisfy
\[
	\mathfrak{U}(t; z) = \frakT(\omega; z) + \int_0^t \frakb(t'; z) \mathfrak{U}(t'; z) \ud t,
\]
we must have $\frakT(s; z) \frakT(\omega; z) = \frakT(s + \omega; z)$. Hence,
\[
	\frakT_s(\omega; z) = \frakT(s; z) \frakT(\omega; z) \frakT(s; z)^{-1}.
\]
In particular, $\discr \frakT_s(\omega; z) = \discr \frakT(\omega; z)$. Thus for $z \in \Lambda_-$ and $t \in [0, \omega]$, 
we have $\discr \frakT_t(\omega; z) < 0$. Moreover, $[\frakT_t(\omega; z)]_{12} \neq 0$. Hence, by Proposition \ref{prop:6}, 
there are $\delta > 0$ and $M \geq 1$ such that for all $t \in [0, \omega]$, $n \geq M$ and $z \in K$,
\begin{equation}
	\label{eq:57}
	\discr X_n(t; z) < -\delta,
	\quad\text{and}\quad
	\abs{[X_n(t; z)]_{12}} \geq \delta.
\end{equation}
Consequently, each matrix $X_n(t; z)$ has two distinct eigenvalues
\[
	\lambda_n^+(t; z) = \frac{\tr X_n(t; z) + i \sqrt{-\discr X_n(t; z)}}{2}, 
	\quad\text{and}\quad
	\lambda_n^-(t; z) = \frac{\tr X_n(t; z) - i \sqrt{-\discr X_n(t; z)}}{2}.
\]
In view of \eqref{eq:57}, $X_n(t; z)$ can be diagonalized
\begin{equation}
    \label{eq:153}
    X_n(t; z) = C_n(t; z) D_n(t; z) C_n(t; z)^{-1}
\end{equation}
where
\begin{equation}
    \label{eq:149}
	C_n = 
	\begin{pmatrix}
		1 & 1 \\
		\frac{\lambda_n^+ - [X_n]_{11}}{[X_n]_{12}} & \frac{\lambda_n^- - [X_n]_{11}}{[X_n]_{12}}
	\end{pmatrix}
	\quad\text{and}\quad
	D_n = 
	\begin{pmatrix}
	\lambda_n^+ & 0 \\
	0 & \lambda_n^-
	\end{pmatrix}.
\end{equation}
By \eqref{eq:24}, Proposition \ref{prop:5} and \cite[Lemma 2]{SwiderskiTrojan2019} we obtain the following.
\begin{corollary}
	Let $\omega > 0$. Suppose that $(p, q, w)$ are $\omega$-asymptotically periodic Sturm--Liouville parameters,
	and let $\frakT$ be the transfer matrix corresponding to $(\frakp, \frakq, \frakw)$. Let
	$K$ be a compact subset of $\Lambda_-$. If
	\[
		q, w, \frac{1}{p} \in \calD^\omega_1\big(L^1; \RR\big),
	\]
	then there is $M \geq 1$ such that both sequences $(C_n : n \geq M)$ and $(D_n : n \geq M)$ belongs to
	$\calD_1\big([0, \omega] \times K; \GL(2, \CC) \big)$.
\end{corollary}

Now, Proposition \ref{prop:6} allows us to repeat the arguments in the proof of Theorem \ref{thm:1} to get the next
theorem. For asymptotically periodic Sturm--Liouville parameters it is more convenient to define the Tur\'an determinant
as
\[
    S(t, \eta; z) 
	=
	\det
	\begin{pmatrix}
		u(t+\omega,\eta;z) & u(t,\eta;z) \\
		(p \partial_t u)(t+\omega,\eta;z) & (p \partial_t u)(t,\eta; z)
	\end{pmatrix}.
\]
Such a definition incorporates $p$ and assures that it is a continuous function.
\begin{theorem}
	\label{thm:5}
	Let $\omega > 0$. Suppose that $(p, q, w)$ are $\omega$-asymptotically periodic Sturm--Liouville parameters,
	and let $\frakT$ be the transfer matrix corresponding to $(\frakp, \frakq, \frakw)$. Assume that
	\[
		q, w, \frac{1}{p} \in \calD_1^\omega(L^1; \RR). 
	\]
	Then each solution $\mathbf{u}$ of \eqref{eq:179'} 
	the limit
	\[
		\lim_{n \to \infty} |S_n(t, \eta; z)|
	\]
	exists locally uniformly with respect to $(t, \eta, z) \in [0, \omega] \times \sS^1 \times \Lambda_-$, and it is 
	a positive continuous function.
\end{theorem}
\begin{proof}
	Let us observe that
	\[
		S_n(t, \eta, z) 
		= \sprod{E \mathbf{u}_{n+1}(t, \eta; z)}{\mathbf{u}_n(t, \eta; z)}.
	\]
	Since $\det X_n = 1$, for each $k \in \NN$ we have
	\[
		S_n 
		= 
		\bigg\langle E  \bigg(\prod_{j = 0}^{k-1} X_{n+j}\bigg) \bigg(\prod_{j=1}^k X_{n+j}\bigg)^{-1}
		\mathbf{u}_{n+k}, \mathbf{u}_{n+k-1} \bigg\rangle.
	\]
	Therefore,
	\[
		S_n - S_{n+k}
		=
		\bigg\langle E Y_{n,k} \bigg(\prod_{j=1}^k X_{n+j} \bigg)^{-1} 
		\mathbf{u}_{n+k}, \mathbf{u}_{n+k-1} \bigg\rangle
	\]
	where
	\[
		Y_{n, k}= \bigg(\prod_{j=0}^{k-1} X_{n+j} \bigg) - \bigg(\prod_{j=1}^k X_{n+j} \bigg),
	\]
	which by the Cauchy--Schwarz inequality leads to
	\[
		\big|S_n - S_{n+k} \big|
		\leq
		c
		\bigg\| \bigg(\prod_{j=1}^k X_{n+j}\bigg)^{-1} \bigg\| \cdot \big\|Y_{n, k} \big\| 
		\cdot \big\|\mathbf{u}_{n+k}\big\|^2.
	\]
	For a fixed compact subset $K \subset \Lambda_-$, by Proposition \ref{prop:5}, we have
	\[
		\lim_{n \to \infty} \sup_{z \in K} \sup_{t \in [0, \omega]} 
		\big\| X_n(t; z) - \frakT_t(\omega; z) \big\| = 0.
	\]
	The rest of reasoning follows the same line as in the proof of Theorem \ref{thm:1}.
\end{proof}
The next statement is a straightforward consequence of Theorem \ref{thm:5}.
\begin{corollary}
	Suppose that the hypotheses of Theorem \ref{thm:5} are satisfied. Then for any compact subset $K \subset \Lambda_-$, there 
	is $c > 0$ such that for each solution $\mathbf{u}$ to \eqref{eq:179'}, every $n \in \NN$, $t \in [0, \omega]$, 
	$\eta \in \sS^1$, and $z \in K$, we have
	\begin{equation}
            \label{eq:158}
		c^{-1} \leq \big\|\mathbf{u}_n(t, \eta; z) \big\|^2 \leq c. 
	\end{equation}
\end{corollary}

Taking advantage of the bound \eqref{eq:158} we can prove absolute continuity of $H_\eta$ on $\Lambda_-$.
\begin{corollary}
    \label{cor:9}
    Suppose that the hypotheses of Theorem \ref{thm:5} are satisfied. Then $\tau$ is in the limit point case. Moreover, for any $\eta \in \sS^1$ measure $\mu_\eta$ is absolutely continuous on $\Lambda_-$ and for any compact $K \subset \Lambda_-$ there are constants $c_1,c_2>0$ such that the density of $\mu_\eta$ satisfies
    \[
        c_1 < \mu_\eta'(\lambda) < c_2
    \]
    for almost all $\lambda \in K$.
\end{corollary}
\begin{proof}
    By analogous reasoning to Corollary~\ref{cor:8} we can show that for any compact $K \subset \Lambda_-$ there are constants $c_3>0$ and $M \geq 1$ such that for any $z \in K$, $t \in [0,\omega]$, $\eta \in \sS^1$ and $n \geq M$,
    \begin{equation}
        \label{eq:156'}
        |u_n(t,\eta;z)|^2
        \leq
        \| \mathbf{u}_n(t,\eta;z) \|^2 
        \leq 
        c_3 \big( |u_n(t,\eta;z)|^2 + |u_{n+1}(t,\eta;z)|^2\big).
    \end{equation}
    The rest of the proof follows along the same lines as the proof of Corollary~\ref{cor:7}.
\end{proof}

Now the proof of Theorem \ref{thm:2} leads to the following results.
\begin{theorem}
	Let $\omega > 0$. Suppose that $(p, q, w)$ are $\omega$-asymptotically periodic Sturm--Liouville parameters,
	and let $\frakT$ be the transfer matrix $\frakT$ corresponding to $(\frakp, \frakq, \frakw)$. Assume that
	\[
		q, w, \frac{1}{p} \in \calD_1^\omega(L^1; \RR).
	\]
	Then for each compact $K \subset \Lambda_-$, there are $M \geq 1$ and a non-vanishing function 
	$\vphi: [0, \omega] \times \sS^1 \times K \rightarrow \CC$ such that for every solution $\mathbf{u}$ of \eqref{eq:179'},
	\begin{equation}
		\label{eq:59}
		\lim_{n \to \infty} \sup_{(t, \eta, z) \in [0, \omega] \times \sS^1 \times K}
		\bigg|\frac{u_{n+1}(t, \eta; z) - \lambda_n^-(t; z) u_n(t, \eta; z)}
		{\prod_{k = M}^{n-1} \lambda_k^+(t, z)}
		-\vphi(t, \eta; z)\bigg| = 0
	\end{equation}
	where $u_n = \sprod{\mathbf{u}_n}{e_1}$. Furthermore,
	\begin{equation}
		\label{eq:58}
		\frac{u_n(t, \eta; z)}{\prod_{k = M}^{n-1} \lambda_k^+(t; z)}
		=
		\frac{\vphi(t, \eta; z)}{\sqrt{4 - \abs{\tr \frakT(\omega; z)}^2}}
		\sin\Big(\sum_{k = M}^{n-1} \theta_k(t; z) + \arg \vphi(t, \eta; z) \Big)
		+
		E_n(t, \eta; z)
	\end{equation}
	where
	\[
		\theta_k(t; z) = \arccos\bigg(\frac{\tr X_k(t; z)}{2 \sqrt{\det X_k(t; z)}}\bigg),
	\]
	and
	\[
		\lim_{n \to \infty} \sup_{(t, \eta, z) \in [0, \omega] \times \sS^1 \times K} {\abs{E_n(t, \eta; z)}} = 0.
	\]
\end{theorem}
\begin{proof}
	Let us observe that the only part that requires modification is the proof of Claim \ref{clm:2}. Notice that 
	\[
		\abs{\lambda^+_k(t; z)}^2 = \lambda^+_k(t; z) \overline{\lambda^+_k(t; z)} = \det D_k(t; z) = \det X_k(t; z) = 1,
	\]
	thus
	\[
		\prod_{k = m}^{n-1} \abs{\lambda_k^+(t; z)}^2 = 1.
	\]
	On the other hand, by Theorem \ref{thm:5}, we have
	\[
		c_1 \leq p_{L_j}(t) D_{L_j}(t, \eta; z) \leq c_2 \|\mathbf{u}_{L_j}(t, \eta; z)\|^2.
	\]
	Therefore,
	\[
		\frac{\|\mathbf{u}_{L_j}(t, \eta; z)\|}{\prod_{k = m}^{L_j-1} \abs{\lambda_k^+(t; z)}^2}
		\geq
		\sqrt{\frac{c_1}{c_2}},
	\]
	which proves the claim. The rest of reasoning is the same as in the proof of Theorem \ref{thm:2}
\end{proof}

For $s \in [0, \omega]$ and $n \in \NN$, we set
\[
	\rho_n(s) = \sum_{j = 0}^n \gamma_j(s) = \int_s^\omega \varrho_n(s + t) \ud t, 
    \quad \text{where} \quad
    \varrho_n(s) = \sum_{j = 0}^n w_j(s).
\]
Lastly, we set
\begin{equation}
    \label{eq:175}
	\rho_L = \int_0^L w(t) \ud t.
\end{equation}
Now, the reasoning as in the proof of Theorem \ref{thm:3}, leads to the following statement.
\begin{theorem}
	\label{thm:6}
	Let $\omega > 0$. Suppose that $(p, q, w)$ are $\omega$-asymptotically periodic Sturm--Liouville parameters,
	and let $\frakT$ be the transfer matrix $\frakT$ corresponding to $(\frakp, \frakq, \frakw)$. Assume that
	\begin{equation}
            \label{eq:143}
		q, w, \frac{1}{p} \in \calD_1^\omega(L^1; \RR).
	\end{equation}
	If for almost all $t \in [0, \omega]$,
	\[
		\lim_{n \to \infty} \varrho_n(t) = +\infty,
		\qquad\text{and}\qquad
		\lim_{n \to \infty} \frac{w_n(t)}{w_{n+1}(t)} = 1,
	\]
	then for every solution $\mathbf{u}$ to \eqref{eq:179'} there is $M \geq 1$, such that 
	\[
		\lim_{n \to \infty} 
		\frac{1}{\rho_n(t)} \sum_{m = M}^n 
		\abs{\sprod{\mathbf{u}_m(t, \eta; z)}{e_1}}^2 w_m(t)
		=
		\frac{\abs{\vphi(t, \eta; z)}^2}{4 - \abs{\tr \frakT(\omega; z)}^2} p_M(t),
	\]
	uniformly with respect to $(\eta, z) \in \sS^1 \times \Lambda_-$.
\end{theorem}

Taking into account Proposition~\ref{prop:8} we get the following result.
\begin{corollary} 
    \label{thm:12}
	Suppose that the hypotheses of Theorem \ref{thm:6} are satisfied. Then
	\[
		\lim_{L \to \infty}
		\frac{1}{\rho_L} K_L(z, z; \eta)
		=
		\int_0^\omega 
		\frac{ \abs{\vphi(t, \eta; z)}^2 }{4 - |\tr \frakT(\omega; z)|^2} \ud t,
	\]
	locally uniformly with respect to $(\eta, z)  \in \sS^1 \times \Lambda_-$.
\end{corollary}

Let us consider an analogue of Theorem~\ref{thm:10}.
\begin{theorem} \label{thm:13}
Let $\omega > 0$. Suppose that $(p, q, w)$ are $\omega$-asymptotically periodic Sturm--Liouville parameters,
	and let $\frakT$ be the transfer matrix corresponding to $(\frakp, \frakq, \frakw)$. Assume that
	\[
		q, w, \frac{1}{p} \in \calD_1^\omega(L^1; \RR). 
	\]
    Then for each $\eta \in \sS^1$ we have 
    \begin{equation}
        \label{eq:162}
        \sigmaS(H_\eta) \cap \Lambda_- = \emptyset \quad \text{and} \quad
        \sigmaAC(H_\eta) = \sigmaEss(H_\eta) = \cl{\Lambda_-}.
    \end{equation}
\end{theorem}
\begin{proof}
In view of Corollary~\ref{cor:9} we only need to prove that $\sigmaEss(H_\eta) \cap \Lambda_+ = \emptyset$, where
\[
        \Lambda_+ = \{ z \in \RR : \discr \frakT(\omega;z) > 0 \}.
\]
To do so, we proceed similarly as in the proof of Theorem~\ref{thm:10}. Namely, let $z_0 \in \Lambda_+$ and let $K \subset \Lambda_+$ be a compact interval containing $z_0$ in its interior. Again, we consider the equation~\eqref{eq:94}. Now we have that the matrix $\frakT(\omega;z)$ has two eigenvalues
\[
    \lambda_-(z) = \xi_- \Big( \frac{\tr \frakT(\omega;z)}{2} \Big), \quad \text{and} \quad
    \lambda_+(z) = \xi_+ \Big( \frac{\tr \frakT(\omega;z)}{2} \Big).
\]
Since $|\tr \frakT(\omega;z)|>2$, we have
\[
    0<|\lambda_-(z)| < 1 < |\lambda_+(z)|, \quad z \in K.
\]
By \eqref{eq:49} we again obtain~\eqref{eq:110} and by the continuity of Joukowsky map we get
\[
    \lim_{n \to \infty} \lambda^-_n(0; z) = 
    \lambda_-(z),
    \quad
    \lim_{n \to \infty} \lambda^+_n(0; z) =
    \lambda_+(z),
\]
uniformly with respect to $z \in K$. In view of Proposition~\ref{prop:6} we have $(X_n(0;\cdot) : n \geq 0) \in \calD_1(K; \GL(2,\RR))$. Analogously as in the proof of Theorem~\ref{thm:10} there exists a compact interval $K' \subset K$ containing $z_0$ in its interior such that for any $\epsilon$ satisfying
\[
    0<\epsilon< 1 - \sup_{z \in K'} |\lambda_-(z)| < 1
\] 
there are constants $c>0$ and $M' \geq M$ such that 
\[
    \| \mathbf{u}^-(t+n\omega;z) \| \leq c (1-\epsilon)^n, \quad t \in [0,\omega], n \geq M, z \in K'.
\]
Therefore,
\[
    \int_{M'\omega}^\infty \sup_{z \in K'} |u^-(x;z)|^2 w(x) \ud x \leq 
    c^2 \sum_{n=M'}^\infty (1-\epsilon)^{2n} \int_0^\omega w_n(t) \ud t,
\]
which thanks to \eqref{eq:45} is finite. Therefore, our result follows by the reasoning below \eqref{eq:154}.
\end{proof}

\begin{theorem} 
    \label{thm:14}
    Let $\omega > 0$. Suppose that $(p, q, w)$ are $\omega$-asymptotically periodic Sturm--Liouville parameters,
	and let $\frakT$ be the transfer matrix corresponding to $(\frakp, \frakq, \frakw)$. Assume that
	\[
		q, w, \frac{1}{p} \in \calD_1^\omega(L^1; \RR). 
	\]
    Then for every $f \in \calC(\RR)$ such that
    \[
        \lim_{|\lambda| \to \infty} (1+\lambda^2) |f(\lambda)| = 0.
    \]
    we have
    \[
        \lim_{n \to \infty} \frac{1}{\rho_{n\omega}} \int_\RR f \ud \nu_\eta^{n\omega} = \int_\RR f \ud \nu_\infty,
    \]
    where the measure $\nu_\infty$ is absolutely continuous with the density
    \[
        \frac{\ud \nu_\infty}{\ud \lambda} = \frac{1}{\pi} \frac{|\partial_z \tr \frakT(\omega;\lambda)|}{\gamma \sqrt{4 - |\tr \frakT(\omega;\lambda)|^2}} \mathds{1}_{\Lambda_-}(\lambda).
    \]
\end{theorem}
\begin{proof}
    Let us start by proving an analogue of Theorem~\ref{thm:4} for $s=0$. 
    By Proposition~\ref{prop:5} we have
    \begin{equation}
        \label{eq:168}
        \lim_{n \to \infty} X_n(0;z) = \frakT(\omega;z)
    \end{equation}
    locally uniformly with respect to $z \in \CC$. Recall that by \cite[formula (7.5.47)]{GesztesyBook2024}
    \begin{equation} 
        \label{eq:166}
        \big[ \tr \frakT(\omega;\cdot) \big]^{-1}([-2,2]) \subset \RR
    \end{equation}
    and by \cite[formula (7.5.68)]{GesztesyBook2024} 
    \begin{equation}
        \label{eq:167}
        [\frakT(\omega;z)]_{1,2} \neq 0, \quad z \in \CC \setminus \RR.
    \end{equation}
    Therefore, for each compact set $K \subset \CC_+$ there are $\delta>0$  and $L_0 \geq 1$ such that for all $n \geq L_0$, $t \in [0,\omega]$ and $z \in K$ the formula~\eqref{eq:48} holds true. Therefore, we can define the eigenvalues $\lambda^+_n(0;z)$ and $\lambda^-_n(0;z)$ of $X_n(0;z)$ as in formula~\eqref{eq:106}. Note that by~\eqref{eq:168} we have
    \begin{equation}
        \label{eq:168}
        \lambda_\infty^+(z) := \lim_{n \to \infty} \lambda_n^+(0;z) = 
        \xi_+ \Big( \frac{\tr \frakT(\omega;z)}{2} \Big)
        \quad \text{and} \quad
        \lambda_\infty^-(z) :=
        \lim_{n \to \infty} \lambda_n^-(0;z) = 
        \xi_- \Big( \frac{\tr \frakT(\omega;z)}{2} \Big)
    \end{equation}
    uniformly with respect to $z \in K$. In view of \eqref{eq:166} it implies that there is $M \geq L_0$ and $\epsilon \in (0,1) $ such that for any $n \geq M$
    \begin{equation}
        \label{eq:169}
        \inf_{z \in K} |\lambda^+_n(0;z)| \geq (1+\epsilon)
        \quad \text{and} \quad 
        \sup_{z \in K} |\lambda^-_n(0;z)| \leq
        (1-\epsilon).
    \end{equation}
    In particular, \eqref{eq:77} is satisfied for $s=0$. Next, we can diagonalize $X_n(0;z)$ in the form~\eqref{eq:31a} and \eqref{eq:31b}. By Proposition~\ref{prop:6} we easily get that both sequences $(C_n : n \geq M)$ and $(D_n : n \geq M)$ belong to $\calD_1 \big( [0,\omega] \times K; \GL(2,\CC) \big)$. Then by \cite[Theorem 6.1]{SwiderskiTrojan2023} we get sequences of maps $(\mathbf{u}^-_n(0;\cdot) : n \geq M)$ and $(\mathbf{u}^+_n(0;\cdot) : n \geq M)$, solving \eqref{eq:70'}, continuous on $K$ and holomorphic on $\intr{K}$ and satisfying \eqref{eq:117} uniformly with respect to $z \in K$, where
    \begin{equation}
        \label{eq:172}
        \calC_\infty(z) = 
        \begin{pmatrix}
            1 & 1 \\
            \frac{\lambda^+_\infty(z) - [\frakT(\omega;z)]_{1,1}}{[\frakT(\omega;z)]_{1,2}} &
            \frac{\lambda^-_\infty(z) - [\frakT(\omega;z)]_{1,1}}{[\frakT(\omega;z)]_{1,2}}
        \end{pmatrix}.
    \end{equation}
    By following the proof of Theorem~\ref{thm:4} we construct the mappings $\mathbf{u}^+(\cdot;z)$ and $\mathbf{u}^-(\cdot;z)$, which are linearly independent solutions of \eqref{eq:179'}. Moreover, they satisfy \eqref{eq:117a} and \eqref{eq:117b}. Let us define $\phi^+$ and $\phi^-$ by the formula \eqref{eq:171}. Then by \eqref{eq:172} we immediately get $\phi^+(0;z) \neq 0$ and $\phi^-(0;z) \neq 0$ for any $z \in K$. Observe that by \eqref{eq:117b} there is a constant $c>0$ such that
    \[
        \| \mathbf{u}_n(t;z) \|^2 \leq c \prod_{M \leq j < n} |\lambda_j^-(0;z)|^2, \quad t \in [0,\omega], n \geq M.
    \]
    Therefore, by \eqref{eq:169} there is a constant $c'>0$ such that
    \[
        \int_{M\omega}^\infty |u^-(x;z)|^2 w(x) \ud x \leq
        c' \sum_{n=M}^\infty (1-\epsilon)^n \int_0^\omega w_n(t) \ud t,
    \]
    which by \eqref{eq:45} is finite. Thus $u^-(\cdot;z) \in L^2(w)$.

    Now, let us follow the proof of Theorem~\ref{thm:9}. Then if $\mathbf{v}(\cdot;z)$ is any solution to \eqref{eq:173}, it can be written in the form~\eqref{eq:76}. Then by a similar argument we get \eqref{eq:80}. Now, follow the proof from \eqref{eq:174}. In view of \eqref{eq:175} and \eqref{eq:165} we get \eqref{eq:54}. Now, by \eqref{eq:54}, \eqref{eq:176} and \eqref{eq:168} together with \eqref{eq:155} we obtain
    \begin{equation}
        \label{eq:177}
        \lim_{n \to \infty} \frac{1}{\rho_n(0)} \frac{\partial_z u_n(0,\eta;z)}{u_n(0,\eta;z)} = \frac{1}{\gamma} \frac{\partial_z \lambda_\infty^+(z)}{\lambda_\infty^+(z)}
    \end{equation}
    uniformly on $K$. Therefore, by \cite[Lemma 4.1]{SwiderskiTrojan2023} there exists a Borel measure measure $\nu_\infty$ on the real line such that for every $f \in \calC(\RR)$ satisfying $\lim_{|\lambda| \to \infty} (1+\lambda^2) |f(x)| = 0$ we have
    \[
        \lim_{n \to \infty} \frac{1}{\rho_{n\omega}} \int_\RR f \ud \nu^{n\omega} =
        \int_\RR f \ud \nu_\infty.
    \]
    Moreover, the measure $\nu_\infty$ is uniquely defined by the right-hand side of \eqref{eq:177}. Therefore, we need only to indentify the measure $\nu_\infty$. To do so, notice that \eqref{eq:177} holds in particular for $(\frakp, \frakq, \frakw)$. In such a case it is known (see, e.g. \cite[p. 44]{Brown2013}) that
    \[
        \lim_{n \to \infty} \frac{1}{n \omega} \int_\RR f \ud \nu^{n\omega} =
        \int_{\Lambda_-} f(\lambda) \frac{1}{\pi \omega} \frac{|\partial_z \tr \frakT(\omega;\lambda)|}{\sqrt{4 - |\tr \frakT(\omega;\lambda)|^2}} \ud \lambda, \quad f \in \calC_c(\RR).
    \]
    Since by Stolz--Ces\`{a}ro theorem and \eqref{eq:155}
    \[
        \lim_{n \to \infty} \frac{n \omega}{\rho_{n\omega}}  = 
        \lim_{n \to \infty} \frac{\omega}{\gamma_n(0)} = \frac{\omega}{\gamma}
    \]
    our result easily follows.
\end{proof}

Taking into account Theorem~\ref{thm:14} and Corollary~\ref{thm:12} the proof of the following result is analogous to the proof of Corollary~\ref{cor:6}.
\begin{corollary}
    Suppose that the hypotheses of Theorem~\ref{thm:6} are satisfied. Then the measure $\mu_\eta$ is absolutely continuous on $\Lambda_-$ with the density
    \[
        \mu_\eta'(\lambda) = \frac{1}{\pi} \frac{|\partial_z \tr \frakT(\omega;\lambda)|}{\gamma \sqrt{4 - |\tr \frakT(\omega;\lambda)}|^2} \frac{1}{g(\lambda)}, \quad \lambda \in \Lambda_-,
    \]
    where
    \[
        g(\lambda) = \lim_{L \to \infty} \frac{1}{\rho_L} K_L(\lambda,\lambda;\eta), \quad \lambda \in \Lambda_-.
    \]
    In particular, $\mu_\eta'$ is a continuous and positive function on $\Lambda_-$.
\end{corollary}

Let us recall that recently in \cite{Behrndt2023} Sturm--Liouville parameters such that
\begin{equation}
    \label{eq:164}
    \int_0^\infty
	\Big|\frac{1}{p(x)} - \frac{1}{\frakp(x)} \Big| +
	|q(x) - \frakq(x)| +
	|w(x) - \frakw(x)| \ud x < \infty,
\end{equation}
were considered. One of their main result, \cite[Theorem 1.1]{Behrndt2023}, showed \eqref{eq:162}. Let us observe that \eqref{eq:164} implies the hypotheses of Theorem~\ref{thm:13}. The next example shows that our conditions are weaker than \eqref{eq:164}.

\begin{example}
    Consider
    \[
        p(x) \equiv 1, \quad 
        q(x) \equiv 1, \quad 
        w(x) = 2 + \frac{\sin(\log\log{x})}{\log{x}} \mathds{1}_{\big( \exp(\exp(1)),\infty \big)}(x).
    \]
    It immediately follows that \eqref{eq:45} is satisfied for
    \[
        \frakp(x) \equiv 1, \quad \frakq(x) \equiv 1, \quad \frakw(x) \equiv 2
    \]
    and any $\omega>0$.
    Notice that $w' \in L^1([0,\infty))$. Therefore, by \cite[p.2]{Stolz1991a} we easily get that \eqref{eq:143} is satisfied. It can be shown that the rest of the hypotheses of Corollary~\ref{thm:12} are satisfied. However, our $(p,q,w)$ do not satisfy \eqref{eq:164}.
\end{example}

\begin{bibliography}{jacobi,schrodinger}
	\bibliographystyle{amsplain}
\end{bibliography}

\end{document}